\documentclass[reqno]{amsart}
\usepackage[T1]{fontenc}
\usepackage{amssymb}
\usepackage[textwidth=15.5cm,hcentering,top=1in,bottom=1in]{geometry}
\usepackage{enumitem}
\usepackage{physics}
\usepackage[foot]{amsaddr}
\usepackage{tikz}
\usepackage{mathtools}
\usepackage{subcaption}
\usepackage[maxbibnames=99]{biblatex}
\theoremstyle{remark}
\newtheorem*{remark}{Remark}
\newenvironment{mycases}
   {\begin{dcases*}}
   {\end{dcases*}}
\theoremstyle{plain}
\newcommand{\supp}{\mathrm{supp}\,}
\newcommand{\fin}{f_{\mathrm{in}}}
\newcommand{\fVPin}{f^{\mathrm{VP}}_{\mathrm{in}}}
\newcommand{\fVP}{f^{\mathrm{VP}}}
\newcommand{\EVP}{E^{\mathrm{VP}}}
\newtheorem{theorem}{Theorem}[section]
\newtheorem{lemma}[theorem]{Lemma}

\newtheorem{proposition}[theorem]{Proposition}
\newtheorem{corollary}[theorem]{Corollary}
\counterwithout{figure}{section}
\counterwithin{equation}{section}
\newcommand{\vp}{\varphi}
\newcommand{\br}[1]{\left(#1\right)}
\newcommand{\ak}{\abs{k}}
\newcommand{\ax}{\abs{x}}
\newcommand{\aal}{\abs{\alpha}}
\newcommand{\aab}{\abs{\beta}}
\newcommand{\sbr}[1]{\left[#1\right]}

\newcommand{\p}{\partial}
\newcommand{\jb}[1]{\left\langle#1\right\rangle}
\newcommand{\set}[1]{\left\{#1\right\}}
\newcommand{\qaq}{\quad\text{and}\quad}
\newcommand{\ep}{\varepsilon}
\newcommand{\R}{\mathbb{R}}
\newcommand{\C}{\mathbb{C}}
\newcommand{\lam}{\lambda}
\newcommand{\ETin}{E_{\mathrm{in},T}}
\newcommand{\ELin}{E_{\mathrm{in},L}}
\newcommand{\Ein}{E_{\mathrm{in}}}
\newcommand{\Bin}{B_{\mathrm{in}}}
\newcommand{\Jin}{J_{\mathrm{in}}}
\newcommand{\Bes}{\mathrm{Bes}}
\newcommand{\ETosc}{E_T^{\mathrm{osc}}}
\newcommand{\Bosc}{B^{\mathrm{osc}}}
\usepackage[hidelinks]{hyperref}
\hypersetup{pdftitle={The uniform-in-time electrostatic limit of the linearised Vlasov--Maxwell system: the Poisson equilibrium},pdfauthor={Marnie Smith}}

\title[The uniform-in-time electrostatic limit of linearised Vlasov--Maxwell]{The uniform-in-time electrostatic limit of the linearised Vlasov--Maxwell system: the Poisson equilibrium}
\author{Marnie Smith}
\address{Department of Pure Mathematics and Mathematical Statistics, University of Cambridge}
\email{ms2724@cam.ac.uk}
\begin{document}

\date{20 September 2026}

\begin{abstract}
The Vlasov--Maxwell system with Newtonian particle transport is linearised about the Poisson equilibrium on~$\R^3$, with the speed of light~$c$ treated as a large parameter. For uniformly controlled initial data, with transverse fields that may remain of order one, the difference between the Vlasov--Maxwell perturbation and its linearised Vlasov--Poisson counterpart is bounded uniformly in time by their initial discrepancy plus a term of order~$c^{-1}$. The same estimate holds for their scattering states, and the rate in~$c$ is sharp. The longitudinal dynamics are governed by the same Volterra equation as in linearised Vlasov--Poisson and undergo Landau damping. The transverse fields exhibit Klein--Gordon-type dispersion, and control of their accumulated effect along free transport yields the uniform comparison.\end{abstract}
\maketitle
\tableofcontents

\section{Introduction}
\subsection{The Vlasov--Maxwell system and its electrostatic limit}
Landau damping is a collisionless relaxation phenomenon. Most rigorous results concern the Vlasov--Poisson system, which retains only the longitudinal electrostatic interaction. There, phase mixing in the particle dynamics shears spatial inhomogeneity into increasingly fine velocity-space structure, producing cancellation under velocity integration and driving the self-consistent electric field to zero despite the absence of dissipation. The Vlasov--Maxwell system additionally contains magnetic and transverse electric fields, which together support Klein--Gordon-type electromagnetic waves that decay dispersively. This paper studies the electrostatic limit $c\to\infty$ for Vlasov--Maxwell with Newtonian particle transport, linearised about a spatially homogeneous equilibrium. The central question is whether linearised Vlasov--Poisson nevertheless provides an approximation to the Vlasov--Maxwell particle dynamics that remains valid uniformly for all $t\geq0$ and extends to their scattering states.

The starting point is the Vlasov--Maxwell system
\begin{align}\label{fullVM}
    \begin{mycases}
\p_tf+v\cdot\nabla_xf+\br{E+\frac{v}{c}\times B}\cdot\nabla_vf=0,\\
\p_tB+c\nabla_x\times E=0,\qquad \nabla_x\cdot E=\rho[f]-n_0,\\
-\p_t E+c\nabla_x\times B=j[f],\qquad \nabla_x\cdot B=0.
    \end{mycases}
\end{align}
Here, for a function $g=g(x,v)$,
\begin{align*}
    \rho[g]:=\int_{\R^3}g(\cdot,v)\,dv,
    \qquad
    j[g]:=\int_{\R^3}v g(\cdot,v)\,dv
\end{align*}
denote the associated charge and current densities, and $n_0>0$ is the density of the neutralising homogeneous background. The second line contains Faraday's and Gauss's laws, while the third contains the Amp\`ere--Maxwell law and the magnetic divergence constraint. The particle transport is Newtonian, whereas the electromagnetic field obeys Maxwell's equations with propagation speed~$c$. Consequently, the model is not Lorentz invariant and the velocity variable is not constrained by $\abs{v}<c$.

Formally, as $c\to\infty$, the magnetic force is suppressed and Faraday's law forces the electric field to be irrotational. The limiting dynamics are those of the Vlasov--Poisson system, whose particle distribution, electric field and potential are denoted $\fVP$, $\EVP$ and $\phi^{\mathrm{VP}}$:
\begin{align}\label{VP}
    \begin{mycases}
        \p_t\fVP+v\cdot\nabla_x\fVP+\EVP\cdot\nabla_v\fVP=0,\\
        \EVP=-\nabla_x\phi^{\mathrm{VP}},\qquad -\Delta_x\phi^{\mathrm{VP}}=\rho[\fVP]-n_0.
    \end{mycases}
\end{align}
For smooth solutions with suitably matched initial data, convergence to Vlasov--Poisson on a common finite interval of existence is classical. But Landau damping and scattering concern the behaviour as $t\to\infty$. The transverse electromagnetic sector need not be small initially, and its influence on the particle dynamics may accumulate over long times. 

Both systems are linearised about the Poisson equilibrium
\begin{align}\label{Poisson}
    \mu(v)=\frac{1}{\pi^2(1+\abs{v}^2)^2},
    \qquad
    \widehat\mu(\eta)=e^{-\abs{\eta}},
\end{align}
for which $n_0=\int_{\R^3}\mu(v)\,dv=1$. Then $\nabla_v\mu$ is parallel to $v$, and $(E,B,f)=(0,0,\mu)$ is a formal equilibrium of the Vlasov--Maxwell system: $\abs{v}\mu\notin L^1_v$, so $j[\mu]$, which vanishes by oddness, is not absolutely convergent. Writing the particle distribution as $\mu+f$ in~\eqref{fullVM}, so that $f$ now denotes the perturbation, and discarding the terms quadratic in the perturbation, while using $(v\times B)\cdot\nabla_v\mu=0$, gives the linearised Vlasov--Maxwell system studied here:
\begin{align}\label{linVM}
    \begin{mycases}
\p_tf+v\cdot\nabla_xf+E\cdot\nabla_v\mu=0,\\
\p_tB+c\nabla_x\times E=0,\qquad \nabla_x\cdot E=\rho[f],\\
-\p_t E+c\nabla_x\times B=j[f],\qquad \nabla_x\cdot B=0.
    \end{mycases}
\end{align}
The initial data for~\eqref{linVM},
\begin{align*}
    f|_{t=0}=\fin,\qquad E|_{t=0}=\Ein,\qquad B|_{t=0}=\Bin,
\end{align*}
are assumed to satisfy the compatibility conditions
\begin{align}
    \nabla_x\cdot \Ein=\rho[\fin],\qquad \nabla_x\cdot \Bin=0,\qquad \int_{\R^3}\int_{\R^3}\fin(x,v)\,dvdx=0.\label{compatibilityconditions}
\end{align}

Likewise, writing the particle distribution as $\mu+\fVP$ in~\eqref{VP} and retaining only the terms that are linear in the perturbation gives the linearised Vlasov--Poisson system, the comparison system studied here,
\begin{align}\label{linVP}
    \begin{mycases}
        \p_t\fVP+v\cdot\nabla_x\fVP+\EVP\cdot\nabla_v\mu=0,\\
        \EVP=-\nabla_x\phi^{\mathrm{VP}},\qquad -\Delta_x\phi^{\mathrm{VP}}=\rho[\fVP],
    \end{mycases}
\end{align}
with initial datum $\fVP|_{t=0}=\fVPin$ satisfying \begin{align}
    \int_{\R^3}\int_{\R^3}\fVPin(x,v)\,dvdx=0.\label{zeromassVP}
\end{align}
For the initial data considered below, both linearised systems admit unique global mild solutions by standard linear theory.

The two linearised systems share their longitudinal structure, which is exhibited by writing~\eqref{linVM} in Coulomb gauge. Introduce scalar and vector potentials $\phi$ and $A$ for the Vlasov--Maxwell fields by writing
\begin{align}\label{Coulombgauge}
    E=-\nabla_x\phi-\frac{1}{c} \p_tA,
    \qquad
    B=\nabla_x\times A,
    \qquad
    \nabla_x\cdot A=0,
\end{align}
so that Faraday's law and $\nabla_x\cdot B=0$ hold automatically. The Helmholtz decomposition $E=E_L+E_T$ is therefore given by
\begin{align*}
    E_L=-\nabla_x\phi,
    \qquad
    E_T=-\frac{1}{c}\p_tA.
\end{align*}
Then $\nabla_x\cdot E_T=0$, and the initial field decomposes as $\Ein=\ELin+\ETin$. In these variables, Maxwell's equations in~\eqref{linVM} reduce to
\begin{align}\label{APhi}
    -\Delta_x\phi=\rho[f],
    \qquad
    \frac{1}{c^2}\p_t^2A-\Delta_xA=\frac{1}{c} \mathbb Pj[f],
\end{align}
where $\mathbb P:=I-\nabla_x\Delta_x^{-1}\nabla_x\cdot$ is the Leray projection onto divergence-free fields. The second equation in~\eqref{APhi} is the forced wave equation for the Coulomb-gauge vector potential $A$. The first equation determines the longitudinal electric field from the density through the same Poisson equation as in~\eqref{linVP}.

Throughout, the Fourier transform is taken with the convention $\mathcal F[g](k)=\widehat g(k)=\int_{\R^3}e^{-ik\cdot x}g(x)\,dx$ in the spatial variable, and likewise in $(x,v)$ with the pair $(k,\eta)$. For $k\neq0$, the Vlasov--Maxwell longitudinal field and the Vlasov--Poisson field satisfy
\begin{align}\label{longitudinal-fields}
\widehat E_L(t,k)=-i\frac{k}{\ak^2}\widehat\rho(t,k),
\qquad
\widehat \EVP(t,k)=-i\frac{k}{\ak^2}\widehat{\rho^{\mathrm{VP}}}(t,k),
\end{align}
where $\rho:=\rho[f]$ and $\rho^{\mathrm{VP}}:=\rho[\fVP]$. 

A key structural feature of the present problem is that the longitudinal sector $(\rho,E_L)$ decouples exactly from the transverse sector $(A,E_T,B)$. This decoupling is special to the linearisation about a homogeneous radial equilibrium: radiality eliminates the magnetic term at linear order, while linearisation removes the Lorentz force acting on $f$ itself. Indeed, solving the linearised Vlasov equation along the free-transport characteristics and integrating in velocity shows that the electric field enters only through $k\cdot\widehat E(t,k)$. The transverse contribution vanishes because $k\cdot\widehat E_T(t,k)=0$, leaving the closed Volterra equation
\begin{align}\label{Volterra}
\widehat\rho(t,k)=\widehat\fin(k,kt)
-\int_0^t(t-s)\widehat\mu\br{k(t-s)}\widehat\rho(s,k)\,ds.
\end{align}
The Vlasov--Poisson density $\rho^{\mathrm{VP}}$ satisfies the same equation with $\fVPin$ in place of $\fin$. By contrast, the transverse dynamics are encoded by the Coulomb-gauge vector potential $A$. Its equation in~\eqref{APhi} has propagation speed $c$, and it determines the transverse fields $E_T$ and $B$. This sector contains the propagating electromagnetic waves absent from~\eqref{linVP}. If $\fin=\fVPin$, the difference between $f$ and $\fVP$ is therefore driven solely by $E_T$.

\subsection{Main results}\label{mainresultssubsection}

The following notations and conventions are used. For a non-negative integer $m$ and $1\leq p\leq\infty$, define the Sobolev norms
\begin{align*}
\norm{g}_{W^{m,p}_x}
:=\sum_{\abs{\alpha}\leq m}\norm{\nabla_x^\alpha g}_{L^p_x},
\qquad
\norm{f}_{W^{m,p}_{x,v}}
:=\sum_{\abs{\alpha}+\aab\leq m}
\norm{\nabla_x^\alpha\nabla_v^\beta f}_{L^p_{x,v}}.
\end{align*}
Japanese brackets are defined by $\jb{a}:=(1+\abs{a}^2)^{1/2}$ and $\jb{a,b}:=(1+\abs{a}^2+\abs{b}^2)^{1/2}$ for scalar or vector $a$ and $b$. Mixed norms are read from right to left: for function spaces $X$ and $Y$, $\norm{f}_{XY}:=\norm{\norm{f}_{Y}}_{X}$. For a symbol $m(k)$, scalar or matrix-valued, $m(i\nabla_x)$ denotes the Fourier multiplier with symbol $m(k)$, that is, $m(i\nabla_x)g$ is the inverse Fourier transform of $m(k)\widehat g(k)$. 

Fix a radial $\chi\in C_c^\infty(\R^3)$ with $\chi=1$ on $\set{\ak\leq1/2}$ and $\supp\chi\subset\set{\ak\leq1}$, and set $\vp:=\chi(\cdot/2)-\chi$. Define $\vp_{-1}:=\chi$ and $\vp_\ell:=\vp(2^{-\ell}\cdot)$ for $\ell\geq0$, with associated projections $P_\ell:=\vp_\ell(i\nabla_x)$ for every $\ell\geq-1$. Only Besov norms with dyadic summability index $1$ are used below. For $s\in\R$ and $1\leq p\leq\infty$, define the Besov norm by
\begin{align}\label{def:Besovnorm}
    \norm{f}_{\Bes^s_{p,1}}
    :=\norm{P_{-1}f}_{L^p}+
    \sum_{\ell\geq0}2^{s\ell}\norm{P_\ell f}_{L^p}.
\end{align}
The symbols form an inhomogeneous Littlewood--Paley partition of unity, since the sum
\begin{align}\label{eq:LP-partition}
    1=\chi+\sum_{\ell\geq0}\vp(2^{-\ell}\cdot)
\end{align}
telescopes pointwise. Each $\vp_\ell$ is supported in $\set{2^{\ell-1}\leq\ak\leq2^{\ell+1}}$ for every $\ell\geq0$ and $\vp$ is supported in $\set{1/2\leq\ak\leq2}$. The finite overlap of the cutoffs yields the fixed-frequency splitting
\begin{align}\label{eq:chi-splitting}
    \norm{f}_{\Bes^s_{p,1}}
    \sim
    \norm{\chi(i\nabla_x)f}_{L^p}+\norm{(1-\chi(i\nabla_x))f}_{\Bes^s_{p,1}},
\end{align}
as verified in Appendix~\ref{app:LP}, which also collects the elementary properties of these norms and the multiplier and embedding estimates used throughout.

For a non-negative integer $N$, $1\leq p\leq\infty$ and $\delta\in (0,1)$, the unified data norm of an initial triple $(\Ein,\Bin,\fin)$ is defined by
\begin{align}\label{eq:data-norm}
\begin{split}
    &\hspace{-1cm}\mathcal D_{N,p,\delta}(\Ein,\Bin,\fin)\\
    &:=\norm{\abs{\nabla_x}^{-\delta}\ETin}_{\Bes^{N+4}_{1,1}}
    +\norm{\abs{\nabla_x}^{-\delta}\Bin}_{\Bes^{N+4}_{1,1}}
    +\norm{\abs{\nabla_x}^{-\delta}\Delta_x^{-1}\nabla_x\times\Bin}_{\Bes^{N+4}_{1,1}} \\
    &\quad+\norm{\chi(i\nabla_x)\abs{\nabla_x}^{-\delta}\Jin}_{L^1_x}
+
\norm{\jb{v}^{4}\fin}_{W^{N+4,1}_{x,v}}+\norm{\fin}_{L^p_vW^{N,\infty}_x},
\end{split}
\end{align}
where $\Jin:=j[\fin]
=\int_{\R^3}v\fin(\cdot,v)\,dv$. For a non-negative integer $N$ and $1\leq p\leq\infty$, the phase-space discrepancy of two initial perturbations is
\begin{align}\label{eq:phase-space-matching}
    \mathcal M_{N,p}(\fin,\fVPin):=\norm{\fin-\fVPin}_{L^p_vW^{N,\infty}_x}
    +\norm{\jb{v}(\fin-\fVPin)}_{W^{N+3,1}_{x,v}}.
\end{align}

In the theorem statement and in Figure~\ref{map}, a superscript $c$ marks the Vlasov--Maxwell objects that are allowed to vary with the light speed: the initial data, the fields, the perturbation and its scattering state. For a family $\set{(\Ein^c,\Bin^c,\fin^c)}_c$ of Vlasov--Maxwell data indexed by the light speed, write $\mathcal D^c_{N,p,\delta}:=\mathcal D_{N,p,\delta}(\Ein^c,\Bin^c,\fin^c)$. On the other hand, the Vlasov--Poisson datum, solution and scattering state used for comparison are fixed independently of~$c$. For a family $\set{\fin^c}_c$ of Vlasov--Maxwell data and a single Vlasov--Poisson datum $\fVPin$, write $\mathcal M^c_{N,p}:=\mathcal M_{N,p}(\fin^c,\fVPin)$.

\begin{theorem}[Landau damping, scattering and the electrostatic limit]\label{maintheorem}
Fix a non-negative integer~$N$, $1\leq p\leq\infty$ and $\delta\in(0,1)$, and let $\mu$ be the Poisson equilibrium~\eqref{Poisson}. There exists $c_0\geq1$ with the following properties, in which all implicit constants depend only on $N$, $p$ and $\delta$. 

Let $\set{(\Ein^c,\Bin^c,\fin^c)}_{c\geq c_0}$ be a family of initial data satisfying~\eqref{compatibilityconditions}, with $\mathcal D^c_{N,p,\delta}<\infty$ for every $c\geq c_0$. Then, for every $c\geq c_0$, there exists a unique global solution $(E^c,B^c,f^c)$ of the linearised Vlasov--Maxwell system~\eqref{linVM} with initial data $(\Ein^c,\Bin^c,\fin^c)$, and the following hold.
\begin{enumerate}[label=(\roman*)]
\item \label{maintheorem-longitudinal}For every multi-index $\alpha$ with $\aal\leq N$ and every $t\geq0$,
\begin{align*}
\norm{\nabla_x^\alpha E_L^c(t)}_{L^\infty}
\lesssim
\frac{\mathcal D^c_{N,p,\delta}}
{\jb{t}^{2+\aal}}.
\end{align*}

\item \label{maintheorem-transverse} The transverse fields decompose into oscillatory and remainder parts:
\begin{align*}
E_T^c=E_T^{c,\mathrm{osc}}+E_T^{c,\mathrm r},
\qquad
B^c=B^{c,\mathrm{osc}}+B^{c,\mathrm r},
\end{align*}
where, for $t\geq0$,
\begin{align}\label{maintheoremosc}
\norm{E_T^{c,\mathrm{osc}}(t)}_{W^{N,\infty}}
+\norm{B^{c,\mathrm{osc}}(t)}_{W^{N,\infty}}
\lesssim
\frac{\mathcal D^c_{N,p,\delta}}
{\jb{ct,c^{\frac{1}{2}}t^{\frac{3}{2}}}},
\end{align}
and, for every multi-index $\alpha$ with $\aal\leq N$ and $t\geq0$,
\begin{align}
\norm{\nabla_x^\alpha E_T^{c,\mathrm r}(t)}_{L^\infty}
\lesssim
\frac{\mathcal D^c_{N,p,\delta}}
{c^2\jb{t}^{\frac{4}{3}+\frac{\aal}{3}}},\qquad
\norm{\nabla_x^\alpha B^{c,\mathrm r}(t)}_{L^\infty}
\lesssim
\frac{\mathcal D^c_{N,p,\delta}}
{c\jb{t}^{\frac{4}{3}+\frac{\aal}{3}}}.\label{maintheoremrem}
\end{align}

\item \label{maintheorem-scattering}There exists $f_\infty^c\in L^p_vW^{N,\infty}_x$ such that, for $t\geq0$,
\begin{align*}
\norm{f^c(t,x+tv,v)-f_\infty^c(x,v)}_{L^p_vW^{N,\infty}_x}
\lesssim
\frac{\mathcal D^c_{N,p,\delta}}
{\jb{t}^{\frac{1}{3}}}.
\end{align*}
\end{enumerate}

Moreover, let $\fVPin$ be a single datum, fixed independently of $c$, satisfying~\eqref{zeromassVP}, with $\mathcal M^c_{N,p}<\infty$ for every $c\geq c_0$, and let~$\fVP$ be the corresponding solution of the linearised Vlasov--Poisson system~\eqref{linVP}. Then~$\fVP(t,x+tv,v)$ converges in $L^p_vW^{N,\infty}_x$ as $t\to\infty$ to a scattering state $\fVP_\infty\in L^p_vW^{N,\infty}_x$, and, for every $c\geq c_0$,
\begin{align}\label{maintheoremelectrostaticlimit}
    \sup_{t\geq0}
    \norm{f^c(t)-\fVP(t)}_{L^p_vW^{N,\infty}_x}
    +
    \norm{f_\infty^c-\fVP_\infty}_{L^p_vW^{N,\infty}_x}
    \lesssim \mathcal M^c_{N,p}
    +\frac{1}{c}\mathcal D^c_{N,p,\delta}.
\end{align}
The rate $c^{-1}$ in~\eqref{maintheoremelectrostaticlimit} is sharp within the stated data class. 
\end{theorem}

The uniform-in-time estimate~\eqref{maintheoremelectrostaticlimit} concerns the particle distributions and their scattering states. It does not assert uniform convergence of the full electromagnetic fields on $[0,\infty)$, which cannot hold in general in the present data class because $\ETin^c$ and $\Bin^c$ may remain of order one as $c\to\infty$. For every fixed $\tau>0$, however, part~\ref{maintheorem-transverse} and the longitudinal comparison~\eqref{eq:VP-longitudinal-decay} give, for every $c\geq c_0$,
\begin{align*}
    \sup_{t\geq\tau}
    \br{
    \norm{E^c(t)-\EVP(t)}_{W^{N,\infty}}
    +\norm{B^c(t)}_{W^{N,\infty}}
    }
    \lesssim_\tau
    \mathcal M^c_{N,p}
    +\frac{1}{c}\mathcal D^c_{N,p,\delta}.
\end{align*}

The decay rates, the choice of norms and the assumptions on the initial data of Theorem~\ref{maintheorem} are now discussed in turn.

\begin{remark}[Field decay rates]
In the longitudinal sector, the density satisfies the same closed Volterra equation as in linearised Vlasov--Poisson. For the Poisson equilibrium, its resolvent preserves the phase-mixing decay inherited from free transport, giving the base rate $\jb{t}^{-2}$ for $E_L$ and one additional power of time decay for each spatial derivative.

The oscillatory component of the transverse fields is formed by spatial Fourier modes that oscillate in time. In physical space, it is a dispersive Klein--Gordon-type wave with propagation speed $c$. The bound in~\eqref{maintheoremosc} separates three temporal regimes. On the initial interval $t\lesssim c^{-1}$, dispersion has not yet produced decay and the fields may remain of order one. For $c^{-1}\lesssim t\lesssim c$, angular dispersion near the wave cone $\abs{x}\sim ct$ gives the wave-like rate $(ct)^{-1}$. For $t\gtrsim c$, radial curvature supplies an additional half-power of time decay, giving the Klein--Gordon rate $c^{-1/2}t^{-3/2}$. Spatial derivatives do not improve these time-decay rates.

By contrast, the non-oscillatory transverse remainders decay through frequency-dependent damping and free-transport velocity averaging rather than wave dispersion. Their electric and magnetic components carry different powers of $c$, but share the time-decay factor $\jb{t}^{-(4+\aal)/3}$ in~\eqref{maintheoremrem}, each spatial derivative improving the decay by one third of a power. For $t\geq1$, both parts of the transverse fields are $O(c^{-1})$.
\end{remark}

\begin{remark}[Choice of norms]
The choice of norms follows the decomposition of the field. The oscillatory component is governed by Klein--Gordon Fourier multipliers, whose oscillatory factors are of unit modulus and so preserve $L^2$-based norms rather than decay in them. The decay comes instead from the curvature of the phase and is measured from $L^1$ to $L^\infty$ on each dyadic block. Summing the blocks in $\ell^1$ places the norms $\Bes^s_{1,1}$ on the initial data and makes the conclusion pointwise, through the embedding $\Bes^0_{\infty,1}\hookrightarrow L^\infty$ of Lemma~\ref{lem:embeddings}\ref{emb:besov-sobolev}. 

The longitudinal field and the non-oscillatory remainder are estimated instead by direct Fourier integration: for each spatial derivative the proof controls the Fourier $L^1_k$ norm, which dominates the~$\Bes^0_{\infty,1}$ norm by Lemma~\ref{lem:embeddings}\ref{emb:fourier-L1} and hence the $L^\infty$ norm as well. 

Both routes to a pointwise bound therefore pass through the stronger norm $\Bes^0_{\infty,1}$, so the field estimates of Theorem~\ref{maintheorem} also hold with $\Bes^0_{\infty,1}$ in place of $L^\infty$ and with $\Bes^{N}_{\infty,1}$ in place of $W^{N,\infty}$. 
\end{remark}

\begin{remark}[Initial data]
The factors $\abs{\nabla_x}^{-\delta}$ in~\eqref{eq:data-norm} are introduced in the present proof to obtain the low-frequency estimates with constants independent of $c$. The oscillatory estimates involve order-zero multipliers that need not be bounded on $L^1$. Transferring a factor $\ak^\delta$ to the multiplier makes its localised kernel integrable, at the cost of applying $\abs{\nabla_x}^{-\delta}$ to the data. At high frequency, these factors merely lower the Besov index.

Finiteness of the negative-order terms requires cancellation conditions on the data. Since $\ETin$ and~$\Bin$ are integrable and divergence-free, both have zero spatial mean. Finiteness also requires
\begin{align*}
\int_{\R^3}\Delta_x^{-1}\nabla_x\times\Bin\,dx=0,
\qquad
\int_{\R^3}\int_{\R^3}v\fin(x,v)\,dvdx=0.
\end{align*}
When the first spatial moment of $\Bin$ is finite, the condition $\nabla_x\cdot\Bin=0$ makes its first-moment matrix antisymmetric. The first cancellation condition is then equivalent to $\int_{\R^3}x\times\Bin(x)\,dx=0$, which implies that all first moments of $\Bin$ vanish. The second is the zero-total-current condition and is distinct from the zero-mass condition in~\eqref{compatibilityconditions}.

For $\delta\in(0,1)$, convenient sufficient conditions for $\mathcal D_{N,p,\delta}<\infty$ are
\begin{align*}
\ETin, \Bin\in W^{N+4,1},
\qquad
\jb{x}\ETin, \jb{x}^2\Bin\in L^1,
\qquad
\int_{\R^3}x\times\Bin(x)\,dx=0,
\end{align*}
together with
\begin{align*}
\norm{\jb{x}\jb{v}\fin}_{L^1_{x,v}}
+\norm{\jb{v}^4\fin}_{W^{N+4,1}_{x,v}}
+\norm{\fin}_{L^p_vW^{N,\infty}_x}<\infty,
\qquad
\int_{\R^3}\int_{\R^3}v\fin(x,v)\,dvdx=0.
\end{align*}
To verify sufficiency, the high-frequency terms are controlled by the stated Sobolev regularity. At low frequency, the zero means and first spatial moments control the inverse derivatives of $\ETin$, $\Bin$ and~$\Jin$. The term $\Delta_x^{-1}\nabla_x\times\Bin$, which is one order more singular, is controlled by the second spatial moment of $\Bin$ together with the vanishing of its mean and first moments. For the current term, the required spatial moment is bounded by $\norm{\jb{x}\Jin}_{L^1}\lesssim\norm{\jb{x}\jb{v}\fin}_{L^1_{x,v}}$.
\end{remark}

\begin{figure}[ht]
\centering
    \begin{tikzpicture}[domain=0:3,xscale=0.28,yscale=0.28]
\node at (-15,0) {$\fVPin$};
\node at (-15,14) {$\fin^c$};
\draw [->] (-13.25,0)--(-3,0);
\draw [->] (-13.25,14)--(-1.75,14);
\draw [->] (3,0)--(13.25,0);
\draw [->] (1.75,14)--(13.25,14);
\draw [->] (0,12.5)--(0,1.5);
\draw [->] (15,12.5)--(15,1.5);
\draw [->] (-15,12.5)--(-15,1.5);
\node at (0,0) {$\fVP(t)$};
\node at (15,0) {$\fVP_\infty$};
\node at (15,14) {$f_\infty^c$};
\node at (0,14) {$f^c(t)$};
\node [below] at (7.5,0) {\footnotesize{$t\to\infty$}};
\node [above] at (7.5,14) {\footnotesize{$t\to\infty$}};
\node [left] at (0,7) {\footnotesize{$c\to\infty$}};
\node [left] at (15,7) {\footnotesize{$c\to\infty$}};
\node [left] at (-15,7) {\footnotesize{$c\to\infty$}};
\end{tikzpicture}
\captionsetup{width=.8\linewidth}
\caption{Commutation between the electrostatic limit and the scattering map, in the setting of Theorem~\ref{maintheorem}. Convergence as $t\to\infty$ is understood along free transport. For initial data controlled independently of $c$ and matched to order~$c^{-1}$, convergence as $c\to\infty$ holds uniformly in time at rate $c^{-1}$.}
\label{map}
\end{figure}
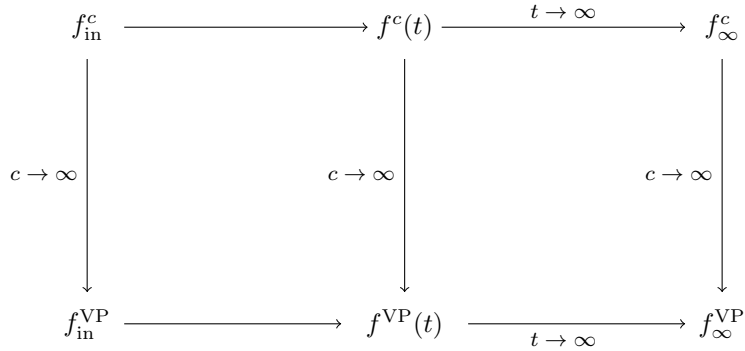

\subsection{Existing results}
Three lines of work meet in the present setting: Landau damping for the Vlasov--Poisson system, linear damping for the Vlasov--Maxwell system at fixed light speed, and the electrostatic limit.

\subsubsection*{The Vlasov--Poisson system}
Landau predicted damping of the electric field for the linearised equation about a Maxwellian equilibrium in 1946~\cite{Landau1946}. Mouhot and Villani proved nonlinear damping for small analytic perturbations of Penrose-stable equilibria on the torus~\cite{MouhotVillani2011}, where nonlinear plasma echoes obstruct a general extension of such damping estimates to Sobolev regularity~\cite{Bedrossian2021Echoes}. On the whole space, the spatial spreading produced by free transport permits finite-regularity arguments. For screened interactions, asymptotic stability holds with essentially the rates of free transport~\cite{BedrossianMasmoudiMouhot2018,HanKwanNguyenRousset2021Screened}, whereas for the Coulomb interaction the long-wave response depends strongly on the velocity tails of the equilibrium~\cite{GlasseySchaeffer1994,GlasseySchaeffer1995}. Thin-tailed equilibria such as the Maxwellian support weakly damped Langmuir modes, whose Klein--Gordon-type dispersion produces the rate $t^{-3/2}$~\cite{HanKwanNguyenRousset2021Linearized,BedrossianMasmoudiMouhot2022}, while for sufficiently heavy-tailed profiles the field decays on the phase-mixing time scale. Detailed linear descriptions for broad classes of radial equilibria were obtained in~\cite{NguyenSurvival2026,IonescuPausaderWangWidmayer2023Stability}.

For the unscreened Coulomb Vlasov--Poisson system on the whole space, Ionescu, Pausader, Wang and Widmayer proved the first nonlinear asymptotic-stability result near a smooth non-trivial homogeneous equilibrium. For the Poisson equilibrium~\eqref{Poisson}, they proved scattering at a polynomial rate~\cite{IonescuPausaderWangWidmayer2024Poisson}, and Nguyen, Wei and Zhang subsequently gave a streamlined proof~\cite{NguyenWeiZhang2025}. This equilibrium has heavy velocity tails, which prevent Langmuir waves from persisting beyond the phase-mixing time scale, and~$\widehat\mu(\eta)=e^{-\abs{\eta}}$ makes the linear response rational and the resolvent explicit.

\subsubsection*{The Vlasov--Maxwell system}
On the torus, Han-Kwan, Nguyen and Rousset~\cite{HanKwanNguyenRousset2018LongTime} proved Sobolev stability for the nonlinear relativistic system near Penrose-stable equilibria with constants independent of $c$, on time intervals controlled by powers of $c$ and of the inverse perturbation size. For unstable equilibria, Han-Kwan and Nguyen~\cite{HanKwanNguyen2016Instability} constructed perturbations initially polynomially small in $c^{-1}$ that become order one within time $O(\log c)$.

On the whole space, global classical existence for the nonlinear three-dimensional relativistic system remains open for general large data. Near vacuum, global existence and decay for small data were established by Glassey and Strauss~\cite{GlasseyStrauss1987}, with subsequent extensions to polynomially decaying data~\cite{Bigorgne2020SharpAsymptotics,Wang2022VM}. The initial Maxwell field can also be large, provided the initial particle distribution is sufficiently small~\cite{WeiYang2021,Bigorgne2022VM}. The long-time dynamics in this regime have been described through modified scattering~\cite{Bigorgne2022VM,PankavichBenArtzi2025,Breton2026ModifiedScattering} and the construction of a scattering map~\cite{Bigorgne2023ScatteringMap}. Under a generic condition on the asymptotic charge, Breton showed that linear scattering in $L^1$ fails~\cite{Breton2026Completeness}.

Han-Kwan, Nguyen and Rousset~\cite{HanKwanNguyenRousset2025LinearVM} established linear Landau damping for the relativistic Vlasov--Maxwell system about homogeneous equilibria on the whole space. At fixed light speed, they decomposed the electromagnetic field into oscillatory and regular components and proved algebraic decay of both for radial equilibria under analyticity and velocity-decay assumptions; the velocity-decay assumption excludes the Poisson equilibrium studied here.

\subsubsection*{The electrostatic limit}
The classical theory of the electrostatic limit concerns convergence on finite time intervals. For smooth solutions on the whole space, Asano and Ukai~\cite{AsanoUkai1986} and Degond~\cite{Degond1986} proved convergence from Vlasov--Maxwell with Newtonian particle transport to Vlasov--Poisson on a time interval independent of $c$. Schaeffer~\cite{Schaeffer1986} established the corresponding limit from relativistic Vlasov--Maxwell to Vlasov--Poisson at rate $O(c^{-1})$ under additional assumptions on the initial data. Brigouleix and Han-Kwan~\cite{BrigouleixHanKwan2022} extended the theory to measure-valued solutions under macroscopic bounds. The convergence bounds in these results depend on the length of the time interval. 

For the long-time problem, Hong and Pankavich~\cite{HongPankavich2026} used decay estimates to obtain uniform-in-time convergence of transported profiles and convergence of scattering states for near-vacuum Vlasov equations with short-range interactions, without Maxwell coupling. A uniform-in-time electrostatic limit with Maxwell coupling near a non-trivial homogeneous equilibrium is not covered by these results.

\subsection{Extensions and limitations}
The assumptions of the main theorem, namely the choice of the Poisson profile, its radiality, the Newtonian transport and the linearisation itself, play different roles and the corresponding extensions present difficulties of different kinds.

\vspace{-0.145cm}
\subsubsection*{The Poisson equilibrium}
The restriction to this equilibrium is a feature of the present proof rather than an expected restriction of the damping mechanism. Its usefulness rests on two distinct features: the rationality of its response functions and its heavy velocity tails. The same profile is also the equilibrium treated in the nonlinear Vlasov--Poisson theory of Ionescu, Pausader, Wang and Widmayer~\cite{IonescuPausaderWangWidmayer2024Poisson}, providing a direct point of comparison between the present electromagnetic analysis and an established nonlinear electrostatic theory.

The rationality of the response functions is a special algebraic feature of the chosen profile. Indeed, for $\Re\lam+\ak>0$,
\begin{align*}
    \int_0^\infty e^{-\lam t}\widehat\mu(kt)\,dt
    =
    \frac{1}{\lam+\ak}.
\end{align*}
Consequently, at each spatial frequency, the longitudinal and transverse spectral problems reduce to algebraic equations, permitting a complete analysis of the roots and residues uniformly in $c$. 

The second feature concerns the longitudinal conclusion, which is more sensitive to the velocity tails of the equilibrium. Qualitatively similar longitudinal damping behaviour is expected for a broader class of equilibria with sufficiently heavy tails, even when their response functions are no longer rational. For thin-tailed equilibria such as the Maxwellian, the longitudinal field instead contains weakly damped, dispersive Langmuir modes~\cite{IonescuPausaderWangWidmayer2023Stability,BedrossianMasmoudiMouhot2022,HanKwanNguyenRousset2021Linearized,NguyenSurvival2026}, so the present longitudinal decay rate is not expected to transfer unchanged. The transverse dispersive structure is expected to be less sensitive to the equilibrium tails. For the relativistic system at fixed light speed, Han-Kwan, Nguyen and Rousset~\cite{HanKwanNguyenRousset2025LinearVM} obtained a similar decomposition into Klein--Gordon-type oscillations and a remainder decaying essentially like $t^{-4/3}$ for a class of rapidly decaying radial equilibria that includes the Maxwellian. The mechanism behind the electrostatic limit may be more robust, since the two systems share the same longitudinal response for every radial equilibrium and the factor $c^{-1}$ rests on bounds for the vector potential uniform in $c$ and on the integrability of $\jb{v}\nabla_v\mu$.

For non-radial equilibria, the term $(v\times B)\cdot\nabla_v\mu$ need not vanish at the linear level, and the longitudinal and transverse equations no longer decouple in general. Such equilibria can also exhibit growing transverse modes, as exemplified by the Weibel instability~\cite{Weibel1959}.

\vspace{-0.145cm}
\subsubsection*{Newtonian particle transport}
Newtonian transport plays a structural role in the uniform-in-time comparison. The Vlasov--Maxwell and Vlasov--Poisson evolutions share the same free transport, so translation invariance in $x$ makes comparison of their distributions in $L^p_vW^{N,\infty}_x$ equivalent to comparison of their transported profiles. With relativistic transport, the velocity corresponding to momentum $v$ is
\begin{align*}
u_c(v):=\frac{v}{\sqrt{1+c^{-2}\abs{v}^2}}.
\end{align*}
Although $u_c(v)-v=O(c^{-2})$ on bounded momentum sets, the displacement $t(u_c(v)-v)$ is of order one on times of order $c^2$ and prevents, already for free transport, a general uniform-in-time comparison of the distributions in the same coordinates.

A comparison with relativistic Vlasov--Poisson at the same $c$ would preserve the common free transport and avoid this kinematic obstruction, although the comparison system would then also depend on $c$. For comparison with Newtonian Vlasov--Poisson, a natural formulation would instead compare the profiles in their respective free-transport frames, and their scattering states when these exist. Hong and Pankavich~\cite{HongPankavich2026} establish such profile and scattering-state comparisons for nonlinear Vlasov equations near vacuum with short-range interactions and without Maxwell coupling. Extending this approach to the present setting would require additional analysis.

Relativistic transport also changes the longitudinal response and removes the rational structure used here. Bounded particle speeds can prevent resonant damping of long-wave oscillations even for heavy momentum tails, so the longitudinal damping mechanism discussed above does not transfer directly to the relativistic setting~\cite{HanKwanNguyenRousset2025LinearVM}. Nevertheless, the Newtonian response is recovered as $c\to\infty$ on fixed time and frequency scales, so the associated benefit of heavy tails remains relevant in this regime.

\vspace{-0.145cm}
\subsubsection*{Linearity}

For the nonlinear problem with Newtonian transport around the Poisson equilibrium, the perturbation equation contains the additional bilinear forcing
\begin{align*}
\p_tf+v\cdot\nabla_xf+E\cdot\nabla_v\mu
=-\br{E+\frac{v}{c}\times B}\cdot\nabla_vf.
\end{align*}
The nonlinear density and current sources depend on the full Lorentz force and the evolving distribution, so the longitudinal and transverse dynamics are coupled.

Treating the right-hand side as a source gives a forced version of the Volterra equation for the density~\eqref{Volterra} with the same linear operator, while the transverse equation is inverted by the same resolvent kernels with an additional nonlinear current source. This is analogous to the formulations used in the nonlinear Vlasov--Poisson theory near the Poisson equilibrium~\cite{IonescuPausaderWangWidmayer2024Poisson,NguyenWeiZhang2025}. Estimates for these sources must account for interactions between the oscillatory fields and the evolving distribution, together with the deviation of the particle trajectories from free transport. These interactions are resonant when the field and transport frequencies match. In~\cite{IonescuPausaderWangWidmayer2024Poisson}, the nonlinear Vlasov--Poisson analysis combines a static--oscillatory field decomposition with different representations near and away from resonance and estimates for the nonlinear characteristics. 

\subsection{Outline of proof}
The main steps of the proof are outlined below, with the dependencies summarised in Figure~\ref{proof-structure}.
\begin{itemize}
\item Section~\ref{sectionlongitudinalfielddecay} isolates the longitudinal dynamics. At each spatial frequency, the transverse field disappears from the density equation, leaving the same scalar Volterra equation for the Vlasov--Maxwell and Vlasov--Poisson densities. For the Poisson equilibrium this equation can be inverted explicitly, and the resulting representation preserves the phase-mixing decay of the free-transport source. Formulating the estimate for a general initial datum allows the same argument to control $E_L$, $\EVP$ and $E_L-\EVP$.
\item Sections~\ref{sec:wave} and~\ref{resolventSection} analyse the transverse dynamics through the Coulomb-gauge vector potential. For $k\neq0$, Fourier and Laplace transformation yield a rational dispersion function with one negative real zero and a complex-conjugate pair; Laplace inversion then gives a three-mode representation of the resolvent kernels, leading to a decomposition of the transverse fields into oscillatory and remainder components.
\item Section~\ref{sectionSf} establishes the estimates for the free-transport current needed by both parts of this decomposition. 
\item Section~\ref{complexrootSection} treats the complex modes with frequency $\omega(\ak)=\beta(c\ak)$, where $\beta$ has the uniform curvature bounds of a Klein--Gordon phase, and Section~\ref{remainderSection} the non-oscillatory remainders.
\item Section~\ref{sec:proofs} first compares the data functionals adapted to the separate longitudinal, oscillatory and remainder estimates with the unified norm $\mathcal D_{N,p,\delta}$. The resulting field bounds are then combined to prove damping and the integrability in time required for scattering. Finally, the Vlasov--Maxwell and Vlasov--Poisson profile equations are compared along free transport. Since $E_T=-c^{-1}\p_tA$, integration by parts along these characteristics expresses the transverse impulse through the vector potential, and bounds on $A$ and on the time integral of $\nabla_xA$, uniform in $c$, then give the impulse estimate needed for both the uniform-in-time electrostatic limit and the convergence of the scattering states.
\end{itemize}

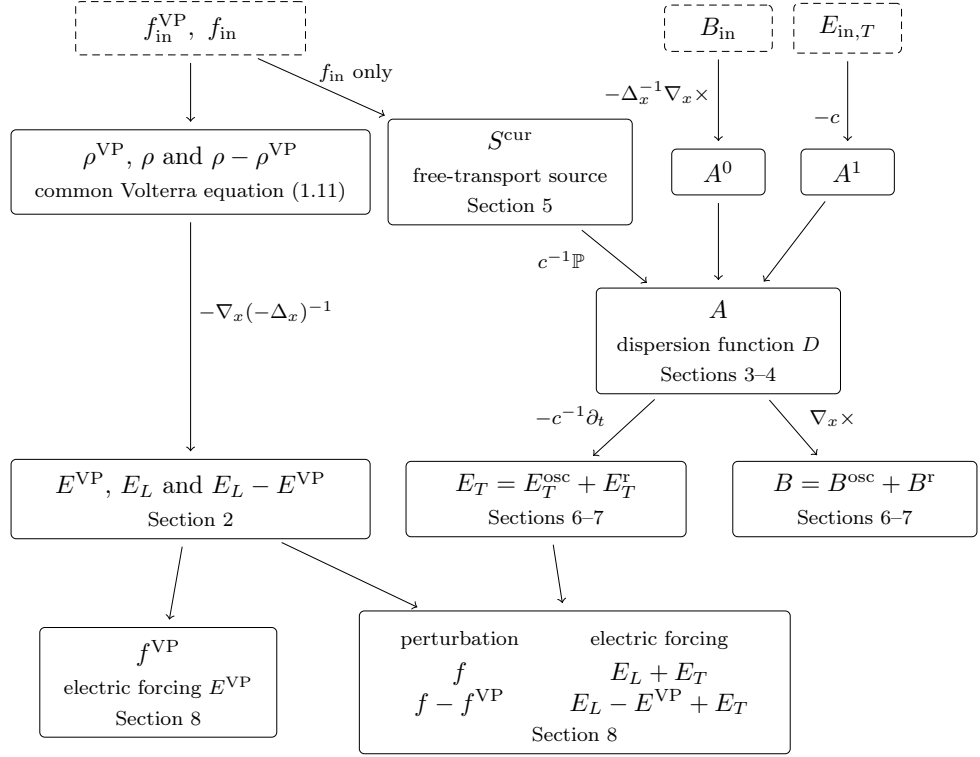
\begin{figure}[h]
\centering
\scalebox{0.95}{%
    \begin{tikzpicture}[
      box/.style={draw,rounded corners=2pt,align=center,inner sep=5pt},
      data/.style={draw,densely dashed,rounded corners=2pt,align=center,inner sep=5pt},
      flow/.style={->,shorten <=3pt,shorten >=3pt},
      op/.style={font=\footnotesize,align=center,inner sep=2pt,fill=white}
    ]
\node [data,minimum width=3.2cm] (kin) at (-4.35,4.7) {$\fVPin,\ \fin$};
\node [data,minimum width=1.5cm] (Bin) at (3.0,4.7) {$\Bin$};
\node [data,minimum width=1.5cm] (ETin) at (4.8,4.7) {$\ETin$};
\node [box,minimum width=5cm] (rho) at (-4.35,2.7)
    {$\rho^{\mathrm{VP}}$, $\rho$ and $\rho-\rho^{\mathrm{VP}}$\\[1pt]
     {\footnotesize common Volterra equation~\eqref{Volterra}}};
\node [box,minimum width=5cm] (EL) at (-4.35,-1.85)
    {$\EVP$, $E_L$ and $E_L-\EVP$\\[1pt]
     {\footnotesize\hyperref[sectionlongitudinalfielddecay]{Section~\ref*{sectionlongitudinalfielddecay}}}};
\node [box,minimum width=3.4cm] (cur) at (0.1,2.7)
    {$S^{\mathrm{cur}}$\\[1pt]
     {\footnotesize free-transport source}\\
     {\footnotesize\hyperref[sectionSf]{Section~\ref*{sectionSf}}}};
\node [box,minimum width=1.3cm] (A0) at (3.0,2.7) {$A^0$};
\node [box,minimum width=1.3cm] (A1) at (4.8,2.7) {$A^1$};
\node [box,minimum width=3.4cm] (A) at (3.0,0.35)
    {$A$\\[1pt]
     {\footnotesize dispersion function $D$}\\
     {\footnotesize\hyperref[sec:wave]{Sections~\ref*{sec:wave}}--\hyperref[resolventSection]{\ref*{resolventSection}}}};
\node [box,minimum width=3.9cm] (ET) at (0.6,-1.85)
    {$E_T=\ETosc+E_T^{\mathrm r}$\\[1pt]
     {\footnotesize\hyperref[complexrootSection]{Sections~\ref*{complexrootSection}}--\hyperref[remainderSection]{\ref*{remainderSection}}}};
\node [box,minimum width=3.4cm] (B) at (4.9,-1.85)
    {$B=\Bosc+B^{\mathrm r}$\\[1pt]
     {\footnotesize\hyperref[complexrootSection]{Sections~\ref*{complexrootSection}}--\hyperref[remainderSection]{\ref*{remainderSection}}}};
\node [box,minimum width=3.3cm] (vp) at (-4.8,-4.4)
    {$\fVP$\\[1pt]
     {\footnotesize electric forcing $\EVP$}\\
     {\footnotesize\hyperref[sec:proofs]{Section~\ref*{sec:proofs}}}};
\node [box,minimum width=6cm] (dist) at (1.0,-4.4)
    {$\begin{array}{c@{\qquad}c}
        \text{\footnotesize perturbation} & \text{\footnotesize electric forcing}\\[2pt]
        f              & E_L+E_T\\
        f-\fVP         & E_L-\EVP+E_T
     \end{array}$\\[2pt]
     {\footnotesize\hyperref[sec:proofs]{Section~\ref*{sec:proofs}}}};
\draw [flow] (kin)--(rho);
\draw [flow] (kin)--node [op,inner sep=1pt,pos=.46,anchor=south west] {$\fin$ only} (cur);
\draw [flow] (Bin)--node [op,pos=.42,left=1pt] {$-\Delta_x^{-1}\nabla_x\times$} (A0);
\draw [flow] (ETin)--node [op,pos=.68,left=1pt] {$-c$} (A1);
\draw [flow] (rho)--node [op,pos=.4,right=1pt] {$-\nabla_x(-\Delta_x)^{-1}$} (EL);
\draw [flow] (cur)--node [op,pos=.25,anchor=north east,xshift=-2pt,yshift=-2pt] {$c^{-1}\mathbb P$} (A);
\draw [flow] (A0)--(A);
\draw [flow] (A1)--(A);
\draw [flow] (A)--node [op,pos=.6,anchor=south east,xshift=-2pt,yshift=1pt] {$-c^{-1}\p_t$} (ET);
\draw [flow] (A)--node [op,pos=.6,anchor=south west,xshift=2pt,yshift=1pt] {$\nabla_x\times$} (B);
\draw [flow] (EL)--(vp);
\draw [flow] (EL)--(dist);
\draw [flow] (ET)--(dist);
\end{tikzpicture}}
\captionsetup{width=.92\linewidth}
\caption{Dependencies in the proof.}
\label{proof-structure}
\end{figure}

\section{Decay of the longitudinal electric field}\label{sectionlongitudinalfielddecay}
This section proves a common decay estimate for the longitudinal electric fields of the two systems and for their difference. The argument begins with their shared Vlasov equation,
\begin{align*}
    \p_tf+v\cdot\nabla_xf+E\cdot\nabla_v\mu=0,
\end{align*}
where $(f,E)$ denotes the perturbation and electric field of either~\eqref{linVM} or~\eqref{linVP}. Fourier transformation in $(x,v)$ gives
\begin{align*}
    \p_t\widehat f(t,k,\eta)-k\cdot\nabla_\eta\widehat f(t,k,\eta)
    +i\widehat E(t,k)\cdot\eta\widehat\mu(\eta)=0.
\end{align*}
Duhamel's formula along the free-transport characteristics then yields
\begin{align}\label{eq:vlasov-duhamel}
    \widehat f(t,k,\eta)
    =\widehat\fin(k,\eta+kt)
    -i\int_0^t\widehat E(s,k)\cdot\br{\eta+k(t-s)}
    \widehat\mu\br{\eta+k(t-s)}\,ds.
\end{align}
Since $\widehat\rho(t,k)=\widehat f(t,k,0)$, evaluating at $\eta=0$ gives the density equation
\begin{align*}
    \widehat\rho(t,k)
    =\widehat\fin(k,kt)
    -i\int_0^t(t-s)\widehat E(s,k)\cdot k\widehat\mu\br{k(t-s)}\,ds.
\end{align*}
The electric field therefore enters the density equation only through $\widehat E(s,k)\cdot k$. At $k=0$, the zero-mass conditions~\eqref{compatibilityconditions} and~\eqref{zeromassVP} give $\widehat\rho(t,0)=\widehat\fin(0,0)=0$ in either system. In the Vlasov--Maxwell system, only the longitudinal field contributes because $k\cdot\widehat E_T=0$; in Vlasov--Poisson, the field is already longitudinal. In either case, Gauss's law~\eqref{longitudinal-fields} gives $k\cdot\widehat E=-i\widehat\rho$. Substitution yields the same scalar Volterra equation~\eqref{Volterra} for both densities, with their respective initial data and no dependence on the transverse field or on~$c$.

By linearity, a single estimate for a general initial datum $F_{\mathrm{in}}$ will apply to both systems and their difference. Denoting the corresponding density by $\rho_F$, equation~\eqref{Volterra} becomes, for $k\neq0$,
\begin{align}\label{VolterraF}
    \widehat\rho_F(t,k)
    =\widehat F_{\mathrm{in}}(k,kt)
    -\int_0^t(t-s) e^{-\ak(t-s)} \widehat\rho_F(s,k)\,ds,
\end{align}
where the identity $\widehat\mu\br{k(t-s)}=e^{-\ak(t-s)}$ follows from the choice of the Poisson equilibrium~\eqref{Poisson}. The associated longitudinal field $E_F$ is then defined, as in~\eqref{longitudinal-fields}, by
\begin{align}\label{EF-rhoF}
    \widehat E_F(t,k)=-i\frac{k}{\ak^2} \widehat\rho_F(t,k),
    \qquad k\neq0.
\end{align}
For an integer $s\geq0$, set
\begin{align}\label{eq:longitudinal-data}
    \mathcal H_s\sbr{F_{\mathrm{in}}}:=\sup_{k,\eta\in\R^3}\jb{k,\eta}^s\br{\abs{\widehat F_{\mathrm{in}}(k,\eta)}+\abs{\nabla_\eta\widehat F_{\mathrm{in}}(k,\eta)}}.
\end{align}
This functional is controlled by a weighted Sobolev norm. Indeed, multiplication by $k$ and $\eta$ in Fourier space corresponds to differentiation in $x$ and $v$, respectively, while the Fourier transform maps $L^1$ into $L^\infty$ and $\nabla_\eta\widehat F_{\mathrm{in}}=\widehat{-ivF_{\mathrm{in}}}$. Hence
\begin{align}\label{Hs-data-control}
    \mathcal H_s\sbr{F_{\mathrm{in}}}
    \lesssim\norm{F_{\mathrm{in}}}_{W^{s,1}_{x,v}}+\norm{vF_{\mathrm{in}}}_{W^{s,1}_{x,v}}
    \lesssim\norm{\jb{v}F_{\mathrm{in}}}_{W^{s,1}_{x,v}},
\end{align}
the last inequality by Lemma~\ref{lem:velocity-weights}.

One elementary frequency integral is used here and again in Section~\ref{remainderSection}: for every integer $a\geq0$, every $m>a+2$ and every $t\geq0$,
\begin{align}\label{eq:freq-integral}
    \int_{\R^3}\frac{\ak^{a-1}}{\jb{k,kt}^{m}}\,dk\lesssim\frac{1}{\jb{t}^{a+2}},
\end{align}
by the identity $\jb{k,kt}=\jb{k\jb{t}}$ and the substitution $k\mapsto k/\jb{t}$, the resulting integral converging at the origin because $a\geq0$ and at infinity because $m>a+2$.

\begin{proposition}[Longitudinal decay]\label{ELtheorem}
Let $\jb{v}F_{\mathrm{in}}\in L^1_{x,v}$, let $\rho_F$ satisfy~\eqref{VolterraF} with initial datum $F_{\mathrm{in}}$, with $\widehat\rho_F(\cdot,k)$ continuous on $[0,\infty)$ for each $k\neq0$, and let $E_F$ be given by~\eqref{EF-rhoF}. If $M\geq3$ is an integer and $\mathcal H_M\sbr{F_{\mathrm{in}}}<\infty$, then, for every multi-index $\alpha$ with $\aal<M-2$ and every $t\geq0$,
    \begin{align*}
        \norm{\nabla_x^\alpha E_F(t)}_{L^\infty}\lesssim\frac{\mathcal H_M\sbr{F_{\mathrm{in}}}}{\jb{t}^{2+\aal}}.
    \end{align*}
\end{proposition}

\begin{proof}[Proof of Proposition~\ref{ELtheorem}]
Fix $k\neq0$ and set $G_k(t):=\widehat F_{\mathrm{in}}(k,kt)$. The assumptions on $F_{\mathrm{in}}$ imply that $G_k$ is bounded and continuously differentiable, while standard Volterra theory gives a unique continuous solution of exponential order. For a function of time $h$, write
\begin{align*}
    \widetilde h(\lam):=\int_0^\infty e^{-\lam t}h(t)\,dt
\end{align*}
for its Laplace transform, defined wherever the integral converges absolutely. For quantities depending on $x$, the tilde denotes the Laplace transform of the spatial Fourier transform. Transforming~\eqref{VolterraF} therefore gives, for $\Re\lam$ sufficiently large,
\begin{align*}
    \widetilde\rho_F(\lam,k)
    =\frac{\widetilde G_k(\lam)}{1+(\lam+\ak)^{-2}}
    =\widetilde G_k(\lam)\br{1-\frac{1}{(\lam+\ak)^2+1}}.
\end{align*}
Since the transform of $e^{-\ak t}\sin t$ is $\br{(\lam+\ak)^2+1}^{-1}$, inverting gives
\begin{align*}
    \widehat\rho_F(t,k)=G_k(t)-\int_0^tG_k(s)e^{-\ak(t-s)}\sin(t-s)\,ds.
\end{align*}
As $\br{\p_s-\ak}\br{e^{-\ak(t-s)}\cos(t-s)}=e^{-\ak(t-s)}\sin(t-s)$, integrating by parts in $s$ turns this into
\begin{align}\label{rho-ibp-longitudinal}
    \widehat\rho_F(t,k)=\widehat F_{\mathrm{in}}(k,0)e^{-\ak t}\cos t+\int_0^t\br{\p_s+\ak}\widehat F_{\mathrm{in}}(k,ks)e^{-\ak(t-s)}\cos(t-s)\,ds,
\end{align}
the upper endpoint cancelling the free term $G_k(t)$ and the lower contributing $G_k(0)e^{-\ak t}\cos t$.

The definition of $\mathcal H_M\sbr{F_{\mathrm{in}}}$, together with the identity $\p_s\widehat F_{\mathrm{in}}(k,ks)=k\cdot\nabla_\eta\widehat F_{\mathrm{in}}(k,ks)$, gives
\begin{align*}
    \abs{\widehat F_{\mathrm{in}}(k,0)}\lesssim\frac{\mathcal H_M\sbr{F_{\mathrm{in}}}}{\jb{k}^{M}}\qaq \abs{\br{\p_s+\ak}\widehat F_{\mathrm{in}}(k,ks)}\lesssim\frac{\ak\mathcal H_M\sbr{F_{\mathrm{in}}}}{\jb{k,ks}^{M}}.
\end{align*}
Hence
\begin{align*}
    \abs{\widehat\rho_F(t,k)}\lesssim\mathcal H_M\sbr{F_{\mathrm{in}}}\frac{e^{-\ak t}}{\jb{k}^M}+\mathcal H_M\sbr{F_{\mathrm{in}}}\int_0^t\frac{\ak e^{-\ak(t-s)}}{\jb{k,ks}^M}\,ds,
\end{align*}
where the boundary term satisfies $e^{-\ak t}\jb{k}^{-M}\lesssim\jb{k,kt}^{-M}$. 

Split the integral at $s=t/2$. On $[0,t/2]$, the bound $\jb{k,ks}\geq\jb{k}$ gives
\begin{align*}
    \int_0^{\frac{t}{2}}\frac{\ak e^{-\ak(t-s)}}{\jb{k,ks}^M}\,ds
    \leq\frac{e^{-\frac{\ak t}{2}}}{\jb{k}^M}
    \lesssim\frac{1}{\jb{k,kt}^M},
\end{align*}
where the last inequality follows from $\jb{k,kt}\leq\jb{k}\jb{kt}$ and exponential decay. On $[t/2,t]$, the lower bound $\jb{k,ks}\gtrsim\jb{k,kt}$ gives
\begin{align*}
    \int_{\frac{t}{2}}^t\frac{\ak e^{-\ak(t-s)}}{\jb{k,ks}^M}\,ds\lesssim\frac{1}{\jb{k,kt}^M}\int_{\frac{t}{2}}^t\ak e^{-\ak(t-s)}\,ds\lesssim\frac{1}{\jb{k,kt}^M}.
\end{align*}
Hence $\abs{\widehat\rho_F(t,k)}\lesssim\mathcal H_M\sbr{F_{\mathrm{in}}}\jb{k,kt}^{-M}$ and, by~\eqref{EF-rhoF}, $\abs{\widehat E_F(t,k)}\lesssim\mathcal H_M\sbr{F_{\mathrm{in}}}\ak^{-1}\jb{k,kt}^{-M}$. 

Integrating in $k$ and applying~\eqref{eq:freq-integral} with $a=\aal$, which is admissible because $M>\aal+2$, gives
\begin{align*}
    \norm{\nabla_x^\alpha E_F(t)}_{L^\infty}
    \lesssim\mathcal H_M\sbr{F_{\mathrm{in}}}\int_{\R^3}\frac{\ak^{\aal-1}}{\jb{k,kt}^M}\,dk
    \lesssim\frac{\mathcal H_M\sbr{F_{\mathrm{in}}}}{\jb{t}^{\aal+2}}.\tag*{\qedhere}
\end{align*}
\end{proof}

\section{Coulomb-gauge vector potential and transverse dispersion}\label{sec:wave}

This section derives a closed equation for the Coulomb-gauge vector potential $A$ and passes to its Fourier--Laplace representation. The resulting scalar transverse dispersion function has one real root and a complex-conjugate pair, which are then located and estimated. Here, and throughout, set~$\ep=c^{-1}$.

\subsection{Closed equation for the vector potential}

Recall from the Coulomb-gauge formulation~\eqref{Coulombgauge} of the introduction that $E=-\nabla_x\phi-\ep\p_tA$ and $B=\nabla_x\times A$, with $\nabla_x\cdot A=0$. In particular, $E_L=-\nabla_x\phi$ and $E_T=-\ep\p_tA$. Set $A^0:=A|_{t=0}$ and $A^1:=\p_tA|_{t=0}$. Since
\begin{align*}
    \nabla_x\times B=\nabla_x\times\br{\nabla_x\times A}=\nabla_x(\nabla_x\cdot A)-\Delta_xA=-\Delta_xA,
\end{align*}
and $\ETin=-\ep A^1$, the initial data for the potential satisfy
\begin{align*}
    A^0=-\Delta_x^{-1}\br{\nabla_x\times \Bin},\qquad \ep A^1=-\ETin.
\end{align*}
Introduce the free-transport current of the initial perturbation,
\begin{align}
    S^{\mathrm{cur}}(t,x)
    :=\int_{\R^3}v\fin(x-tv,v)\,dv,\qquad \widehat S^{\mathrm{cur}}(t,k)
    =i\nabla_\eta\widehat\fin(k,kt),
    \label{Scurdef}
\end{align}
with $S^{\mathrm{cur}}(0)=\Jin$. Fourier--Laplace transforms are denoted by a tilde, following the convention of Section~\ref{sectionlongitudinalfielddecay}. The Fourier symbol of the Leray projection $\mathbb P$ of~\eqref{APhi} is $\mathbb P_k=I-\ak^{-2}k\otimes k$.

\begin{lemma}[Fourier--Laplace equation for the vector potential]\label{lem:transverse-wave-equation}
Let $\mu$ be given by~\eqref{Poisson}, and let $A$ be the Coulomb-gauge vector potential of a solution of~\eqref{linVM} whose initial datum satisfies $\jb{v}\fin\in L^1_{x,v}$. Then, for $k\neq0$ and $\Re\lam$ sufficiently large, its Fourier--Laplace transform $\widetilde A$ satisfies
\begin{align}\label{eq:wave-final}
    D(\lam,\ak)\widetilde A(\lam,k)
    =\ep\mathbb P_k\widetilde S^{\mathrm{cur}}(\lam,k)
    +\ep^2\br{\lam\widehat A^0(k)+\widehat A^1(k)+\frac{\widehat A^0(k)}{\lam+\ak}},
\end{align}
where the transverse dispersion function is
\begin{align}
    D(\lam,r):=\ep^2\lam^2+r^2+\frac{\ep^2\lam}{\lam+r}.
    \label{eq:wave-dispersion-function}
\end{align}
\end{lemma}
\begin{proof}
The Leray projection annihilates gradients and fixes divergence-free fields, so that $\mathbb PE=E_T=-\ep\p_tA$ and $\mathbb P\br{\nabla_x\times B}=\nabla_x\times B$. Applying $\mathbb P$ to Amp\`ere's law $\nabla_x\times B=\ep j[f]+\ep\p_tE$ and recalling $\nabla_x\times B=-\Delta_xA$ therefore gives the forced wave equation $\ep^2\p_t^2A-\Delta_xA=\ep \mathbb Pj[f]$ of~\eqref{APhi}, which in Fourier variables reads
\begin{align*}
    \br{\ep^2\p_t^2+\ak^2}\widehat A(t,k)=\ep\mathbb P_k\widehat{j[f]}(t,k).
\end{align*}
The current on the right is computed from $\widehat{j[f]}(t,k)=i\nabla_\eta\widehat f(t,k,0)$ and the Duhamel formula~\eqref{eq:vlasov-duhamel}, in which $\xi_s:=\eta+k(t-s)$ is written for the shifted variable. Since $\nabla_\eta\xi_s=I$,
\begin{align*}
    \nabla_\eta\sbr{\widehat E(s,k)\cdot\xi_s\,\widehat\mu(\xi_s)}
    =\widehat\mu(\xi_s)\widehat E(s,k)+\br{\widehat E(s,k)\cdot\xi_s}\nabla_\eta\widehat\mu(\xi_s)
\end{align*}
wherever $\xi_s\neq0$. The profile $\widehat\mu(\eta)=e^{-\abs{\eta}}$ is not differentiable at the origin, but $\widehat\mu(\xi)=1+O(\abs{\xi})$, so $\xi\mapsto(\widehat E(s,k)\cdot\xi)\widehat\mu(\xi)$ is differentiable at $\xi=0$ with gradient $\widehat\mu(0)\widehat E(s,k)$. The identity therefore persists at $\xi_s=0$.

Evaluate at $\eta=0$, so that $\xi_s=k(t-s)$, which vanishes only at the endpoint $s=t$. As $\widehat\mu$ is radial, $\nabla_\eta\widehat\mu(k(t-s))$ is parallel to $k$ for $s<t$ and the second term is annihilated by $\mathbb P_k$, while in the first $\mathbb P_k\widehat E(s,k)=\widehat E_T(s,k)=-\ep\p_s\widehat A(s,k)$. Hence
\begin{align*}
    \mathbb P_k\widehat{j[f]}(t,k)
    =\mathbb P_k\widehat S^{\mathrm{cur}}(t,k)-\ep\int_0^t\widehat\mu(\ak(t-s))\p_s\widehat A(s,k)\,ds,
\end{align*}
and substituting into the forced wave equation gives the closed transverse Volterra equation with memory
\begin{align}\label{eq:wave-closed}
\br{\ep^2\p_t^2+\ak^2}\widehat A(t,k)
=\ep\mathbb P_k\widehat S^{\mathrm{cur}}(t,k)-\ep^2\int_0^t\widehat\mu(\ak(t-s))\p_s\widehat A(s,k)\,ds.
\end{align}

For each $k\neq0$, the source $\widehat S^{\mathrm{cur}}(t,k)=i\nabla_\eta\widehat f_{\mathrm{in}}(k,kt)$ is bounded and continuous by the assumption $\jb{v}\fin\in L^1_{x,v}$. Standard Volterra theory gives a unique global solution of~\eqref{eq:wave-closed} with initial data $(\widehat A^0(k),\widehat A^1(k))$, and Gr\"onwall's inequality shows that $\widehat A(\cdot,k)$ and $\p_t\widehat A(\cdot,k)$ are of exponential order. Equation~\eqref{eq:wave-closed} may therefore be transformed in time for $\Re\lam$ sufficiently large. There the second derivative contributes the initial data through $\int_0^\infty e^{-\lam t}\p_t^2\widehat A\,dt=\lam^2\widetilde A-\lam\widehat A^0-\widehat A^1$, while the convolution becomes the product of
\begin{align*}
    \int_0^\infty e^{-\lam t}\widehat\mu(\ak t)\,dt=\frac{1}{\lam+\ak},
    \qquad
    \int_0^\infty e^{-\lam t}\p_t\widehat A(t,k)\,dt=\lam\widetilde A(\lam,k)-\widehat A^0(k),
\end{align*}
the first by $\widehat\mu(\eta)=e^{-\abs{\eta}}$ from~\eqref{Poisson}. Therefore
\begin{align*}
    \br{\ep^2\lam^2+\ak^2}\widetilde A(\lam,k)-\ep^2\br{\lam\widehat A^0(k)+\widehat A^1(k)}
    =\ep\mathbb P_k\widetilde S^{\mathrm{cur}}(\lam,k)-\frac{\ep^2}{\lam+\ak}\br{\lam\widetilde A(\lam,k)-\widehat A^0(k)},
\end{align*}
which rearranges to~\eqref{eq:wave-final} with $D$ as in~\eqref{eq:wave-dispersion-function}.
\end{proof}

\subsection{Roots of the transverse dispersion relation}
As $D$ is rational in $\lam$, defining
\begin{align}\label{P}
    P(\lam,\ak) := (\lam+\ak)D(\lam,\ak)
    = (\lam+\ak)\br{\ep^2\lam^2+\ak^2}+\ep^2\lam,
\end{align}
gives a cubic in $\lam$ with real coefficients. The value $P(-\ak,\ak)=-\ep^2\ak\neq0$ shows that the factor $\lam+\ak$ contributes no zero, so the roots of $P(\cdot,\ak)$ coincide with those of $D(\cdot,\ak)$. Differentiating in $\lam$, and writing $r=\ak$,
\begin{align}\label{eq:P-derivative}
    \p_\lam P(\lam,r)= 3\ep^2\br{\lam+\frac{r}{3}}^2 + r^2\br{1-\frac{\ep^2}{3}}+\ep^2>0
\end{align}
for all $r>0$ and $\ep\in(0,1]$, so $P(\cdot,r)$ is strictly increasing on $\R$. Since $P(\lam,r)\to\pm\infty$ as $\lam\to\pm\infty$, strict monotonicity gives exactly one real root $\lam_0(r)$, necessarily simple, since $\p_\lam P>0$ excludes a repeated root. The remaining two roots are non-real and, the coefficients of $P(\cdot,r)$ being real, occur as a complex conjugate pair $\lam_\pm(r)=\gamma(r)\pm i\omega(r)$ with $\omega(r)>0$. Thus $P(\cdot,r)$ has three distinct roots, all of which are simple, and the implicit function theorem yields their smooth dependence on $r$. These roots depend on $\ep$ through $D$, although this dependence is suppressed in the notation.

\begin{lemma}[Roots of the dispersion relation]\label{lem:real-root}
For each $r>0$, the dispersion relation $D(\cdot,r)=0$ has a unique real root $\lam_0(r)$ and a complex conjugate pair
\begin{align*}
    \lam_\pm(r)=\gamma(r)\pm i\omega(r),
\end{align*}
with
\begin{align*}
    \lam_0\in C^\infty\br{(0,\infty),\R},
    \qquad
    \lam_\pm\in C^\infty\br{(0,\infty),\C}.
\end{align*}
For $\rho>0$, set
\begin{align*}
    \theta(\rho):=\frac{1}{\ep}\br{\lam_0(\ep\rho)+\ep\rho},
    \qquad
    \zeta(\rho):=\frac{1}{\ep}\gamma(\ep\rho),
    \qquad
    \beta(\rho):=\omega(\ep\rho).
\end{align*}
There exists $\ep_0\in(0,1]$ such that, for every integer $n\geq0$, the bounds below hold for all $\ep\in(0,\ep_0]$ and $\rho>0$, with constants $C_n>0$ independent of $\ep$ and $\rho$.
\begin{enumerate}[label=(\roman*)]
\item\label{realroot} The rescaled correction $\theta$ for the real root satisfies
\begin{align}\label{realrootupperandlowerbound}
    \frac{\rho}{2\jb{\rho}^2}
    \leq\theta(\rho)
    \leq\frac{\rho}{\jb{\rho}^2},\qquad\abs{\theta^{(n)}(\rho)}
    \leq\frac{C_n}{\jb{\rho}^{n+1}}.
\end{align}

\item\label{complexroot-zeta} The rescaled real part $\zeta$ of the complex pair satisfies
\begin{align*}
    -\frac{\rho}{2\jb{\rho}^2}
    \leq\zeta(\rho)
    \leq-\frac{\rho}{4\jb{\rho}^2},\qquad\abs{\zeta^{(n)}(\rho)}
    \leq\frac{C_n}{\jb{\rho}^{n+1}}.
\end{align*}

\item\label{complexroot-beta} The rescaled oscillation frequency $\beta$ of the complex pair satisfies, for some absolute constants $c_\beta,C_\beta>0$,
\begin{align*}
    c_\beta\jb{\rho}\leq\beta(\rho)\leq C_\beta\jb{\rho},\quad c_\beta\frac{\rho}{\jb{\rho}}\leq\beta'(\rho)\leq C_\beta\frac{\rho}{\jb{\rho}},\quad \frac{c_\beta}{\jb{\rho}^3}\leq\beta''(\rho)\leq\frac{C_\beta}{\jb{\rho}^3},
\end{align*}and
\begin{align*}
    \abs{\beta^{(n)}(\rho)}
    \leq\frac{C_n}{\jb{\rho}^{n-1}}.
\end{align*}
Moreover, the radial function $k\mapsto\beta(\ak)$ is smooth on $\R^3$.
\end{enumerate}
\end{lemma}

The proof of Lemma~\ref{lem:real-root} is deferred to Appendix~\ref{appendixroots}. Translating its estimates from the rescaled variable $\rho=r/\ep$ to the original frequency variable $r$ gives the following corollary.

\begin{corollary}[Root bounds in the original variables]\label{cor:roots}
There exists $\ep_0\in(0,1]$ such that, for every integer $n\geq0$, the bounds below hold for all $\ep\in(0,\ep_0]$ and $r>0$, with constants $C_n>0$ independent of $\ep$ and $r$.
\begin{enumerate}[label=(\roman*)]
\item\label{realrootitem} The real root satisfies $-r<\lam_0(r)<0$ and
\begin{align*}
\lam_0(r)\leq-\frac{1}{2}\min\set{\frac{r^3}{\ep^2},r}.
\end{align*}

\item\label{realpartitem} The real and imaginary parts of the complex pair $\lam_\pm=\gamma\pm i\omega$ satisfy
\begin{align*}
    -\frac{r}{2\jb{\frac{r}{\ep}}^2}
    \leq\gamma(r)
    \leq-\frac{r}{4\jb{\frac{r}{\ep}}^2},
    \qquad
    r^n\abs{\p_r^n\gamma(r)}\leq C_n\frac{r}{\jb{\frac{r}{\ep}}^2},\qquad\abs{\omega(r)}\gtrsim\jb{\frac{r}{\ep}}.
\end{align*}

\item\label{imaginarypartitem} The complex roots and their reciprocals satisfy
\begin{align*}
r^n\abs{\p_r^n\lam_\pm(r)}
\leq C_n\jb{\frac{r}{\ep}},
\qquad
r^n\abs{\p_r^n\sbr{\frac{1}{\lam_\pm(r)}}}
\leq\frac{C_n}{\jb{\frac{r}{\ep}}}.
\end{align*}
\end{enumerate}
\end{corollary}
\begin{proof}
By the definitions of Lemma~\ref{lem:real-root}, the real root satisfies $\lam_0(r)=-r+\ep\theta(r/\ep)$, so~\eqref{realrootupperandlowerbound} gives $0<\ep\theta(r/\ep)\leq r\jb{r/\ep}^{-2}<r$ and hence $-r<\lam_0(r)<0$, together with
\begin{align*}
    \lam_0(r)\leq -r+\frac{r}{\jb{\frac{r}{\ep}}^2}
    =-r\frac{\br{\frac{r}{\ep}}^2}{1+\br{\frac{r}{\ep}}^2}.
\end{align*}
The elementary inequality $x^2/(1+x^2)\geq\min\set{x^2,1}/2$ at $x=r/\ep$ then yields the two cases of~\ref{realrootitem}. The remaining bounds are rescalings of Lemma~\ref{lem:real-root}\ref{complexroot-zeta} and~\ref{complexroot-beta} through $\gamma(r)=\ep\zeta(r/\ep)$ and $\omega(r)=\beta(r/\ep)$. For $n=0$ they follow directly from the pointwise bounds there. For $n\geq1$, the chain rule gives $\p_r^n\gamma(r)=\ep^{1-n}\zeta^{(n)}(r/\ep)$ and $\p_r^n\omega(r)=\ep^{-n}\beta^{(n)}(r/\ep)$, so $r^n\abs{\p_r^n\gamma(r)}\lesssim r(r/\ep)^{n-1}\jb{r/\ep}^{-n-1}\leq r\jb{r/\ep}^{-2}$ and $r^n\abs{\p_r^n\omega(r)}\lesssim(r/\ep)^{n}\jb{r/\ep}^{1-n}\leq\jb{r/\ep}$. The first estimate in~\ref{imaginarypartitem} follows from $\lam_\pm=\gamma\pm i\omega$ and $r\leq\jb{r/\ep}$. Since $\abs{\lam_\pm(r)}\geq\abs{\omega(r)}\gtrsim\jb{r/\ep}$, the reciprocal estimate holds for $n=0$. The cases $n\geq1$ follow inductively by differentiating $\lam_\pm\lam_\pm^{-1}=1$ and using the first estimate in~\ref{imaginarypartitem}.
\end{proof}

\section{Transverse resolvent kernels and field decomposition}\label{resolventSection}
The Fourier--Laplace equation for the Coulomb-gauge vector potential is now inverted. The resulting three-mode representation of $A$ splits the transverse fields $B=\nabla_x\times A$ and $E_T=-\ep\p_tA$ into an oscillatory part carried by the complex pair and a non-oscillatory remainder.

\subsection{The kernels and the solution representation}\label{sec:representation}

To invert~\eqref{eq:wave-final}, introduce the resolvent kernel
\begin{align*}
    \widetilde H(\lam,\ak):=\frac{1}{D(\lam,\ak)},
\end{align*}
which is meromorphic on $\C$, with simple poles at the three roots $\lam_0,\lam_\pm$ of $D(\cdot,\ak)$. Since
\begin{align*}
    P(\lam,\ak)=(\lam+\ak)D(\lam,\ak)=\ep^2(\lam-\lam_0)(\lam-\lam_+)(\lam-\lam_-),
\end{align*}
partial fractions give
\begin{align}\label{defnR}
    \widetilde H(\lam,\ak)=\frac{\lam+\ak}{P(\lam,\ak)}
    =\sum_{j\in\set{0,\pm}}\frac{R_j(\ak)}{\lam-\lam_j(\ak)},
    \qquad
    R_j(\ak):=\frac{\lam_j(\ak)+\ak}{\p_\lam P(\lam_j(\ak),\ak)}.
\end{align}
Since $\widetilde H(\cdot,\ak)$ is rational in $\lam$, decays at least like $\abs{\lam}^{-2}$ as $\abs{\lam}\to\infty$ and has all its poles in $\set{\Re\lam<0}$ by Corollary~\ref{cor:roots}\ref{realrootitem} and~\ref{realpartitem}, its Bromwich integral may be closed to the left and evaluated by residues, which yields the time-domain kernel
\begin{align*}
    H(t,\ak):=\sum_{j\in\set{0,\pm}}R_j(\ak)e^{\lam_j(\ak)t}.
\end{align*}

Dividing~\eqref{eq:wave-final} by $D$ shows that $\widetilde H$ carries the free-transport current $S^{\mathrm{cur}}$ and the field data $A^0$, $A^1$. The equilibrium boundary term generated by $A^0$ enters with the additional factor $(\lam+\ak)^{-1}$ and is therefore carried by the second kernel
\begin{align}\label{eq:Hmu}
    \widetilde H^{\mu}(\lam,\ak):=\frac{\widetilde H(\lam,\ak)}{\lam+\ak}=\frac{1}{P(\lam,\ak)}.
\end{align}
Since $\widetilde H(\lam,\ak)=(\lam+\ak)/P(\lam,\ak)$ and $P(-\ak,\ak)=-\ep^2\ak\neq0$, the pole of the factor $(\lam+\ak)^{-1}$ is cancelled by the simple zero of $\widetilde H$ at $\lam=-\ak$, so $\widetilde H^{\mu}$ has poles only at the three roots of $D$. Partial fractions and the residue theorem give the time-domain kernel
\begin{align}\label{eq:Hmu-time}
    H^{\mu}(t,\ak):=\sum_{j\in\set{0,\pm}}R^{\mu}_j(\ak)e^{\lam_j(\ak)t},
    \qquad
    R^{\mu}_j(\ak):=\frac{1}{\p_\lam P(\lam_j(\ak),\ak)}.
\end{align}

\begin{lemma}[Solution representation]\label{lem:solution-representation}
Let $A$ be as in Lemma~\ref{lem:transverse-wave-equation}. Then, for $k\neq0$ and $t\geq0$,
\begin{align}\label{eq:solution-representation}
    \widehat A(t,k)
    =\ep H(t,\ak)*_t\mathbb P_k\widehat S^{\mathrm{cur}}(t,k)
    +\ep^2\br{\p_tH(t,\ak)\widehat A^0(k)
    +H(t,\ak)\widehat A^1(k)
    +H^{\mu}(t,\ak)\widehat A^0(k)},
\end{align}
where $*_t$ denotes time convolution over $[0,t]$.
\end{lemma}

\begin{proof}
Dividing~\eqref{eq:wave-final} by $D$ and using~\eqref{eq:Hmu} gives
\begin{align*}
    \widetilde A(\lam,k)=\widetilde H(\lam,\ak)
    \sbr{\ep\mathbb P_k\widetilde S^{\mathrm{cur}}(\lam,k)
    +\ep^2\br{\lam\widehat A^0(k)+\widehat A^1(k)}}+\ep^2\widetilde H^{\mu}(\lam,\ak)\widehat A^0(k).
\end{align*}
The only inverse transform requiring justification is that of $\lam\widetilde H(\lam,\ak)$. From the partial-fraction expansion, as $\abs{\lam}\to\infty$,
\begin{align*}
    \widetilde H(\lam,\ak)
    =\frac{1}{\lam}\sum_{j\in\set{0,\pm}}R_j(\ak)
    +O\br{\abs{\lam}^{-2}}.
\end{align*}
Since $P(\lam,\ak)$ is cubic in $\lam$, the definition of $\widetilde H$ also gives
\begin{align*}
    \widetilde H(\lam,\ak)
    =\frac{\lam+\ak}{P(\lam,\ak)}
    =O\br{\abs{\lam}^{-2}}.
\end{align*}
Comparison yields $\sum_jR_j(\ak)=0$, so the time-domain representation gives $H(0,\ak)=0$. Therefore
\begin{align*}
    \mathcal L\sbr{\p_tH}(\lam,\ak)
    =\lam\widetilde H(\lam,\ak)-H(0,\ak)
    =\lam\widetilde H(\lam,\ak).
\end{align*}
The source term $\ep\widetilde H \mathbb P_k\widetilde S^{\mathrm{cur}}$ inverts by the convolution theorem, the remaining products invert term by term, and inverting the transformed equation gives~\eqref{eq:solution-representation}.
\end{proof}

For each $j\in\set{0,\pm}$, write
\begin{align*}
H_j(t,\ak):=R_j(\ak)e^{\lam_j(\ak)t},
\qquad
H^{\mu}_j(t,\ak):=R^{\mu}_j(\ak)e^{\lam_j(\ak)t},
\end{align*}
so that $H=\sum_jH_j$ and $H^{\mu}=\sum_jH^{\mu}_j$, with $\p_tH_j=\lam_jH_j$ and hence $\p_tH=\sum_j\lam_jH_j$.

\subsection{Residue bounds}
The following estimates give the uniform bound on the residues.

\begin{lemma}[Residue $R_0$]\label{lem:R0-bounds}
There exists $\ep_0\in(0,1]$ such that, for all $\ep\in(0,\ep_0]$ and $r>0$, the residue $R_0$ satisfies
\begin{align}\label{eq:R0-global-bound}
\abs{R_0(r)}
\lesssim\frac{r}{\ep^2\jb{\frac{r}{\ep}}^4}
\lesssim\min\set{\frac{r}{\ep^2},\frac{\ep^2}{r^3}}.
\end{align}
\end{lemma}

\begin{proof}
In the scaled variable $\rho:=r/\ep$, substituting $\lam_0=\ep(\theta(\rho)-\rho)$ and $r=\ep\rho$ gives
\begin{align*}
    \p_\lam P(\lam_0,r)&=3\ep^2\lam_0^2+2\ep^2r\lam_0+r^2+\ep^2=\ep^2 b(\rho),\\
    b(\rho)&:=1+\rho^2+\ep^2\br{\rho^2-4\rho\theta+3\theta^2},
\end{align*}
so that, by~\eqref{defnR},
\begin{align}\label{eq:R0-scaled}
    R_0(r)=\frac{\ep\theta(\rho)}{\ep^2 b(\rho)}=\frac{1}{\ep}\frac{\theta(\rho)}{b(\rho)}.
\end{align}
Completing the square,
\begin{align*}
    \rho^2-4\rho\theta+3\theta^2=3\br{\theta-\frac23\rho}^2-\frac13\rho^2\geq-\frac13\rho^2,
\end{align*}
so that, since $\ep\leq1$,
\begin{align}\label{eq:b-lower}
    b(\rho)\geq1+\rho^2-\frac{\ep^2}{3}\rho^2\geq\frac23(1+\rho^2)=\frac23\jb{\rho}^2.
\end{align}
With $\abs{\theta(\rho)}\leq\rho/\jb{\rho}^2$ from Lemma~\ref{lem:real-root}\ref{realroot}, \eqref{eq:R0-scaled} and \eqref{eq:b-lower} give
\begin{align*}
    \abs{R_0(r)}
    \lesssim\frac{1}{\ep}\frac{\frac{\rho}{\jb{\rho}^2}}{\jb{\rho}^2}
    =\frac{\rho}{\ep\jb{\rho}^4}
    =\frac{r}{\ep^2\jb{\frac{r}{\ep}}^4},
\end{align*}
the first bound in \eqref{eq:R0-global-bound}. The two cases follow from $\jb{r/\ep}\geq\max\set{r/\ep,1}$.
\end{proof}

In the scaled variable $\rho=r/\ep$, the derivative bounds at the complex pair rest on the root differences
\begin{align}\label{eq:pq}
    p_\pm(\rho):=\lam_\pm(\ep\rho)+\ep\rho,
    \qquad
    q_\pm(\rho):=\lam_\pm(\ep\rho)-\lam_0(\ep\rho),
    \qquad
    \lam_\pm(\ep\rho)-\lam_\mp(\ep\rho)=\pm2i\beta(\rho),
\end{align}
with $\beta$ as in Lemma~\ref{lem:real-root}. The identity and $q_-(\rho)=\overline{q_+(\rho)}$ follow from $\lam_\pm=\gamma\pm i\omega$ with $\gamma$ and $\lam_0$ real. The bounds below serve again for the equilibrium residues in Lemma~\ref{lem:Rmu-bounds}.

\begin{lemma}[Bounds on the root differences]\label{lem:factorisations}
There exists $\ep_0\in(0,1]$ such that, for all $\ep\in(0,\ep_0]$ and $\rho>0$, the root differences satisfy
\begin{align}\label{eq:q-beta-lower}
    \abs{q_\pm(\rho)}\gtrsim\jb{\rho},
    \qquad
    \abs{\beta(\rho)}\gtrsim\jb{\rho}.
\end{align}
Moreover, for every pair of integers $m,n\geq0$,
\begin{align}
    \abs{p_\pm^{(m)}(\rho)}+\abs{q_\pm^{(m)}(\rho)}
    &\lesssim\jb{\rho}^{1-m},\label{eq:Rpm-factor-derivs}\\
    \abs{\br{\frac{1}{q_\pm}}^{(n)}(\rho)}+\abs{\br{\frac{1}{\beta}}^{(n)}(\rho)}
    &\lesssim\jb{\rho}^{-n-1}.\label{eq:reciprocal-derivs}
\end{align}
The implicit constants are independent of $\ep$ and $\rho$.
\end{lemma}

The proof, deferred to Appendix~\ref{appendixroots}, expresses $p_\pm$ and $q_\pm$ through $\theta$ and $\beta$ by Vieta's formulas.

\begin{lemma}[Residues $R_\pm$]\label{lem:Rpm-bounds}
There exists $\ep_0\in(0,1]$ such that, for all $\ep\in(0,\ep_0]$ and $r>0$,
\begin{align*}
    r^n\abs{\p_r^nR_\pm(r)}\lesssim\frac{1}{\ep^2\jb{\frac{r}{\ep}}},
    \quad n\geq0.
\end{align*}
\end{lemma}

\begin{proof}
Since $P(\lam,r)=\ep^2(\lam-\lam_0)(\lam-\lam_+)(\lam-\lam_-)$, the identities~\eqref{eq:pq} give, at $r=\ep\rho$,
\begin{align*}
    \p_\lam P(\lam_\pm,r)
    =\ep^2(\lam_\pm-\lam_0)(\lam_\pm-\lam_\mp)
    =\pm2i\ep^2 \beta(\rho) q_\pm(\rho),
    \qquad
    \lam_\pm+r=p_\pm(\rho),
\end{align*}
so that
\begin{align*}
    R_\pm(r)=\frac{\mp i}{2\ep^2} p_\pm(\rho)\frac{1}{q_\pm(\rho)}\frac{1}{\beta(\rho)}.
\end{align*}
The three factors on the right each lose one power of $\jb{\rho}$ per derivative, by~\eqref{eq:Rpm-factor-derivs} and~\eqref{eq:reciprocal-derivs}. By the Leibniz rule, the $n$-th $\rho$-derivative of the product $p_\pm(1/q_\pm)(1/\beta)$ is a sum of terms $p_\pm^{(i)}(1/q_\pm)^{(j)}(1/\beta)^{(k)}$ with $i+j+k=n$, each of size $\jb{\rho}^{1-i}\jb{\rho}^{-1-j}\jb{\rho}^{-1-k}=\jb{\rho}^{-n-1}$ up to constants independent of $\ep$ and $\rho$. Since each derivative $\p_r$ of a function of $\rho=r/\ep$ produces a factor $\ep^{-1}$, and $r^n\ep^{-n}=\rho^n\leq\jb{\rho}^n$, the claimed bound follows.
\end{proof}

The same root-difference bounds give sharper decay for the equilibrium residues $R_j^{\mu}$ of the kernel $H^\mu$, defined in~\eqref{eq:Hmu-time}, since these residues contain no factor $\lam_j+r$ in the numerator.

\begin{lemma}[Equilibrium residues]\label{lem:Rmu-bounds}
There exists $\ep_0\in(0,1]$ such that, for all $\ep\in(0,\ep_0]$ and $r>0$,
\begin{align*}
    \abs{R^{\mu}_0(r)}\lesssim\frac{1}{\ep^2\jb{\frac{r}{\ep}}^2},
    \qquad
    r^{n}\abs{\p_r^{n}R^{\mu}_\pm(r)}\lesssim\frac{1}{\ep^2\jb{\frac{r}{\ep}}^2},
    \quad n\geq0.
\end{align*}
\end{lemma}

\begin{proof}
Set $\rho:=r/\ep$. As $\p_\lam P(\lam_j)=\ep^2\prod_{j'\neq j}(\lam_j-\lam_{j'})$, the identities~\eqref{eq:pq} give
\begin{align*}
    R^{\mu}_0=\frac{1}{\p_\lam P(\lam_0)}=\frac{1}{\ep^2 q_+q_-}=\frac{1}{\ep^2\abs{q_\pm}^2},
    \qquad
    R^{\mu}_\pm=\frac{1}{\p_\lam P(\lam_\pm)}=\frac{1}{\ep^2 q_\pm(\pm2i\beta)}
    =\frac{\mp i}{2\ep^2 \beta q_\pm}.
\end{align*}
The lower bounds~\eqref{eq:q-beta-lower} give $\abs{R^{\mu}_0}\lesssim\ep^{-2}\jb{\rho}^{-2}$ and $\abs{R^{\mu}_\pm}\lesssim\ep^{-2}\jb{\rho}^{-2}$. The complex-pair derivative bounds follow as in Lemma~\ref{lem:Rpm-bounds}, with the factor $p_\pm$ absent.
\end{proof}

\subsection{Decomposition of the transverse field}\label{sec:decomposition}
The transverse electric and magnetic fields are decomposed into oscillatory and remainder parts. The convolutions of the complex-root kernels with the free-transport current are integrated by parts three times in time. Each integration supplies a factor $\lam_\pm^{-1}$ in the amplitude and a time derivative on the current, and both the resulting frequency decay of the amplitudes and the time decay of $\p_t^3S^{\mathrm{cur}}$ are needed for the convolution estimate in Section~\ref{complexrootSection}.

Throughout, the roots $\lam_*$ and the residues $R_*$, $R^{\mu}_*$, $*\in\set{0,\pm}$, act on functions of $x$ as Fourier multipliers with symbols $\lam_*(\ak)$, $R_*(\ak)$ and $R^{\mu}_*(\ak)$, as does $e^{\lam_*t}$ with symbol $e^{\lam_*(\ak)t}$.

\begin{proposition}[Oscillatory--remainder decomposition] 
\label{prop:transverse-decomposition}
Assume $\jb{v}^4\fin\in L^1_{x,v}$. Then the transverse fields split into oscillatory and remainder parts,
\begin{align}\label{eq:field-split}
    E_T=\ETosc+E_T^{\mathrm r},
    \qquad
    B=\Bosc+B^{\mathrm r},
\end{align}
where $\ETosc:=E^{\mathrm{osc}}_{T,+}+E^{\mathrm{osc}}_{T,-}$ and $\Bosc:=B^{\mathrm{osc}}_++B^{\mathrm{osc}}_-$, and these fields admit the representations
\begin{align}\label{eq:osc-explicit}
    \begin{split}E^{\mathrm{osc}}_{T,\pm}&:=\br{\tau^0_\pm +\tau^\mu_\pm}e^{\lam_\pm t}A^0+\tau^1_\pm e^{\lam_\pm t} A^1 +\sum_{j=0}^2\tau^{\mathrm{cur},j+1}_\pm e^{\lam_\pm t}\p_t^jS^{\mathrm{cur}}(0)+\tau^{\mathrm{cur},3}_\pm e^{\lam_\pm t}*_t\p_t^3 S^{\mathrm{cur}}(t),\\
    B^{\mathrm{osc}}_{\pm}&:=\br{\upsilon^0_\pm +\upsilon^\mu_\pm}e^{\lam_\pm t}A^0+\upsilon^1_\pm e^{\lam_\pm t} A^1 +\sum_{j=0}^2\upsilon^{\mathrm{cur},j+1}_\pm e^{\lam_\pm t}\p_t^jS^{\mathrm{cur}}(0)+\upsilon^{\mathrm{cur},3}_\pm e^{\lam_\pm t}*_t\p_t^3 S^{\mathrm{cur}}(t),
    \end{split}
\end{align}
and
\begin{align}\label{eq:rem-explicit}\begin{split}
E_T^{\mathrm r}&:=\br{\tau_0^0+\tau_0^\mu}e^{\lam_0t}A^0+\tau_0^1e^{\lam_0t}A^1 +\tau^{\mathrm{cur},1}_0\br{e^{\lam_0t}S^{\mathrm{cur}}(0)+e^{\lam_0t}*_t \p_tS^{\mathrm{cur}}(t)}\\
&\qquad-\sum_\pm\sum_{j=1}^2\tau_\pm^{\mathrm{cur},j+1}\p_t^{j}S^{\mathrm{cur}}(t),\\
    B^{\mathrm r}&:=\br{\upsilon_0^0+\upsilon_0^\mu} e^{\lam_0t}A^0+\upsilon_0^1e^{\lam_0t}A^1+\upsilon_0^{\mathrm{cur},0}e^{\lam_0t}*_t S^{\mathrm{cur}}(t)-\sum_\pm\sum_{j=0}^2\upsilon_\pm^{\mathrm{cur},j+1}\p_t^{j}S^{\mathrm{cur}}(t).
    \end{split}
\end{align}
For the indices occurring above and $*\in\set{0,\pm}$, the amplitudes are given by
\begin{align*}
    \tau^0_*=-\ep^3\lam_*^2R_*,\quad \tau^1_*=-\ep^3\lam_*R_*,\quad \tau^\mu_*=-\ep^3\lam_*R^\mu_*,\quad \tau^{\mathrm{cur},j}_*=-\ep^2\lam_*^{1-j}R_*\mathbb P,
\end{align*}
\begin{align*}
    \upsilon_*^0=\ep^2\lam_*R_*\nabla_x\times,\;\; \upsilon_*^1=\ep^2 R_*\nabla_x\times,\;\; \upsilon^\mu_*=\ep^2R^\mu_*\nabla_x\times,\;\; \upsilon^{\mathrm{cur},j}_*=\ep\lam_*^{-j}R_*\nabla_x\times.
\end{align*}
The potential itself splits as $A=A^{\mathrm{osc}}_++A^{\mathrm{osc}}_-+A^{\mathrm r}$, where $A^{\mathrm{osc}}_\pm$ and $A^{\mathrm r}$ are defined by the formulas for $B^{\mathrm{osc}}_\pm$ and $B^{\mathrm r}$ with $\nabla_x\times$ replaced by $\mathbb P$ in the amplitudes $\upsilon$, so that $B^{\mathrm{osc}}_\pm=\nabla_x\times A^{\mathrm{osc}}_\pm$ and $B^{\mathrm r}=\nabla_x\times A^{\mathrm r}$.
\end{proposition}
\begin{proof}
    Substituting the modal kernels into~\eqref{eq:solution-representation} splits the potential into one contribution per root, $\widehat A=\widehat A_0+\widehat A_++\widehat A_-$, where, for $j\in\set{0,\pm}$,
\begin{align}\label{eq:mode-split}
    \widehat A_j(t,k)
    :=\ep H_j(t,\ak)*_t \mathbb P_k\widehat S^{\mathrm{cur}}(t,k)
    +\ep^2\lam_j H_j \widehat A^0
    +\ep^2H_j \widehat A^1
    +\ep^2H^{\mu}_j \widehat A^0.
\end{align}
The transverse fields are recovered through $E_T=-\ep\p_tA$ and $B=\nabla_x\times A$, that is $\widehat E_T=-\ep\p_t\widehat A$ and $\widehat B=ik\times\widehat A$. Since $H_j=R_j e^{\lam_j t}$ and $H^{\mu}_j=R^{\mu}_j e^{\lam_j t}$, the three data terms of~\eqref{eq:mode-split} give the terms of~\eqref{eq:osc-explicit} and~\eqref{eq:rem-explicit} with the amplitudes $\tau^0_j,\tau^1_j,\tau^\mu_j$ and $\upsilon^0_j,\upsilon^1_j,\upsilon^\mu_j$.

For the source convolution, let $h(\tau)=Re^{\lam\tau}$ with $\lam\in\C\setminus\set{0}$, so that $h(t-s)=-\lam^{-1}\p_s\sbr{h(t-s)}$. Repeated integration by parts gives, for $g\in C^m([0,t])$ and every integer $m\geq1$,
\begin{align}\label{eq:osc-ibp}
    \int_0^th(t-s)g(s)\,ds
    =\sum_{j=0}^{m-1}\frac{1}{\lam^{j+1}}\sbr{h(t) g^{(j)}(0)-h(0) g^{(j)}(t)}
    +\frac{1}{\lam^m}\int_0^th(t-s) g^{(m)}(s)\,ds.
\end{align}
For the complex pair this is applied with $\lam=\lam_\pm(\ak)$, $h=H_\pm(\cdot,\ak)$, $g=\mathbb P_k\widehat S^{\mathrm{cur}}(\cdot,k)$ and $m=3$. The assumption $\jb{v}^4\fin\in L^1_{x,v}$ ensures that $\widehat S^{\mathrm{cur}}(\cdot,k)\in C^3([0,\infty))$ for each fixed $k$, justifying these integrations by parts. Since $H_\pm(0,\ak)=R_\pm(\ak)$, the source term of $\widehat A_\pm$ becomes
\begin{align*}
    \ep\int_0^t&H_\pm(t-s,\ak)\mathbb P_k\widehat S^{\mathrm{cur}}(s,k)\,ds
    \\
    &=\sum_{j=0}^{2}\frac{\ep R_\pm}{\lam_\pm^{j+1}}\mathbb P_k\sbr{e^{\lam_\pm t}\p_t^{j}\widehat S^{\mathrm{cur}}(0,k)-\p_t^{j}\widehat S^{\mathrm{cur}}(t,k)}+\frac{\ep R_\pm}{\lam_\pm^{3}}\,e^{\lam_\pm t}*_t\mathbb P_k\p_t^{3}\widehat S^{\mathrm{cur}}(t,k).
\end{align*}
Applying $ik\times$, with $k\times\mathbb P_k=k\times$, gives the terms of~\eqref{eq:osc-explicit} with amplitudes $\upsilon^{\mathrm{cur},j+1}_\pm$ and $\upsilon^{\mathrm{cur},3}_\pm$ and the terms of~\eqref{eq:rem-explicit} with amplitudes $-\upsilon^{\mathrm{cur},j+1}_\pm$. Applying $-\ep\p_t$, the terms at $s=0$ acquire the factor $\lam_\pm$ and give the amplitudes $\tau^{\mathrm{cur},j+1}_\pm$, while differentiating the upper limit of the convolution contributes
\begin{align}\label{cancellation}
    -\frac{\ep^2R_\pm}{\lam_\pm^{3}}\mathbb P_k\p_t^3\widehat S^{\mathrm{cur}}(t,k),
\end{align}
which the derivative of the term $j=2$ at $s=t$ cancels exactly. The remaining convolution carries the amplitude $\tau^{\mathrm{cur},3}_\pm$, and the terms at $s=t$ give the amplitudes $-\tau^{\mathrm{cur},j+1}_\pm$ on $\p_t^{j}\widehat S^{\mathrm{cur}}(t,k)$ for $j=1,2$.

For the real root, the magnetic and transverse electric fields are recovered directly from $\widehat A_0$ by applying $ik\times$ and $-\ep\p_t$, respectively. The magnetic source term therefore carries the amplitude $\upsilon^{\mathrm{cur},0}_0$. For the electric field, the time derivative in $-\ep\p_t\widehat A_0$ acts on the source convolution through the identity
\begin{align*}
    \p_t\br{e^{\lam_0t}*_tF(t)}
    =e^{\lam_0t}F(0)+e^{\lam_0t}*_t\p_tF(t),
\end{align*}
which, applied with $F=\mathbb P_k\widehat S^{\mathrm{cur}}(\cdot,k)$, gives the two terms with amplitude $\tau^{\mathrm{cur},1}_0$. Collecting the terms carrying a complex-root kernel into $E^{\mathrm{osc}}_{T,\pm}$ and $B^{\mathrm{osc}}_\pm$ and the others into $E_T^{\mathrm r}$ and $B^{\mathrm r}$ gives~\eqref{eq:field-split}.

Collecting the terms of $\widehat A$ in the same way, before $ik\times$ is applied, gives the decomposition of $A$ with the stated amplitudes, and $k\times\mathbb P_k=k\times$, $\mathbb P_k\widehat A^0=\widehat A^0$ and $\mathbb P_k\widehat A^1=\widehat A^1$ give $B^{\mathrm{osc}}_\pm=\nabla_x\times A^{\mathrm{osc}}_\pm$ and $B^{\mathrm r}=\nabla_x\times A^{\mathrm r}$.
\end{proof}

\section{Estimates for the free-transport current}\label{sectionSf}
This section establishes decay estimates and initial-time low-frequency bounds for the free-transport current~\eqref{Scurdef}, in terms of weighted norms of~$\fin$.

The following elementary facts are used repeatedly. Differentiation under the integral gives, for every multi-index $\alpha$ and integer $n\geq0$,
\begin{align}\label{eq:Sf-derivatives}
    \nabla_x^\alpha\p_t^nS^{\mathrm{cur}}(t,x)
    =(-1)^n\sum_{\abs{\gamma}=n}c_\gamma
    \int_{\R^3}v v^\gamma\br{\nabla_x^{\alpha+\gamma}\fin}(x-tv,v)\,dv,
\end{align}
where $c_\gamma$ are the multinomial coefficients. Moreover, for every integrable $g$, Fubini's theorem and the translation invariance of $L^1_x$ give
\begin{align}\label{eq:transport-L1}
    \norm{\int_{\R^3}g(x-tv,v)\,dv}_{L^1_x}
    \leq\norm{g}_{L^1_{x,v}}.
\end{align}

\begin{lemma}[Decay for free-transport velocity averages]\label{lem:free-transport-velocity-average-L1}
For every multi-index $\alpha$ and every integer $n\geq0$, uniformly for $t\geq0$,
\begin{align*}
    \norm{\nabla_x^\alpha\p_t^nS^{\mathrm{cur}}(t)}_{L^1_x}
    \lesssim_{\alpha,n}
    \jb{t}^{-\aal-n}
    \norm{\jb{v}^{n+1}\fin}_{W^{\aal+n,1}_{x,v}}.
\end{align*}
\end{lemma}

\begin{proof}
For $0\leq t\leq1$, since $\abs{v v^\gamma}\leq\jb{v}^{n+1}$, applying~\eqref{eq:transport-L1} to each term of~\eqref{eq:Sf-derivatives} gives
\begin{align*}
    \norm{\nabla_x^\alpha\p_t^nS^{\mathrm{cur}}(t)}_{L^1_x}
    \lesssim_{\alpha,n}
    \sum_{\abs{\gamma}=n}\norm{\jb{v}^{n+1}\nabla_x^{\alpha+\gamma}\fin}_{L^1_{x,v}}.
\end{align*}
For $t\geq1$, the spatial derivatives are converted into velocity derivatives. For any smooth decaying $g$,
\begin{align*}
    \p_{v_j}\sbr{g(x-tv,v)}=-t\br{\p_{x_j}g}(x-tv,v)+\br{\p_{v_j}g}(x-tv,v),
\end{align*}
and the left side vanishes upon integration in $v$ (for $g\in W^{1,1}_{x,v}$ by density and~\eqref{eq:transport-L1}), so that $\int\br{\p_{x_j}g}(x-tv,v)\,dv=t^{-1}\int\br{\p_{v_j}g}(x-tv,v)\,dv$. Since $v v^\gamma$ does not depend on $x$ and $\nabla_x$ and $\nabla_v$ commute, $\aal+n$ applications of this identity turn~\eqref{eq:Sf-derivatives} into
\begin{align*}
    \nabla_x^\alpha\p_t^nS^{\mathrm{cur}}(t,x)
    =(-1)^nt^{-\aal-n}\sum_{\abs{\gamma}=n}c_\gamma
    \int_{\R^3}\br{\nabla_v^{\alpha+\gamma}(v v^\gamma\fin)}(x-tv,v)\,dv.
\end{align*}
By the Leibniz rule, $\nabla_v^{\alpha+\gamma}(v v^\gamma\fin)$ is a finite linear combination of terms bounded by $\jb{v}^{n+1}\abs{\nabla_v^\beta\fin}$ with $\aab\leq\aal+n$, so by~\eqref{eq:transport-L1}
\begin{align*}
    \norm{\nabla_x^\alpha\p_t^nS^{\mathrm{cur}}(t)}_{L^1_x}
    \lesssim_{\alpha,n}
        t^{-\aal-n}\sum_{\aab\leq\aal+n}\norm{\jb{v}^{n+1}\nabla_v^\beta\fin}_{L^1_{x,v}}.
\end{align*}
Since $\jb{t}^{-\aal-n}\sim1$ for $0\leq t\leq1$, applying Lemma~\ref{lem:velocity-weights} to the two bounds gives the claim for all $t\geq0$.
\end{proof}

Away from zero frequency, the free-transport current decays like $\jb{t}^{-N}$ for arbitrarily large $N$, at the cost of higher regularity of~$\fin$.

\begin{lemma}[High-frequency control of the free-transport current]
\label{lem:high-frequency-projected-S}
Fix a multi-index $\alpha$ and integers $n\geq0$, $s\geq0$ and $N>s+\aal+n$. Then, uniformly for $t\geq0$,
\begin{align}
\label{eq:high-frequency-projected-S}
\begin{split}
    \norm{(1-\chi(i\nabla_x))\nabla_x^\alpha\p_t^nS^{\mathrm{cur}}(t)}_{\Bes^s_{1,1}}
    &\lesssim_{\alpha,n,s,N}
    \jb{t}^{-N}\norm{\jb{v}^{n+1}\fin}_{W^{N,1}_{x,v}}.
\end{split}
\end{align}
\end{lemma}

\begin{proof}
Set $M_*:=N-\aal-n$, so that $M_*>s$ by hypothesis. The cutoff $1-\chi(i\nabla_x)$ retains only the frequencies $\ak\gtrsim1$, on which the Besov norm reduces to its dyadic part: by~\eqref{eq:high-blocks} and the Bernstein estimate~\eqref{eq:block-derivatives} with $M=M_*$,
\begin{align*}
    \norm{(1-\chi(i\nabla_x))\nabla_x^\alpha\p_t^nS^{\mathrm{cur}}(t)}_{\Bes^s_{1,1}}
    &\sim\sum_{\ell\geq0}2^{s\ell}\norm{P_\ell \nabla_x^\alpha\p_t^nS^{\mathrm{cur}}(t)}_{L^1_x}
    \\
    &\lesssim\sum_{\abs{\sigma}=M_*}\sum_{\ell\geq0}2^{(s-M_*)\ell}
    \norm{P_\ell \nabla_x^{\alpha+\sigma}\p_t^nS^{\mathrm{cur}}(t)}_{L^1_x}.
\end{align*}
The operators $P_\ell$ are bounded on $L^1_x$ uniformly in $\ell$ by~\eqref{eq:scaled-compact-multiplier}, and $\abs{\alpha+\sigma}+n=\aal+M_*+n=N$, so Lemma~\ref{lem:free-transport-velocity-average-L1} bounds \begin{align*}
    \norm{P_\ell \nabla_x^{\alpha+\sigma}\p_t^nS^{\mathrm{cur}}(t)}_{L^1_x}\lesssim\jb{t}^{-N}\norm{\jb{v}^{n+1}\fin}_{W^{N,1}_{x,v}},
\end{align*} uniformly in $\ell$, in $\sigma$ and in $t\geq0$. The remaining dyadic sum $\sum_{\ell\geq0}2^{(s-M_*)\ell}$ converges because $M_*>s$, which gives~\eqref{eq:high-frequency-projected-S}.
\end{proof}

At the initial time, every positive time derivative of $S^{\mathrm{cur}}$ contains the same number of spatial derivatives of $\fin$, which compensate for the negative-order multiplier at low frequency.

\begin{lemma}[Low-frequency control at the initial time]\label{lem:initial-current-low-frequency}
Fix $\delta\in(0,1)$. For every integer $n\geq1$,
\begin{align*}
\norm{\chi(i\nabla_x)\abs{\nabla_x}^{-\delta}
\p_t^nS^{\mathrm{cur}}(0)}_{L^1}
\lesssim_{n,\delta}
\norm{\jb{v}^{n+1}\fin}_{L^1_{x,v}}.
\end{align*}
\end{lemma}

\begin{proof}
By~\eqref{eq:Sf-derivatives},
\begin{align*}
\p_t^nS^{\mathrm{cur}}(0)
=(-1)^n\sum_{\abs{\gamma}=n}c_\gamma\int_{\R^3}v v^\gamma\nabla_x^\gamma\fin(x,v)\,dv.
\end{align*}
For $\abs{\gamma}=n$ the symbol $q_\gamma:=\chi\ak^{-\delta}(ik)^{\gamma}$ is supported in $\set{\ak\leq1}$ and satisfies~\eqref{mikh} with order $n-\delta>0$, so the shell decomposition $q_\gamma=\sum_{\ell\leq2}q_\gamma\vp(2^{-\ell}\cdot)$ in the proof of Lemma~\ref{lem:multiplier-bounds} and the kernel bound established there give $\norm{\mathcal F^{-1}q_\gamma}_{L^1}\lesssim\sum_{\ell\leq2}2^{\ell(n-\delta)}\lesssim_{n,\delta}1$, the series converging because $n>\delta$. Young's inequality now gives the claim.
\end{proof}

In Fourier variables the velocity average reduces to the single evaluation $\widehat S^{\mathrm{cur}}(t,k)=i\nabla_\eta\widehat\fin(k,kt)$, from which decay of any prescribed order follows directly from the regularity of $\fin$.

\begin{lemma}[Pointwise bound on the current in Fourier variables]\label{lem:current-pointwise}
For all integers $j\geq0$ and $m\geq0$, and for all $t\geq0$ and $k\neq0$,
\begin{align}\label{eq:rem-current}
    \abs{\p_t^{j}\widehat S^{\mathrm{cur}}(t,k)}
    \lesssim
    \frac{\ak^{j}}{\jb{k,kt}^{m}} \norm{\jb{v}^{j+1}\fin}_{W^{m,1}_{x,v}}.
\end{align}
\end{lemma}

\begin{proof}
By~\eqref{Scurdef}, $\widehat S^{\mathrm{cur}}(t,k)=i\nabla_\eta\widehat\fin(k,kt)$, so that differentiation in $t$ gives
\begin{align*}
    \p_t^{j}\widehat S^{\mathrm{cur}}(t,k)=i(k\cdot\nabla_\eta)^{j}\nabla_\eta\widehat\fin(k,kt),
\end{align*}
and therefore
\begin{align*}
    \abs{\p_t^{j}\widehat S^{\mathrm{cur}}(t,k)}
    \lesssim\ak^{j}\max_{\abs{\kappa}=j+1}\abs{\br{\nabla_\eta^{\kappa}\widehat\fin}(k,kt)}.
\end{align*}
Since $\nabla_\eta^{\kappa}\widehat\fin=\widehat{(-iv)^{\kappa}\fin}$, the argument giving~\eqref{Hs-data-control}, with Lemma~\ref{lem:velocity-weights} applied to $v^{\kappa}\fin$, shows that $\jb{k,\eta}^{m}\abs{\nabla_\eta^{\kappa}\widehat\fin(k,\eta)}\lesssim\norm{\jb{v}^{j+1}\fin}_{W^{m,1}_{x,v}}$ for $\abs{\kappa}=j+1$. Evaluating at $\eta=kt$ gives the claim.
\end{proof}

\section{The oscillatory transverse field}\label{complexrootSection}
This section proves the dispersive estimate for the oscillatory fields of Proposition~\ref{prop:transverse-decomposition}. Here, and throughout, $\ep_0\in(0,1/4]$ denotes a single constant depending only on the equilibrium. 

For a non-negative integer $N$ and $\delta\in(0,1)$, the data functional for the oscillatory estimates is
\begin{align}\label{eq:osc-data}
\begin{split}
    \mathcal D^{\mathrm{osc}}_{N,\delta}
    &:=\norm{\chi(i\nabla_x)\abs{\nabla_x}^{-\delta}\ETin}_{L^1}
    +\norm{\chi(i\nabla_x)\abs{\nabla_x}^{-\delta}\Bin}_{L^1}
    +\norm{\chi(i\nabla_x)\abs{\nabla_x}^{-\delta}\Delta_x^{-1}\nabla_x\times\Bin}_{L^1}\\
    &\quad+\norm{(1-\chi(i\nabla_x))\ETin}_{\Bes^{N+3}_{1,1}}
    +\norm{(1-\chi(i\nabla_x))\Bin}_{\Bes^{N+3}_{1,1}}\\
    &\quad+\norm{\chi(i\nabla_x)\abs{\nabla_x}^{-\delta}\Jin}_{L^1}+\norm{\jb{v}^4\fin}_{W^{N+4,1}_{x,v}}.
\end{split}
\end{align}
Define the envelope $\Lambda_\ep$ for $t>0$ by
\begin{align}\label{Lambda}
    \Lambda_\ep(t)
    :=\frac{1}{\jb{\frac{t}{\ep},\frac{t^{\frac{3}{2}}}{\ep^{\frac{1}{2}}}}}
    \sim\min\set{1,\frac{\ep}{t},\frac{\ep^{\frac{1}{2}}}{t^{\frac{3}{2}}}},
\end{align}
and set $\Lambda_\ep(0):=1$.

\begin{proposition}[Oscillatory bounds]\label{prop:complex-osc}
Fix a non-negative integer $N$ and $\delta\in(0,1)$. For every multi-index $\alpha$ with $\aal\leq N$, every $t\geq0$ and each choice of sign, uniformly in $0<\ep\leq\ep_0$,
\begin{align*}
    \norm{\nabla_x^\alpha B^{\mathrm{osc}}_\pm(t)}_{L^\infty}
    +\norm{\nabla_x^\alpha E^{\mathrm{osc}}_{T,\pm}(t)}_{L^\infty}+\norm{\nabla_x^\alpha\nabla_xA^{\mathrm{osc}}_\pm(t)}_{L^\infty}
    \lesssim\Lambda_\ep(t) \mathcal D^{\mathrm{osc}}_{N,\delta},
\end{align*}
where $\Lambda_\ep$ is defined in~\eqref{Lambda}.
\end{proposition}

The dispersion comes from the oscillating factor $e^{\pm i\omega t}$ of the propagator $e^{\lam_\pm t}$, where $\lam_\pm=\gamma\pm i\omega$ with $\gamma\leq0$ by Corollary~\ref{cor:roots}\ref{realpartitem}. By Lemma~\ref{lem:real-root}\ref{complexroot-beta}, $\omega(\ak)=\beta(\ak/\ep)$, where $\beta$ satisfies Klein--Gordon derivative bounds.

\begin{lemma}[Klein--Gordon dispersive estimate]\label{lem:KG-dispersive}
Let $\omega$ be the frequency of the complex pair of Lemma~\ref{lem:real-root}. Then, for all $\ep\in(0,\ep_0]$, $t\geq0$ and $f\in\Bes^3_{1,1}$,
\begin{align}\label{eq:KG-dispersive}
    \norm{e^{\pm i\omega(i\nabla_x)t}f}_{L^\infty}
    \lesssim\Lambda_\ep(t) \norm{f}_{\Bes^3_{1,1}},
\end{align}
where $\Lambda_\ep$ is defined in~\eqref{Lambda}.
\end{lemma}

\begin{proof}
Lemma~\ref{lem:real-root}\ref{complexroot-beta} verifies the hypotheses of Lemma~\ref{lem:short-time-trivial-besov} with constants independent of $\ep$: it gives the bounds on $\beta'$, $\beta''$ and $\beta'''$ and the smoothness of $k\mapsto\beta(\ak)$ on $\R^3$, which, together with the derivative bounds through order four, gives a uniform bound on $\norm{\beta(\abs{\cdot})}_{C^3(\set{\ak\leq6})}$. Since $e^{\pm i\omega(i\nabla_x)t}$ is the Fourier multiplier with symbol $e^{\pm i\beta(\ak/\ep)t}$, that lemma and the embedding $\Bes^0_{\infty,1}\hookrightarrow L^\infty$ give the claim.
\end{proof}
The amplitudes in~\eqref{eq:osc-explicit} are Fourier multipliers whose symbols are products of $\ep$, $\lam_\pm$, $R_\pm$, $R^{\mu}_\pm$, the Leray symbol $\mathbb P_k$ and~$k$. A symbol $m\in C^\infty(\R^3\setminus\set{0})$ is said to have scaled order $\nu\in\R$ if
\begin{align}\label{eq:scaled-order}
    \ak^{\abs{\kappa}}\abs{\nabla_k^\kappa m(k)}\lesssim\jb{\frac{k}{\ep}}^{\nu}
\end{align}
for every multi-index $\kappa$, uniformly in $\ep\in(0,\ep_0]$ and $k\neq0$. Scaled orders add under products, and a symbol of scaled order $\nu$ also has scaled order $\nu'$ for every $\nu'\geq\nu$. On $\set{\ak\geq\frac14}$, where $\jb{k/\ep}\sim\ak/\ep$, scaled order $\nu$ implies~\eqref{mikh} with implicit constant $\lesssim\ep^{-\nu}$.

\begin{lemma}[Scaled orders of the building blocks]\label{lem:mikhlin-multipliers}
The symbols $\ep^2R_\pm$, $\ep^2R^{\mu}_\pm$, $\lam_\pm$, $\lam_\pm^{-1}$ and $\mathbb P_k$ are of scaled orders $-1$, $-2$, $1$, $-1$ and $0$, and $e^{\gamma(\ak)t}$ has scaled order $0$ uniformly in $t\geq0$. The symbol $k$ is $\ep$ times a symbol of scaled order $1$: each factor $k$, and in particular the curl symbol $ik\times\cdot$, raises the scaled order by one and contributes a factor $\ep$.
\end{lemma}

\begin{proof}
Lemma~\ref{lem:Rpm-bounds} gives the scaled order $-1$ of $\ep^2R_\pm$, Lemma~\ref{lem:Rmu-bounds} the scaled order $-2$ of $\ep^2R^\mu_\pm$, and Corollary~\ref{cor:roots}\ref{imaginarypartitem} the scaled orders $1$ and $-1$ of $\lam_\pm$ and $\lam_\pm^{-1}$; in each case the radial chain rule converts the bounds on $r^n\p_r^n$ stated there into~\eqref{eq:scaled-order}. Corollary~\ref{cor:roots}\ref{realpartitem} gives
\begin{align*}
r^j\abs{\p_r^j\gamma(r)}
\lesssim_j\frac{r}{\jb{\frac{r}{\ep}}^2}
\lesssim-\gamma(r),
\qquad j\geq1.
\end{align*}
Repeated differentiation and $\sup_{x\geq0}x^me^{-x}<\infty$ then give $r^n\abs{\p_r^ne^{\gamma(r)t}}\lesssim_n1$. The radial chain rule therefore shows that $e^{\gamma(\ak)t}$ has scaled order $0$ uniformly in $t\geq0$. The entries of $\mathbb P_k=I-\ak^{-2}k\otimes k$ are smooth on $\R^3\setminus\set{0}$ and homogeneous of degree $0$, so each $\kappa$-derivative is homogeneous of degree $-\abs{\kappa}$ and bounded on the unit sphere. By the Leibniz rule, scaled orders add under products. Finally, $\ak^{\abs{\kappa}}\abs{\nabla_k^\kappa k}\lesssim\ak\leq\ep\jb{k/\ep}$.
\end{proof}

\begin{lemma}[Besov estimate for damped amplitudes]\label{lem:osc-multiplier}
Fix $\delta\in(0,1]$ and let $m$ have scaled order $-\nu$ for some $\nu\geq0$. Then, for every $s\in\R$, every $f$ such that $\chi(i\nabla_x)\abs{\nabla_x}^{-\delta}f\in L^1$ and $(1-\chi(i\nabla_x))f\in\Bes^{s-\nu}_{1,1}$, and every $t\geq0$,
\begin{align*}
    \norm{\chi(i\nabla_x)m(i\nabla_x)e^{\gamma(i\nabla_x)t}f}_{L^1}
    &\lesssim\norm{\chi(i\nabla_x)\abs{\nabla_x}^{-\delta}f}_{L^1},\\
    \norm{\br{1-\chi(i\nabla_x)}m(i\nabla_x)e^{\gamma(i\nabla_x)t}f}_{\Bes^{s}_{1,1}}
    &\lesssim\ep^{\nu}\norm{\br{1-\chi(i\nabla_x)}f}_{\Bes^{s-\nu}_{1,1}},
\end{align*}
and consequently, by~\eqref{eq:chi-splitting},
\begin{align*}
    \norm{m(i\nabla_x)e^{\gamma(i\nabla_x)t}f}_{\Bes^{s}_{1,1}}\lesssim\norm{\chi(i\nabla_x)\abs{\nabla_x}^{-\delta}f}_{L^1}+\ep^{\nu}\norm{\br{1-\chi(i\nabla_x)}f}_{\Bes^{s-\nu}_{1,1}}.
\end{align*}
\end{lemma}

\begin{proof}
By Lemma~\ref{lem:mikhlin-multipliers}, $me^{\gamma t}$ has scaled order $-\nu$ uniformly in $t\geq0$. On $\set{0<\ak\leq2}$ it therefore satisfies~\eqref{mikh} with order $0$, since $\jb{k/\ep}^{-\nu}\leq1$, and Lemma~\ref{lem:multiplier-bounds}\ref{lowfreq-smoothing} gives the first estimate. On $\set{\ak\geq\frac14}$ it satisfies~\eqref{mikh} with order $-\nu$ and implicit constant $\lesssim\ep^{\nu}$, and Lemma~\ref{lem:multiplier-bounds}\ref{highfreq-mikhlin} gives the second. The last estimate follows from the first two by~\eqref{eq:chi-splitting}.
\end{proof}

In~\eqref{eq:osc-explicit}, the fixed-data terms are those carrying the propagator $e^{\lam_\pm t}$, with datum $A^0$, $A^1$ or $\p_t^{j}S^{\mathrm{cur}}(0)$. Since $\lam_\pm=\gamma\pm i\omega$, each of them is, up to a constant factor of modulus one,
\begin{align}\label{eq:osc-fixed-mode}
    e^{\pm i\omega(i\nabla_x)t}\,m_\pm(i\nabla_x)e^{\gamma(i\nabla_x)t}\mathsf D,
\end{align}
with a time-independent symbol $m_\pm$ of some scaled order $q$ and a datum $\mathsf D$. Using $\ep A^1=-\ETin$ and $\nabla_x\times A^0=\Bin$, and moving every factor $k$ into the symbol except in $\nabla_x\times A^0$, the pairs $(q,\mathsf D)$ are, for $0\leq j\leq2$,
\begin{align}\label{eq:osc-pairs}
\begin{split}
    E^{\mathrm{osc}}_{T,\pm}:&\quad(0,\ETin),\quad(1,\ep A^0),\quad(-j-1,\p_t^{j}S^{\mathrm{cur}}(0)),\\
    B^{\mathrm{osc}}_{\pm}:&\quad(0,\ETin),\quad(0,\Bin),\quad(-j-1,\p_t^{j}S^{\mathrm{cur}}(0)).
\end{split}
\end{align}
The entries follow by adding the scaled orders of Lemma~\ref{lem:mikhlin-multipliers}, using $\ep A^1=-\ETin$ and $\nabla_x\times A^0=\Bin$. The pair $(1,\ep A^0)$ is the only one of positive scaled order and therefore requires separate treatment below. The datum $A^0$ is handled through the identities
\begin{align}\label{eq:A0-identity}
    A^0=-\Delta_x^{-1}\nabla_x\times\Bin,
    \qquad
    \ak\abs{\widehat A^0(k)}=\abs{\widehat\Bin(k)},
    \qquad k\neq0,
\end{align}
the second since $k\cdot\widehat A^0(k)=0$. The following bound controls all the fixed-data terms.

\begin{lemma}[Common fixed-data bound]\label{lem:osc-fixed-data}
Fix a non-negative integer $N$ and $\delta\in(0,1)$. For every pair $(q,\mathsf D)$ of~\eqref{eq:osc-pairs} and the corresponding amplitude $m_\pm$, uniformly in $0<\ep\leq\ep_0$,
\begin{align*}
    \sup_{t\geq0}\norm{m_\pm(i\nabla_x)e^{\gamma(i\nabla_x)t}\mathsf D}_{\Bes^{N+3}_{1,1}}
    \lesssim \mathcal D^{\mathrm{osc}}_{N,\delta}.
\end{align*}
\end{lemma}

\begin{proof}
The proof is organised by datum; every step below is uniform in $t\geq0$.

\noindent\emph{Data $\ETin$ and $\Bin$.} Both amplitudes have scaled order zero, so Lemma~\ref{lem:osc-multiplier} with $s=N+3$ gives the corresponding terms of $\mathcal D^{\mathrm{osc}}_{N,\delta}$.

\noindent\emph{Data $\p_t^{j}S^{\mathrm{cur}}(0)$.} The multiplier has scaled order $-j-1$, so Lemma~\ref{lem:osc-multiplier} applies with its parameter $\nu=j+1$, and with $s=N+3$ it bounds the norm by
\begin{align*}
\norm{\chi(i\nabla_x)\abs{\nabla_x}^{-\delta}\p_t^{j}S^{\mathrm{cur}}(0)}_{L^1}
+\ep^{j+1}\norm{(1-\chi(i\nabla_x))\p_t^{j}S^{\mathrm{cur}}(0)}_{\Bes^{N+2-j}_{1,1}}.
\end{align*}
The low-frequency term is treated separately according to whether $j=0$. Since $S^{\mathrm{cur}}(0)=\Jin$,
\begin{align*}
\norm{\chi(i\nabla_x)\abs{\nabla_x}^{-\delta}S^{\mathrm{cur}}(0)}_{L^1}
=
\norm{\chi(i\nabla_x)\abs{\nabla_x}^{-\delta}\Jin}_{L^1}.
\end{align*}
For $j=1,2$, Lemma~\ref{lem:initial-current-low-frequency} gives
\begin{align*}
\norm{\chi(i\nabla_x)\abs{\nabla_x}^{-\delta}\p_t^{j}S^{\mathrm{cur}}(0)}_{L^1_x}
\lesssim
\norm{\jb{v}^{j+1}\fin}_{L^1_{x,v}}
\lesssim
\norm{\jb{v}^{4}\fin}_{W^{N+4,1}_{x,v}}.
\end{align*}
On the high block, Lemma~\ref{lem:high-frequency-projected-S} at $t=0$, applied with spatial multi-index $0$, $n=j$, $s=N+2-j$ and decay exponent $N+3$, gives
\begin{align*}
\ep^{j+1}\norm{(1-\chi(i\nabla_x))\p_t^{j}S^{\mathrm{cur}}(0)}_{\Bes^{N+2-j}_{1,1}}
\lesssim
\norm{\jb{v}^{3}\fin}_{W^{N+3,1}_{x,v}},
\end{align*}
where the factor $\ep^{j+1}\leq1$ has been dropped and Lemma~\ref{lem:velocity-weights} has been used to replace the weight $\jb{v}^{j+1}$ by $\jb{v}^{3}$. Thus the low-frequency term for $j=0$ is exactly the current term in $\mathcal D^{\mathrm{osc}}_{N,\delta}$, while the low-frequency terms for $j=1,2$ and all the high-frequency terms are controlled by its ordinary weighted Sobolev term.

\noindent\emph{Datum $\ep A^0$.} Here $q=1$, so the effective symbol $\ep m_\pm$ is not of order zero on the high-frequency block and Lemma~\ref{lem:osc-multiplier} does not apply with a single datum; the norm is split by~\eqref{eq:chi-splitting} instead. On the low block $\ep\jb{k/\ep}\lesssim\ep+\ak\lesssim1$, so $\ep m_\pm e^{\gamma t}$ satisfies~\eqref{mikh} with order $0$ uniformly in $t$, and Lemma~\ref{lem:multiplier-bounds}\ref{lowfreq-smoothing} with the first identity of~\eqref{eq:A0-identity} gives
\begin{align*}
    \norm{\chi(i\nabla_x)m_\pm(i\nabla_x)e^{\gamma(i\nabla_x)t}\,\ep A^0}_{L^1}
    \lesssim\norm{\chi(i\nabla_x)\abs{\nabla_x}^{-\delta}\Delta_x^{-1}\nabla_x\times\Bin}_{L^1}.
\end{align*}
On the high block the substitution $\widehat A^0=i\ak^{-2}k\times\widehat\Bin$ moves the factor $i\ak^{-2}k\times$ into the symbol and the datum becomes $\Bin$: there $\ep\jb{k/\ep}\ak^{-1}\lesssim1$, so the combined symbol satisfies~\eqref{mikh} with order $0$, and Lemma~\ref{lem:multiplier-bounds}\ref{highfreq-mikhlin} gives
\begin{align*}
    \norm{(1-\chi(i\nabla_x))m_\pm(i\nabla_x)e^{\gamma(i\nabla_x)t}\,\ep A^0}_{\Bes^{N+3}_{1,1}}
    \lesssim\norm{(1-\chi(i\nabla_x))\Bin}_{\Bes^{N+3}_{1,1}}.
\end{align*}
Both right-hand sides are terms of $\mathcal D^{\mathrm{osc}}_{N,\delta}$, which proves the lemma.
\end{proof}

The convolution contributions to~\eqref{eq:osc-explicit} obey the following finer, time-integrable bound.

\begin{lemma}[Oscillatory convolution estimate]\label{lem:osc-convolution}
Fix a non-negative integer $N$ and $\delta\in(0,1)$, and let $\sigma_\pm$ be either of the convolution amplitudes $\tau^{\mathrm{cur},3}_\pm$ or $\upsilon^{\mathrm{cur},3}_\pm$ of~\eqref{eq:osc-explicit}. For every multi-index $\alpha$ with $\aal\leq N$, uniformly in $0<\ep\leq\ep_0$, for all $t\geq0$,
\begin{align*}
    \norm{\nabla_x^\alpha\br{\sigma_\pm e^{\lam_\pm t}*_t\p_t^3S^{\mathrm{cur}}(t)}}_{L^\infty}
    \lesssim\frac{\ep}{\jb{t}^{\frac{3}{2}}}\norm{\jb{v}^4\fin}_{W^{\aal+4,1}_{x,v}}
    \lesssim\Lambda_\ep(t)\mathcal D^{\mathrm{osc}}_{N,\delta}.
\end{align*}
\end{lemma}

\begin{proof}
By the amplitude definitions and Lemma~\ref{lem:mikhlin-multipliers}, both $\tau^{\mathrm{cur},3}_\pm$ and $\upsilon^{\mathrm{cur},3}_\pm$ have scaled order $-3$, and so does $\sigma_\pm e^{\gamma(\ak)(t-s)}$, uniformly in $0\leq s\leq t$: for every multi-index $\kappa$,
\begin{align*}
    \ak^{\abs{\kappa}}\abs{\nabla_k^\kappa\sbr{\sigma_\pm(k)e^{\gamma(\ak)(t-s)}}}\lesssim\jb{\frac{k}{\ep}}^{-3}.
\end{align*}
Applying Lemma~\ref{lem:KG-dispersive} to the integrand at time $s$ and retaining only the Klein--Gordon branch of $\Lambda_\ep(t-s)$ gives
\begin{align*}
    &\norm{\nabla_x^\alpha\br{\sigma_\pm e^{\lam_\pm t}*_t\p_t^3S^{\mathrm{cur}}(t)}}_{L^\infty}\lesssim\int_0^t\frac{\norm{\sigma_\pm(i\nabla_x)e^{\gamma(i\nabla_x)(t-s)}\nabla_x^\alpha\p_t^3S^{\mathrm{cur}}(s)}_{\Bes^3_{1,1}}}{\jb{\frac{(t-s)^{\frac{3}{2}}}{\ep^{\frac{1}{2}}}}}\,ds.
\end{align*}
On the high-frequency part $\jb{k/\ep}^{-3}\lesssim\ep^3\ak^{-3}$, so Lemma~\ref{lem:multiplier-bounds}\ref{highfreq-mikhlin} and Lemma~\ref{lem:high-frequency-projected-S}, applied with $n=3$, Besov index $0$ and decay exponent $\aal+4$, give
\begin{align*}
    \norm{(1-\chi(i\nabla_x))\sigma_\pm(i\nabla_x)e^{\gamma(i\nabla_x)(t-s)}\nabla_x^\alpha\p_t^3S^{\mathrm{cur}}(s)}_{\Bes^{3}_{1,1}}
    \lesssim\ep^3\jb{s}^{-\aal-4}\norm{\jb{v}^4\fin}_{W^{\aal+4,1}_{x,v}}.
\end{align*}
On the low-frequency part, set $u:=t-s$ and
\begin{align*}
    a_{\pm,u}(k):=\ep^{-1}\ak\,\sigma_\pm(k)e^{\gamma(\ak)u}.
\end{align*}
Since $\sigma_\pm e^{\gamma u}$ has scaled order $-3$ uniformly in $u\geq0$ and $\ak\leq\ep\jb{k/\ep}$, the symbol $a_{\pm,u}$ has scaled order $-2$, hence in particular scaled order zero, uniformly in $u\geq0$. For each multi-index $\Gamma$ with $\abs{\Gamma}=3$, set
\begin{align*}
    G_\Gamma(t,x):=-\int_{\R^3}v v^\Gamma\fin(x-tv,v)\,dv;
\end{align*}
then~\eqref{eq:Sf-derivatives} gives
\begin{align*}
    \nabla_x^\alpha\p_t^3S^{\mathrm{cur}}
    =\sum_{\abs{\Gamma}=3}c_\Gamma \nabla_x^\alpha\p_x^\Gamma G_\Gamma.
\end{align*}
Each term carries $\aal+3$ spatial derivatives. Use one $\p_{x_a}$ to cancel the factor $\abs{\nabla_x}^{-1}$, in the sense that $\sigma_\pm e^{\gamma(\ak)u}\p_{x_a}=\ep\,a_{\pm,u}\abs{\nabla_x}^{-1}\p_{x_a}$, where $\abs{\nabla_x}^{-1}\p_{x_a}$ is a Riesz transform. The remaining $\aal+2$ derivatives are written as one derivative $\p_{x_b}$ acting on $\nabla_x^{\aal+1}G_\Gamma$, to which Lemma~\ref{lem:multiplier-bounds}\ref{lowfreq-interpolation} applies with the order-zero symbol $a_{\pm,u}\abs{\nabla_x}^{-1}\p_{x_a}$; the argument of Lemma~\ref{lem:free-transport-velocity-average-L1}, applied to $G_\Gamma$, whose velocity weight is bounded by $\jb{v}^4$, then gives, with $\nabla_x^{m}$ denoting any $m$ spatial derivatives,
\begin{align*}
    \norm{\chi(i\nabla_x)\sigma_\pm(i\nabla_x)e^{\gamma(i\nabla_x)(t-s)}\nabla_x^\alpha\p_t^3S^{\mathrm{cur}}(s)}_{L^1}&\lesssim\ep\sum_{\abs{\Gamma}=3}
    \norm{\nabla_x^{\aal+1}G_\Gamma(s)}_{L^1}^{\frac{1}{2}}
    \norm{\nabla_x^{\aal+2}G_\Gamma(s)}_{L^1}^{\frac{1}{2}}
    \\
    &\lesssim\ep \jb{s}^{-\aal-\frac{3}{2}}\norm{\jb{v}^4\fin}_{W^{\aal+4,1}_{x,v}}.
\end{align*}
By~\eqref{eq:chi-splitting}, the two estimates combine, with common factor $\ep$ and common rate $\jb{s}^{-\frac{3}{2}}$, into the $\Bes^3_{1,1}$ norm required by Lemma~\ref{lem:KG-dispersive}. Inserting them into the time integral and using
\begin{align*}
    \int_0^t\jb{\frac{(t-s)^{\frac{3}{2}}}{\ep^{\frac{1}{2}}}}^{-1}\jb{s}^{-\frac{3}{2}}\,ds\lesssim\jb{t}^{-\frac{3}{2}},
\end{align*}
which follows from $\ep\leq 1$ and splitting the convolution at $s=t/2$, proves the first inequality. The second holds since $\aal\leq N$ gives $\norm{\jb{v}^4\fin}_{W^{\aal+4,1}_{x,v}}\leq\mathcal D^{\mathrm{osc}}_{N,\delta}$, and since $\ep\jb{t}^{-3/2}\leq\Lambda_\ep(t)$.
\end{proof}

\begin{proof}[Proof of Proposition~\ref{prop:complex-osc}]
For a pair $(q,\mathsf D)$ of~\eqref{eq:osc-pairs}, set $g_t:=m_\pm(i\nabla_x)e^{\gamma(i\nabla_x)t}\mathsf D$. By~\eqref{eq:osc-fixed-mode}, the corresponding fixed-data term of~\eqref{eq:osc-explicit} is $e^{\pm i\omega(i\nabla_x)t}g_t$. For every multi-index $\alpha$ with $\aal\leq N$, the dispersive estimate~\eqref{eq:KG-dispersive} gives
\begin{align*}
    \norm{\nabla_x^\alpha e^{\pm i\omega(i\nabla_x)t}g_t}_{L^\infty}
    \lesssim\Lambda_\ep(t)\norm{\nabla_x^\alpha g_t}_{\Bes^3_{1,1}}
    \lesssim\Lambda_\ep(t)\norm{g_t}_{\Bes^{N+3}_{1,1}}.
\end{align*}
Lemma~\ref{lem:osc-fixed-data} bounds the last norm by $\mathcal D^{\mathrm{osc}}_{N,\delta}$. Summing over the pairs, adding the convolution contributions by Lemma~\ref{lem:osc-convolution}, and recalling~\eqref{eq:osc-explicit} proves the bounds for $B^{\mathrm{osc}}_\pm$ and $E^{\mathrm{osc}}_{T,\pm}$.

By Proposition~\ref{prop:transverse-decomposition}, the amplitudes for $\nabla_xA^{\mathrm{osc}}_\pm$ are obtained from the magnetic amplitudes by replacing $ik\times$ with $ik\otimes\mathbb P_k$. For the terms involving $A^0$, the identity $\nabla_xA^0=-\nabla_x\Delta_x^{-1}\nabla_x\times\Bin$ gives the datum $\Bin$ with an additional symbol smooth away from the origin and homogeneous of degree zero. This symbol has scaled order $0$, by the same argument as for $\mathbb P_k$ in Lemma~\ref{lem:mikhlin-multipliers}. Absorbing it into the amplitude, the fixed-data terms retain the form~\eqref{eq:osc-fixed-mode} with the same pairs $(q,\mathsf D)$ as the magnetic terms in~\eqref{eq:osc-pairs}. The convolution amplitude, obtained from $\upsilon^{\mathrm{cur},3}_\pm$ by replacing $\nabla_x\times$ with $\nabla_x\mathbb P$, likewise retains scaled order $-3$. The proofs of Lemmas~\ref{lem:osc-fixed-data} and~\ref{lem:osc-convolution} therefore apply to these modified amplitudes, and the preceding argument gives the same bound for $\nabla_xA^{\mathrm{osc}}_\pm$.
\end{proof}

\section{The non-oscillatory transverse remainder}\label{remainderSection}
The remainder fields are estimated directly from the terms of~\eqref{eq:rem-explicit}. Their decay has two sources: the real-root terms carry the damping factor $e^{\lam_0(\ak)t}$, while the endpoint terms decay at the faster rate of the free-transport current.

\begin{proposition}[Remainder bounds]\label{prop:remainder}
Fix a non-negative integer $N$. For every multi-index $\alpha$ with $\aal\leq N$ and every $t\geq0$, the remainder fields of~\eqref{eq:rem-explicit} satisfy, uniformly for $0<\ep\leq\ep_0$,
\begin{align*}
    \norm{\nabla_x^\alpha B^{\mathrm r}(t)}_{L^\infty}+\norm{\nabla_x^\alpha\nabla_xA^{\mathrm r}(t)}_{L^\infty}
    \lesssim\frac{\ep}{\jb{t}^{\frac{\aal+4}{3}}} \mathcal D^{\mathrm r}_N,
    \qquad
    \norm{\nabla_x^\alpha E_T^{\mathrm r}(t)}_{L^\infty}
    \lesssim\frac{\ep^2}{\jb{t}^{\frac{\aal+4}{3}}} \mathcal D^{\mathrm r}_N,
\end{align*}
where
\begin{align}\label{eq:rem-data}
    \mathcal D^{\mathrm r}_N
:=\norm{\ETin}_{\Bes^{N+1}_{1,1}}
+\norm{\Bin}_{\Bes^{N+1}_{1,1}}
+\norm{\Delta_x^{-1}\nabla_x\times\Bin}_{\Bes^{N+1}_{1,1}}
+\norm{\jb{v}^3\fin}_{W^{N+4,1}_{x,v}}.
\end{align}
\end{proposition}

The proof uses the following two damping estimates.

\begin{lemma}[Two-region damping estimate]\label{lem:damped-data}
Let $\alpha$ be a multi-index and let $\mathsf D_{\mathrm{lo}}$ and $\mathsf D_{\mathrm{hi}}$ be the possibly different representations of a datum used on $\ak\leq\ep$ and on $\ak\geq\ep$. Suppose that $u$ satisfies, uniformly in $0<\ep\leq\ep_0$, $t\geq0$ and $k\neq0$,
\begin{align*}
    \abs{\widehat u(t,k)}\lesssim e^{\lam_0(\ak)t}
    \begin{mycases}
        \ak\abs{\widehat{\mathsf D}_{\mathrm{lo}}(k)},&$\ak\leq\ep$,\\
        \frac{\ep^2}{\ak^2}\abs{\widehat{\mathsf D}_{\mathrm{hi}}(k)},&$\ak\geq\ep$.
    \end{mycases}
\end{align*}
Then, for every $t\geq0$,
\begin{align*}
    \norm{\nabla_x^\alpha u(t)}_{L^\infty}
    \lesssim\frac{\ep}{\jb{t}^{\frac{\aal+4}{3}}}
    \br{\ep\norm{\mathsf D_{\mathrm{lo}}}_{\Bes^{\aal+1}_{1,1}}+\norm{\mathsf D_{\mathrm{hi}}}_{\Bes^{\aal+1}_{1,1}}}.
\end{align*}
\end{lemma}
\begin{proof}
Corollary~\ref{cor:roots}\ref{realrootitem} bounds $e^{\lam_0(\ak)t}$ by $e^{-\ak^3t/(2\ep^2)}$ for $\ak\leq\ep$ and by $e^{-\ak t/2}$ for $\ak\geq\ep$. By Lemma~\ref{lem:embeddings}\ref{emb:fourier-L1}, it therefore suffices to integrate $\ak^{\aal}\abs{\widehat u}$ in $k$, over $\ak\leq\ep$, over $\ep\leq\ak\leq1$ and over the dyadic shells $\ak\sim2^\ell$ with $\ell\geq0$. Each region meets only finitely many Littlewood--Paley blocks, and there the datum is bounded by the $L^1$ norms of those blocks, which is how the Besov norms arise. On $\ak\leq\ep\leq1/4$ only the block $P_{-1}$ occurs. The substitution $y=\ak^3t/\ep^2$ for $t\geq1$ and direct integration for $t\leq1$ give
\begin{align*}
    \int_{\ak\leq\ep}\ak^{\aal}\abs{\widehat u}\,dk
    \lesssim\br{\frac{\ep^2}{\jb{t}}}^{\frac{\aal+4}{3}}\norm{P_{-1}\mathsf D_{\mathrm{lo}}}_{L^1}
    \lesssim\frac{\ep^2}{\jb{t}^{\frac{\aal+4}{3}}}\norm{\mathsf D_{\mathrm{lo}}}_{\Bes^{\aal+1}_{1,1}}.
\end{align*}
On $\ep\leq\ak\leq1$, the estimate $\ep^2e^{-\ak t/2}\lesssim\ep\jb{t}^{-1}e^{-\ak t/4}$, which holds since $\ak\geq\ep$ and $\ep e^{-\ep t/4}\lesssim\jb{t}^{-1}$, and polar coordinates give
\begin{align*}
    \int_{\ep\leq\ak\leq1}\ak^{\aal}\abs{\widehat u}\,dk
    \lesssim\frac{\ep}{\jb{t}}\norm{\mathsf D_{\mathrm{hi}}}_{\Bes^{0}_{1,1}}\int_0^1r^{\aal}e^{-\frac{rt}{4}}\,dr
    \lesssim\frac{\ep}{\jb{t}^{\aal+2}}\norm{\mathsf D_{\mathrm{hi}}}_{\Bes^{0}_{1,1}}.
\end{align*}
On the shell $\ak\sim2^\ell$ the amplitude $\ep^2\ak^{\aal-2}$ and the volume $2^{3\ell}$ produce the weight $\ep^22^{(\aal+1)\ell}$, while $e^{-\ak t/2}\leq e^{-2^{\ell-2}t}$; hence, using $\ep^2\leq\ep$,
\begin{align*}
    \int_{\ak\geq1}\ak^{\aal}\abs{\widehat u}\,dk
    \lesssim\ep\sum_{\ell\geq0}2^{(\aal+1)\ell}e^{-2^{\ell-2}t}\norm{P_\ell \mathsf D_{\mathrm{hi}}}_{L^1}
    \lesssim\frac{\ep}{\jb{t}^{m}}\norm{\mathsf D_{\mathrm{hi}}}_{\Bes^{\aal+1}_{1,1}}
\end{align*}
for every $m>0$, since $e^{-2^{\ell-2}t}\leq e^{-t/4}\lesssim_m\jb{t}^{-m}$. The claim follows by taking $m\geq(\aal+4)/3$ and using $\aal+2\geq(\aal+4)/3$.
\end{proof}

\begin{lemma}[Damped convolution]\label{lem:damped-convolution}
Let $m\geq2$ be an integer. Uniformly in $0<\ep\leq\ep_0$, for every $t\geq0$ and $k\neq0$,
\begin{align*}
    \int_0^te^{\lam_0(\ak)(t-s)}\frac{ds}{\jb{k,ks}^{m}}
    \lesssim
    \begin{mycases}
        \frac{\ep^2}{\ak^3}\frac{1}{\jb{k,kt}^{m}}
        +\frac{1}{\ak}e^{-\frac{\ak^3t}{4\ep^2}},&$\ak\leq\ep$,\\
        \frac{1}{\ak}\frac{1}{\jb{k,kt}^{m}},&$\ak\geq\ep$.
    \end{mycases}
\end{align*}
\end{lemma}
\begin{proof}
The identity $\jb{k,kt}^2=\jb{k}^2+\ak^2t^2$ and $s^2+(t-s)^2\geq t^2/2$ give $\jb{k,kt}\leq\sqrt2\jb{k,ks}\jb{\ak(t-s)}$, and $\jb{k,ks}\geq\jb{k,kt}/2$ for $s\geq t/2$. For $\ak\geq\ep$, where Corollary~\ref{cor:roots}\ref{realrootitem} gives $\abs{\lam_0(\ak)}\geq\ak/2$, the first of these and the substitution $y=\ak(t-s)$ give
\begin{align*}
    \int_0^te^{\lam_0(\ak)(t-s)}\frac{ds}{\jb{k,ks}^{m}}
    \lesssim\frac{1}{\jb{k,kt}^{m}}\int_0^te^{-\frac{\ak(t-s)}{2}}\jb{\ak(t-s)}^{m}\,ds
    \lesssim\frac{1}{\ak\jb{k,kt}^{m}}.
\end{align*}
For $\ak\leq\ep$, where $\abs{\lam_0(\ak)}\geq\ak^3/(2\ep^2)$, split at $s=t/2$. On $[t/2,t]$, the second inequality gives $\jb{k,ks}^{-m}\lesssim\jb{k,kt}^{-m}$, and the integral of the exponential is at most $2\ep^2\ak^{-3}$. On $[0,t/2]$, the exponential is at most $e^{-\ak^3t/(4\ep^2)}$, while $\jb{k,ks}\geq\jb{\ak s}$ and $\int_0^\infty\jb{\ak s}^{-m}\,ds\lesssim\ak^{-1}$.
\end{proof}

\begin{proof}[Proof of Proposition~\ref{prop:remainder}]
The amplitudes in~\eqref{eq:rem-explicit} satisfy, uniformly for $0<\ep\leq\ep_0$ and $r=\ak>0$,
\begin{align}\label{eq:rem-amplitudes}
\begin{alignedat}{2}
    \abs{\tau_0^0(k)}&\lesssim\ep^2r\min\set{r,\frac{\ep^3}{r^2}},
    &\qquad
    \abs{\upsilon_0^0(k)}&\lesssim\ep r\min\set{r,\frac{\ep^3}{r^2}},
    \\
    \abs{\tau_0^1(k)}&\lesssim\ep^2\min\set{r,\frac{\ep^3}{r^2}},
    &
    \abs{\upsilon_0^1(k)}&\lesssim\ep\min\set{r,\frac{\ep^3}{r^2}},
    \\
    \abs{\tau_0^\mu(k)}&\lesssim\ep r\min\set{1,\frac{\ep^2}{r^2}},
    &
    \abs{\upsilon_0^\mu(k)}&\lesssim r\min\set{1,\frac{\ep^2}{r^2}},
    \\
    \abs{\tau_0^{\mathrm{cur},1}(k)}&\lesssim\min\set{r,\frac{\ep^4}{r^3}},
    &
    \abs{\upsilon_0^{\mathrm{cur},0}(k)}&\lesssim\min\set{\frac{r^2}{\ep},\frac{\ep^3}{r^2}},
\end{alignedat}
\end{align}
and
\begin{align}\label{endpoints}
    \abs{\tau_\pm^{\mathrm{cur},j}(k)}+\abs{\upsilon_\pm^{\mathrm{cur},j}(k)}\lesssim\frac{\ep^{j}}{r^{j}},\;\;1\leq j\leq3.
\end{align}
The bounds follow from Corollary~\ref{cor:roots},~\eqref{eq:R0-global-bound} and Lemmas~\ref{lem:Rpm-bounds} and~\ref{lem:Rmu-bounds}. The real-root amplitude $\tau_0^{\mathrm{cur},1}=-\ep^2R_0\mathbb P$ also satisfies the weaker bound
\begin{align}\label{eq:rem-in-weakened}
    \abs{\tau_0^{\mathrm{cur},1}(k)}\lesssim\min\set{r,\frac{\ep^3}{r^2}},
\end{align}
used for the direct application of Lemma~\ref{lem:damped-data}, the sharper bound above being used for the convolution term. The terms are treated in four groups.

\noindent\emph{Fixed terms: data $\ETin$ and $S^{\mathrm{cur}}(0)$.}
Using $\ep A^1=-\ETin$, the bounds~\eqref{eq:rem-amplitudes},~\eqref{eq:rem-in-weakened} and Lemma~\ref{lem:damped-data} give factors $\ep^3$, $\ep^2$ and $\ep^2$ for the terms with amplitudes $\tau_0^1$ and $\upsilon_0^1$ and for the fixed-data term $\tau_0^{\mathrm{cur},1}e^{\lam_0t}S^{\mathrm{cur}}(0)$, respectively. Moreover $S^{\mathrm{cur}}(0)=\Jin$, so that
\begin{align*}
    \norm{S^{\mathrm{cur}}(0)}_{\Bes^{N+1}_{1,1}}
    \lesssim\norm{\jb{v}\fin}_{W^{N+2,1}_{x,v}}
\end{align*}
by Lemma~\ref{lem:embeddings}\ref{emb:sobolev-besov}, and all three terms are controlled by $\mathcal D^{\mathrm r}_N$.

\noindent\emph{Fixed terms: datum $A^0$.}
Using $\ak\abs{\widehat A^0}=\abs{\widehat\Bin}$ from~\eqref{eq:A0-identity}, the pairs $(\mathsf D_{\mathrm{lo}},\mathsf D_{\mathrm{hi}})$ in Lemma~\ref{lem:damped-data} for $\tau_0^0$, $\tau_0^\mu$, $\upsilon_0^0$ and $\upsilon_0^\mu$ are, respectively,
\begin{align*}
    (\ep^2\Bin,\ep^3\Bin),
    \qquad(\ep A^0,\ep\Bin),
    \qquad(\ep\Bin,\ep^2\Bin),
    \qquad(A^0,\Bin).
\end{align*}
In the two equilibrium terms the factor $r$ supplies on $\ak\leq\ep$ the factor $\ak$ required by Lemma~\ref{lem:damped-data}, and on $\ak\geq\ep$ converts $A^0$ into $\Bin$. Measuring the region $\ak\leq\ep$ through $A^0$ rather than $\Bin$ is what supplies that extra power there, hence the rate $\jb{t}^{-(\aal+4)/3}$. The datum $\Bin$ alone would give only $\jb{t}^{-(\aal+3)/3}$. The bounds~\eqref{eq:rem-amplitudes} and Lemma~\ref{lem:damped-data} give the factors $\ep^4$, $\ep^2$, $\ep^3$ and $\ep$, respectively, which are the required electric and magnetic powers.

\noindent\emph{Endpoint terms.}
The endpoint bounds of~\eqref{endpoints} and Lemma~\ref{lem:current-pointwise} with $m=N+4$ give, for $j=0,1,2$,
\begin{align*}
    \ak^{\aal}\br{\abs{\tau_\pm^{\mathrm{cur},j+1}(k)}+\abs{\upsilon_\pm^{\mathrm{cur},j+1}(k)}}\abs{\p_t^{j}\widehat S^{\mathrm{cur}}(t,k)}
    &\lesssim\frac{\ep^{j+1}\ak^{\aal-1}}{\jb{k,kt}^{N+4}}\norm{\jb{v}^{j+1}\fin}_{W^{N+4,1}_{x,v}}.
\end{align*}
This bound holds for both amplitudes over the whole range, only $j=1,2$ occurring in the electric remainder. Lemma~\ref{lem:embeddings}\ref{emb:fourier-L1} and~\eqref{eq:freq-integral} give the faster decay $\jb{t}^{-\aal-2}$. The worst powers are $\ep^2$ for the electric endpoints, attained at $j=1$, and $\ep$ for the magnetic endpoints, attained at $j=0$. Since $\aal+2\geq(\aal+4)/3$, these bounds are stronger than required.

\noindent\emph{Convolution terms.}
For the magnetic convolution, the bound for $\upsilon_0^{\mathrm{cur},0}$ in~\eqref{eq:rem-amplitudes}, Lemma~\ref{lem:current-pointwise} with $m=N+4$ and Lemma~\ref{lem:damped-convolution} give, using $\ep^2\leq\ak^2$ on $\ak\geq\ep$,
\begin{align*}
    \ak^{\aal}\abs{\upsilon_0^{\mathrm{cur},0}(k)}\int_0^te^{\lam_0(\ak)(t-s)}\abs{\widehat S^{\mathrm{cur}}(s,k)}\,ds
    \lesssim\norm{\jb{v}\fin}_{W^{N+4,1}_{x,v}}
    \br{\frac{\ep \ak^{\aal-1}}{\jb{k,kt}^{N+4}}+\frac{\ak^{\aal+1}}{\ep}e^{-\frac{\ak^3t}{4\ep^2}}}.
\end{align*}
For $t\geq1$, integrate the preceding bound in $k$
using~\eqref{eq:freq-integral} and the substitution
$y=\ak^3t/\ep^2$.
For $t\leq1$, use $\abs{\upsilon_0^{\mathrm{cur},0}(k)}\lesssim\ep$
and Lemma~\ref{lem:current-pointwise} with $m=N+4$ to estimate
the original convolution directly.
Fourier inversion therefore gives
\begin{align*}
    \norm{\nabla_x^\alpha\br{\upsilon_0^{\mathrm{cur},0}e^{\lam_0t}*_tS^{\mathrm{cur}}(t)}}_{L^\infty}
    &\lesssim\br{\frac{\ep}{\jb{t}^{\aal+2}}+\frac{1}{\ep}\br{\frac{\ep^2}{\jb{t}}}^{\frac{\aal+4}{3}}}
    \norm{\jb{v}\fin}_{W^{N+4,1}_{x,v}}\\
    &\lesssim\frac{\ep}{\jb{t}^{\frac{\aal+4}{3}}}\norm{\jb{v}\fin}_{W^{N+4,1}_{x,v}},
\end{align*}
where the last inequality uses $\ep^{\frac{2(\aal+4)}{3}-1}\leq\ep$. For $\tau_0^{\mathrm{cur},1}e^{\lam_0t}*_t\p_tS^{\mathrm{cur}}$, the derivative of the current supplies the factor $\ak$ previously supplied by the curl, while the amplitude carries one additional factor~$\ep$. The same argument gives
\begin{align*}
    \norm{\nabla_x^\alpha\br{\tau_0^{\mathrm{cur},1}e^{\lam_0t}*_t\p_tS^{\mathrm{cur}}(t)}}_{L^\infty}
    \lesssim\frac{\ep^2}{\jb{t}^{\frac{\aal+4}{3}}}\norm{\jb{v}^2\fin}_{W^{N+4,1}_{x,v}}.
\end{align*}
By Lemma~\ref{lem:velocity-weights}, all the data norms above are bounded by $\mathcal D^{\mathrm r}_N$, which proves the bounds for $B^{\mathrm r}$ and $E_T^{\mathrm r}$. 

Finally, by Proposition~\ref{prop:transverse-decomposition} the terms of $\nabla_xA^{\mathrm r}$ are those of $B^{\mathrm r}$ with $ik\times$ replaced by $ik\otimes\mathbb P_k$. Since $\abs{\mathbb P_k}\leq1$, the amplitude bounds~\eqref{eq:rem-amplitudes} hold for them unchanged, and $\ak\abs{\widehat A^0}=\abs{\widehat\Bin}$ is used as before, so the same argument gives the asserted bound for $\nabla_xA^{\mathrm r}$.
\end{proof}

\section{Proofs of the main theorem}\label{sec:proofs}
This section proves Theorem~\ref{maintheorem} using the field estimates established above. The electrostatic comparison additionally requires a bound on the impulse of $E_T$ along free characteristics.

The following lemma bounds the functionals $\mathcal H_{N+3}$, $\mathcal D^{\mathrm{osc}}_{N,\delta}$ and $\mathcal D^{\mathrm r}_N$, defined in~\eqref{eq:longitudinal-data},~\eqref{eq:osc-data} and~\eqref{eq:rem-data}, respectively, by the unified data norm $\mathcal D_{N,p,\delta}$ of~\eqref{eq:data-norm}.
\begin{lemma}[Comparison with the unified data norm]\label{lem:data-control}
Fix a non-negative integer $N$, $1\leq p\leq\infty$ and $\delta\in(0,1)$. Then
\begin{align*}
    \mathcal H_{N+3}\sbr{\fin}+\mathcal D^{\mathrm{osc}}_{N,\delta}+\mathcal D^{\mathrm r}_N
    \lesssim\mathcal D_{N,p,\delta},
\end{align*}
with implicit constant depending only on $N$ and $\delta$. In particular, all three quantities are finite under the hypotheses of Theorem~\ref{maintheorem}.
\end{lemma}

\begin{proof}
For all $s,s'\in\R$ with $s+\delta\leq s'$ and every $G$,
\begin{align}\label{eq:data-trade}
    \norm{G}_{\Bes^{s}_{1,1}}
    \lesssim\norm{\abs{\nabla_x}^{-\delta}G}_{\Bes^{s'}_{1,1}}.
\end{align}
To see this, split by~\eqref{eq:chi-splitting}: on the low block, apply Lemma~\ref{lem:multiplier-bounds}\ref{lowfreq-smoothing} with symbol $\equiv1$; on the high block, $(1-\chi(i\nabla_x))G=\abs{\nabla_x}^{\delta}(1-\chi(i\nabla_x))\abs{\nabla_x}^{-\delta}G$, and $\ak^{\delta}$ satisfies the Mikhlin bound of order $\delta$ on $\set{\ak\geq1/4}$, so Lemma~\ref{lem:multiplier-bounds}\ref{highfreq-mikhlin} yields the claim. 

The field norms of $\mathcal D^{\mathrm{osc}}_{N,\delta}$ and $\mathcal D^{\mathrm r}_N$ involve only $G\in\set{\ETin,\Bin,\Delta_x^{-1}\nabla_x\times\Bin}$, and each is bounded by~\eqref{eq:data-trade}: the low-frequency terms are single blocks of it, while the high-frequency terms are bounded by~\eqref{eq:data-trade}, using $\delta<1$. The current term and the weighted Sobolev norm of $\mathcal D^{\mathrm{osc}}_{N,\delta}$ are terms of $\mathcal D_{N,p,\delta}$. By Lemma~\ref{lem:velocity-weights}, the weighted Sobolev norm of $\mathcal D^{\mathrm r}_N$ is bounded by the latter and $\mathcal H_{N+3}\sbr{\fin}$ is controlled similarly via~\eqref{Hs-data-control}.
\end{proof}

\subsection{Proof of the damping and scattering conclusions}

\begin{proof}[Proof of parts~\ref{maintheorem-longitudinal}--\ref{maintheorem-scattering} of Theorem~\ref{maintheorem}]
Take $c_0\geq\ep_0^{-1}$, so that the preceding estimates apply uniformly for $c\geq c_0$. The unique global mild solution is obtained from the closed equations~\eqref{VolterraF} and~\eqref{eq:wave-closed} through the Duhamel formula~\eqref{eq:vlasov-duhamel}, and the representation formulae of Sections~2--4 apply.

\noindent\emph{Field decay.} For the longitudinal field, the pair $(\rho,E_L)$ satisfies~\eqref{VolterraF}--\eqref{EF-rhoF} with $F_{\mathrm{in}}=\fin$, as established in Section~\ref{sectionlongitudinalfielddecay}, while $\jb{v}\fin\in L^1_{x,v}$ and Lemma~\ref{lem:data-control} bounds $\mathcal H_{N+3}\sbr{\fin}$ by $\mathcal D_{N,p,\delta}$. Proposition~\ref{ELtheorem} with $M=N+3$ is therefore available for every $\aal\leq N$, that is for every $\aal<M-2$, and yields part~\ref{maintheorem-longitudinal}. 

For the transverse fields, the decomposition asserted in part~\ref{maintheorem-transverse} is~\eqref{eq:field-split}. Proposition~\ref{prop:complex-osc} and Lemma~\ref{lem:data-control}, summed over the two signs and over all multi-indices $\alpha$ with $\aal\leq N$, give~\eqref{maintheoremosc} after the substitution $\ep=1/c$. For the remainder, Proposition~\ref{prop:remainder} and Lemma~\ref{lem:data-control} give~\eqref{maintheoremrem}. This proves part~\ref{maintheorem-transverse}.

\noindent\emph{Scattering of the profile.} Write $g(t,x,v):=f(t,x+tv,v)$ for the profile of $f$ along free transport. By the Vlasov equation of~\eqref{linVM}, $\p_tg(t,x,v)=-\nabla_v\mu(v)\cdot E(t,x+tv)$, so that integration in time gives
\begin{align}\label{eq:profile-duhamel}
    g(t,x,v)=\fin(x,v)-\int_0^t\nabla_v\mu(v)\cdot E(s,x+sv)\,ds,
\end{align}
and define
\begin{align}\label{eq:profile-limit}
    f_\infty(x,v):=\fin(x,v)-\int_0^\infty\nabla_v\mu(v)\cdot E(s,x+sv)\,ds.
\end{align}
The field entering~\eqref{eq:profile-limit} is $E=E_L+\ETosc+E_T^{\mathrm r}$, and each of its three parts has been bounded above, leading to the overall rate
\begin{align}\label{eq:E-integrable}
    \norm{\nabla_x^\alpha E(s)}_{L^\infty}
    \lesssim\frac{\mathcal D_{N,p,\delta}}{\jb{s}^{\frac{4}{3}}},
    \qquad \aal\leq N,\quad s\geq0,
\end{align}
the transverse remainder being the slowest of the three at derivative order zero, and the oscillatory part for $\aal\geq1$. Indeed, $\Lambda_\ep(s)\leq1$ for $s\leq1$, while for $s\geq1$ one has $\Lambda_\ep(s)\leq\ep^{1/2}s^{-3/2}\leq s^{-4/3}$ at every order $\aal\leq N$. The longitudinal and electric-remainder estimates decay at least as fast. The right-hand side of~\eqref{eq:E-integrable} is therefore integrable on $[0,\infty)$. Hence the integral in~\eqref{eq:profile-limit} converges absolutely and, spatial translations preserving the $L^\infty_x$ norm at each fixed $v$ and $\norm{\nabla_v\mu}_{L^p_v}$ being finite for the Poisson equilibrium~\eqref{Poisson}, summing over $\aal\leq N$ and using the term $\norm{\fin}_{L^p_vW^{N,\infty}_x}$ of~\eqref{eq:data-norm} gives $f_\infty\in L^p_vW^{N,\infty}_x$ with $\norm{f_\infty}_{L^p_vW^{N,\infty}_x}\lesssim\mathcal D_{N,p,\delta}$. Subtracting \eqref{eq:profile-limit} from~\eqref{eq:profile-duhamel} and using~\eqref{eq:E-integrable} once more,
\begin{align*}
    \norm{\nabla_x^\alpha g(t)-\nabla_x^\alpha f_{\infty}}_{L^p_vL^\infty_x}
    =\norm{\int_t^\infty\nabla_v\mu(v)\cdot\nabla_x^\alpha E(s,x+sv)\,ds}_{L^p_{v}L^\infty_x}
    \lesssim\mathcal D_{N,p,\delta}\int_t^\infty\frac{ds}{\jb{s}^{\frac{4}{3}}}
    \lesssim\frac{\mathcal D_{N,p,\delta}}{\jb{t}^{\frac{1}{3}}},
\end{align*}
and summing over $\aal\leq N$ gives part~\ref{maintheorem-scattering}.
\end{proof}

\subsection{Proof of the electrostatic-limit conclusion}
The comparison rests on the following two lemmas.

\begin{lemma}[Uniform bound on the vector potential]\label{lem:potential-bound}
Fix a non-negative integer $N$, $1\leq p\leq\infty$ and $\delta\in(0,1)$, and let $A$ be the Coulomb-gauge vector potential, given by~\eqref{eq:solution-representation}. Then, uniformly in $0<\ep\leq\ep_0$,
\begin{align*}
    \sup_{t\geq0}\norm{A(t)}_{W^{N,\infty}_x}\lesssim\mathcal D_{N,p,\delta}.
\end{align*}
\end{lemma}

\begin{proof}
By Corollary~\ref{cor:roots}, all three roots have negative real part, with $\abs{\lam_0(r)}\leq r$ and $\abs{\lam_\pm(r)}\lesssim\jb{r/\ep}$. Since $\abs{e^{\lam_jt}}\leq1$ for $t\geq0$, the residue bounds of Lemmas~\ref{lem:R0-bounds}, \ref{lem:Rpm-bounds} and~\ref{lem:Rmu-bounds} give, for $r>0$ and $t\geq0$,
\begin{align*}
    \ep\abs{H(t,r)}
    &\leq\ep\sum_{j\in\set{0,\pm}}\abs{R_j(r)}
    \lesssim\frac{r}{\ep\jb{\frac{r}{\ep}}^{4}}+\frac{1}{\ep\jb{\frac{r}{\ep}}}
    \lesssim\frac{1}{r},\\
    \ep^2\br{\abs{\p_tH(t,r)}+\abs{H^\mu(t,r)}}
    &\leq\ep^2\sum_{j\in\set{0,\pm}}\br{\abs{\lam_j(r)R_j(r)}+\abs{R^\mu_j(r)}}
    \lesssim\frac{r^2}{\jb{\frac{r}{\ep}}^{4}}+1
    \lesssim1,
\end{align*}
the last inequalities because $\ep\jb{r/\ep}\geq r$ and $r^2\jb{r/\ep}^{-4}\leq\ep^2\leq1$. Substituting these bounds into~\eqref{eq:solution-representation}, with $\abs{\mathbb P_k}\leq1$, $\ep A^1=-\ETin$ and $\ak\abs{\widehat A^0(k)}=\abs{\widehat\Bin(k)}$ from~\eqref{eq:A0-identity}, gives
\begin{align}\label{eq:potential-fourier-bound}
    \sup_{t\geq0}\abs{\widehat A(t,k)}
    \lesssim\frac{\abs{\widehat\Bin(k)}+\abs{\widehat\ETin(k)}}{\ak}
    +\frac{1}{\ak}\int_0^\infty\abs{\widehat S^{\mathrm{cur}}(s,k)}\,ds,
    \qquad k\neq0.
\end{align}
For the current term, Lemma~\ref{lem:current-pointwise} with $j=0$ and $m=N+4$, followed by the substitution $u=\ak s/\jb{k}$, for which $\jb{k,ks}=\jb{k}\jb{u}$, gives
\begin{align*}
    \int_0^\infty\abs{\widehat S^{\mathrm{cur}}(s,k)}\,ds
    \lesssim\norm{\jb{v}\fin}_{W^{N+4,1}_{x,v}}\int_0^\infty\frac{ds}{\jb{k,ks}^{N+4}}
    \lesssim\frac{\norm{\jb{v}\fin}_{W^{N+4,1}_{x,v}}}{\ak\jb{k}^{N+3}}.
\end{align*}
After multiplication by $\jb{k}^N\ak^{-1}$, the resulting frequency weight is integrable:
\begin{align*}
    \int_{\R^3}\frac{dk}{\ak^2\jb{k}^{3}}
    =4\pi\int_0^\infty\frac{dr}{\jb{r}^{3}}<\infty.
\end{align*}
For either initial field $G\in\set{\Bin,\ETin}$, the partition~\eqref{eq:LP-partition} and the bound $\abs{\widehat{P_\ell G}}\leq\norm{P_\ell G}_{L^1}$ give
\begin{align*}
    \int_{\R^3}\frac{\jb{k}^{N}}{\ak}\abs{\widehat G(k)}\,dk
    \lesssim\norm{P_{-1}G}_{L^1}+\sum_{\ell\geq0}2^{(N+2)\ell}\norm{P_\ell G}_{L^1}
    =\norm{G}_{\Bes^{N+2}_{1,1}},
\end{align*}
since $\ak^{-1}$ is integrable on the low block, while on the support of the $\ell$-th block, of volume $\lesssim2^{3\ell}$, the weight $\jb{k}^{N}\ak^{-1}$ is of order $2^{(N-1)\ell}$. Applying Lemma~\ref{lem:embeddings}\ref{emb:fourier-L1} to each spatial derivative of order at most $N$ and integrating~\eqref{eq:potential-fourier-bound} now gives
\begin{align*}
    \sup_{t\geq0}\norm{A(t)}_{W^{N,\infty}_x}
    \lesssim\int_{\R^3}\jb{k}^{N}\sup_{t\geq0}\abs{\widehat A(t,k)}\,dk\lesssim\mathcal D_{N,p,\delta},
\end{align*}
the last inequality by~\eqref{eq:data-trade}, by Lemma~\ref{lem:velocity-weights} and the definition~\eqref{eq:data-norm}.
\end{proof}

\begin{lemma}[Transverse impulse along free characteristics]\label{lem:characteristic-impulse}
Fix a non-negative integer $N$, $1\leq p\leq\infty$ and $\delta\in(0,1)$. Then, uniformly in $0<\ep\leq\ep_0$ and $v\in\R^3$,
\begin{align*}
    \sup_{t\geq0}
    \norm{\int_0^t E_T(s,x+sv)\,ds}_{W^{N,\infty}_x}
    \lesssim\ep\jb{v}\mathcal D_{N,p,\delta}.
\end{align*}
\end{lemma}

\begin{proof}
Since $E_T=-\ep\p_tA$, integration by parts along the free characteristic gives
\begin{align}\label{eq:impulse-identity}
    \int_0^t E_T(s,x+sv)\,ds
    =-\ep\br{A(t,x+tv)-A^0(x)}
    +\ep\int_0^t(v\cdot\nabla_x)A(s,x+sv)\,ds.
\end{align}
Spatial translations preserve the $W^{N,\infty}_x$ norm, so Lemma~\ref{lem:potential-bound}, including its bound at $t=0$, controls the endpoint terms by $\ep\mathcal D_{N,p,\delta}$. For the integral term, the potential decomposition of Proposition~\ref{prop:transverse-decomposition}, Propositions~\ref{prop:complex-osc} and~\ref{prop:remainder}, and Lemma~\ref{lem:data-control} give
\begin{align*}
    \norm{\ep\int_0^t(v\cdot\nabla_x)A(s,x+sv)\,ds}_{W^{N,\infty}_x}
    &\leq\ep\abs{v}\int_0^t\norm{\nabla_xA(s)}_{W^{N,\infty}_x}\,ds\\
    &\lesssim\ep\abs{v}\mathcal D_{N,p,\delta}
    \int_0^t\br{\Lambda_\ep(s)+\frac{\ep}{\jb{s}^{\frac{4}{3}}}}\,ds.
\end{align*}
Since $\ep\leq1$ and $\Lambda_\ep(s)\lesssim\jb{s}^{-3/2}$, both time integrals are bounded independently of $t$ and $\ep$. Combining the two contributions and using $1+\abs{v}\lesssim\jb{v}$ proves the claim.
\end{proof}

\begin{proof}[Proof of the electrostatic-limit conclusion of Theorem~\ref{maintheorem}]
The phase-space discrepancy of~\eqref{eq:phase-space-matching} is abbreviated to $\mathcal M_{N,p}:=\mathcal M_{N,p}(\fin,\fVPin)$. All implicit constants are independent of $c\geq c_0$ and of the data, so that every bound is homogeneous in $\mathcal D_{N,p,\delta}$ and $\mathcal M_{N,p}$.

\noindent\emph{Comparison of the longitudinal fields.} By~\eqref{Hs-data-control} and~\eqref{eq:phase-space-matching}, \begin{align*}
    \mathcal H_{N+3}\sbr{\fin-\fVPin}\lesssim\norm{\jb{v}(\fin-\fVPin)}_{W^{N+3,1}_{x,v}}\leq\mathcal M_{N,p}.
\end{align*} As established in Section~\ref{sectionlongitudinalfielddecay}, the pair $(\rho-\rho^{\mathrm{VP}},E_L-\EVP)$ satisfies~\eqref{VolterraF}--\eqref{EF-rhoF} with datum $\fin-\fVPin$, so, as $\jb{v}(\fin-\fVPin)\in L^1_{x,v}$, Proposition~\ref{ELtheorem} with $M=N+3$ applies to it; summing over $\aal\leq N$, and writing $\EVP=E_L-(E_L-\EVP)$ with the longitudinal bound of part~\ref{maintheorem-longitudinal},
\begin{align}\label{eq:VP-longitudinal-decay}
    \norm{E_L(t)-\EVP(t)}_{W^{N,\infty}}
    \lesssim\frac{\mathcal M_{N,p}}{\jb{t}^2},
    \qquad
    \norm{\EVP(t)}_{W^{N,\infty}}
    \lesssim\frac{\mathcal D_{N,p,\delta}+\mathcal M_{N,p}}{\jb{t}^2}.
\end{align}

\noindent\emph{Convergence of the Vlasov--Poisson profile.} The two systems are linearised about the same equilibrium, so the Vlasov equation of~\eqref{linVP} differs from that of~\eqref{linVM} only in its field, and the Vlasov--Poisson profile $g^{\mathrm{VP}}(t,x,v):=\fVP(t,x+tv,v)$ satisfies~\eqref{eq:profile-duhamel} with $\fVPin$ and $\EVP$ in place of $\fin$ and $E$. Its datum inherits the regularity of $\fin$, since \begin{align*}
\norm{\fVPin}_{L^p_vW^{N,\infty}_x}\leq\norm{\fin}_{L^p_vW^{N,\infty}_x}+\norm{\fin-\fVPin}_{L^p_vW^{N,\infty}_x}\lesssim\mathcal D_{N,p,\delta}+\mathcal M_{N,p}
\end{align*} by~\eqref{eq:phase-space-matching}, and the bound~\eqref{eq:VP-longitudinal-decay} is integrable in time. The argument of part~\ref{maintheorem-scattering} of Theorem~\ref{maintheorem} therefore applies verbatim and shows that $g^{\mathrm{VP}}$ converges in $L^p_vW^{N,\infty}_x$ to
\begin{align*}
    f^{\mathrm{VP}}_\infty(x,v)
    :=\fVPin(x,v)-\int_0^\infty\nabla_v\mu(v)\cdot\EVP(s,x+sv)\,ds.
\end{align*}

\noindent\emph{Comparison of the two profiles.} Subtracting the two Duhamel formulas, splitting $E=E_L+E_T$ and using that the Vlasov--Poisson field is longitudinal,
\begin{align*}
    g(t)-g^{\mathrm{VP}}(t)
    &=(\fin-\fVPin)
    -\int_0^t\nabla_v\mu(v)\cdot\br{E_L(s,x+sv)-\EVP(s,x+sv)}\,ds\\
    &\quad-\int_0^t\nabla_v\mu(v)\cdot E_T(s,x+sv)\,ds,
\end{align*}
the transverse field appearing in the Vlasov--Maxwell profile alone. Translations preserve the Sobolev norm at each fixed $v$, so the three terms may be estimated separately. The datum term is bounded by $\mathcal M_{N,p}$ by the definition~\eqref{eq:phase-space-matching}, and the longitudinal integral by $\mathcal M_{N,p}$ uniformly in $t$ by~\eqref{eq:VP-longitudinal-decay}, the time integral of $\jb{s}^{-2}$ being finite. For the transverse term, $\jb{v}\abs{\nabla_v\mu(v)}\lesssim\jb{v}^{-4}$ for the Poisson equilibrium~\eqref{Poisson}, so that $\norm{\jb{v}\nabla_v\mu}_{L^p_v}\lesssim1$ for every $1\leq p\leq\infty$, and Lemma~\ref{lem:characteristic-impulse} gives, uniformly in $t$,
\begin{align*}
    \norm{\int_0^t\nabla_v\mu(v)\cdot E_T(s,x+sv)\,ds}_{L^p_vW^{N,\infty}_x}
    \lesssim\ep\mathcal D_{N,p,\delta}\norm{\jb{v}\nabla_v\mu}_{L^p_v}
    \lesssim\ep\mathcal D_{N,p,\delta}.
\end{align*}

\noindent\emph{Conclusion.} Collecting the three bounds, and using once more that translations preserve the Sobolev norm at each fixed $v$,
\begin{align*}
    \norm{f(t)-\fVP(t)}_{L^p_vW^{N,\infty}_x}
    =\norm{g(t)-g^{\mathrm{VP}}(t)}_{L^p_vW^{N,\infty}_x}
    \lesssim\mathcal M_{N,p}+\ep\mathcal D_{N,p,\delta}.
\end{align*}
Letting $t\to\infty$ and using the convergence of the two profiles to their scattering states gives
\begin{align*}
    \norm{f_\infty-f^{\mathrm{VP}}_\infty}_{L^p_vW^{N,\infty}_x}
    \lesssim\mathcal M_{N,p}+\ep\mathcal D_{N,p,\delta}.
\end{align*}
Substituting $\ep=1/c$ proves the displayed estimate.

\noindent\emph{Sharpness of the $c^{-1}$ rate.}
Choose a non-zero real-valued divergence-free $A_{\mathrm{in}}\in\mathcal S(\R^3;\R^3)$ whose Fourier transform is supported in an annulus $\set{r_1\leq\ak\leq r_2}$ with $0<r_1<r_2$, and set
\begin{align*}
\Ein=0,\qquad
\fin=\fVPin=0,\qquad
\Bin=\nabla_x\times A_{\mathrm{in}}
\end{align*}
for every $\ep$. These data satisfy~\eqref{compatibilityconditions}, since $\nabla_x\cdot\Bin=0$ while $\fin=0$ gives both $\rho[\fin]=0=\nabla_x\cdot\Ein$ and zero total mass. Here $\ETin=0$ and $\Jin=0$, while $\Bin$ and $\Delta_x^{-1}\nabla_x\times\Bin=-A_{\mathrm{in}}$ have Fourier transforms supported in $\set{r_1\leq\ak\leq r_2}$, on which $\ak^{-\delta}$ is smooth and bounded. Every term of~\eqref{eq:data-norm} is therefore finite, so that $\mathcal D_{N,p,\delta}$ is finite and independent of $\ep$, while $\mathcal M_{N,p}=0$ and $\fVP_\infty=0$. Moreover, $A^0=A_{\mathrm{in}}$, $A^1=0$ and $S^{\mathrm{cur}}=0$, while $\rho\equiv0$ by uniqueness for~\eqref{VolterraF}, so that $E_L\equiv0$ and $E=E_T$. 

By Corollary~\ref{cor:roots}, for each fixed $\ep$ the real parts of the three roots are bounded above by a negative constant, uniformly for $k$ in the compact set $\supp\widehat A_{\mathrm{in}}\subset\R^3\setminus\set{0}$, and no uniformity in $\ep$ is needed here. Hence $\widehat A(t,k)$, a finite sum of exponentials by~\eqref{eq:solution-representation}, decays exponentially in $t$ uniformly on $\supp\widehat A_{\mathrm{in}}$, and Lemma~\ref{lem:embeddings}\ref{emb:fourier-L1} gives $\norm{A(t)}_{W^{N,\infty}_x}\to0$ as $t\to\infty$. Letting $t\to\infty$ in~\eqref{eq:impulse-identity}, and using~\eqref{eq:profile-limit} with $\fin=0$ and $E=E_T$, therefore gives
\begin{align*}
    \ep^{-1}f_\infty(x,v)
    =-\nabla_v\mu(v)\cdot A_{\mathrm{in}}(x)
    -\int_0^\infty\nabla_v\mu(v)\cdot(v\cdot\nabla_x)A(s,x+sv)\,ds,
\end{align*}
where $f_\infty(x,v)=\lim_{t\to\infty}f(t,x+tv,v)$. The integral term is estimated as in the proof of Lemma~\ref{lem:characteristic-impulse},
\begin{align*}
    \norm{\ep^{-1}f_\infty+\nabla_v\mu\cdot A_{\mathrm{in}}}_{L^p_vW^{N,\infty}_x}
    \lesssim\norm{\abs{v}\nabla_v\mu}_{L^p_v}\int_0^\infty\norm{\nabla_xA(s)}_{W^{N,\infty}_x}\,ds
    \lesssim\int_0^\infty\Lambda_\ep(s)\,ds+\ep,
\end{align*}
which tends to $0$ as $\ep\to0$ by dominated convergence, since $\Lambda_\ep(s)\to0$ for every $s>0$ while $\Lambda_\ep(s)\leq\min\set{1,s^{-3/2}}$. The leading term is not identically zero: $A_{\mathrm{in}}\neq0$ and $\nabla_v\mu(v)$ is a non-zero multiple of $v$ for $v\neq0$, so $v\mapsto\norm{\nabla_v\mu(v)\cdot A_{\mathrm{in}}}_{W^{N,\infty}_x}$ is a continuous function which is positive on a non-empty open set. Since $\fVP_\infty=0$, the triangle inequality gives
\begin{align*}
    \lim_{\ep\to0}\ep^{-1}\norm{f_\infty-\fVP_\infty}_{L^p_vW^{N,\infty}_x}
    =\norm{\nabla_v\mu\cdot A_{\mathrm{in}}}_{L^p_vW^{N,\infty}_x}>0.
\end{align*}
Thus the rate cannot in general be improved.
\end{proof}

\appendix

\section{Root estimates for the transverse dispersion function}\label{appendixroots}
This appendix contains the proofs of Lemmas~\ref{lem:real-root} and~\ref{lem:factorisations}.

\begin{proof}[Proof of Lemma~\ref{lem:real-root}]
Recall from~\eqref{P} and the discussion following it that, for every $r>0$, the zeros of $D(\cdot,r)$ are those of the cubic $P(\cdot,r)$, which has exactly one real root $\lam_0(r)$ and a non-real conjugate pair $\lam_\pm(r)$. Throughout, let $\rho>0$ and set $r=\ep\rho$.

Begin with~\ref{realroot}. By~\eqref{eq:P-derivative}, the restriction of $P(\cdot,r)$ to $\R$ is strictly increasing. Since $P(-r,r)=-\ep^2r<0=P(\lam_0(r),r)$, it follows that $\lam_0(r)>-r$; hence $\theta(\rho)>0$. For $\lam=\ep(y-\rho)$ a direct computation gives
\begin{align}\label{eq:Phi}
    P(\lam,r)=\ep^3 \Phi(y,\rho),
    \qquad
    \Phi(y,\rho):=y\sbr{1+\rho^2+\ep^2(\rho-y)^2}-\rho.
\end{align}
The root $\lam_0(r)$ corresponds to $y=(\lam_0(r)+r)/\ep=\theta(\rho)$. Hence $\theta(\rho)$ is the unique real zero of $\Phi(\cdot,\rho)$.

Evaluating $\Phi(\cdot,\rho)$ at the two endpoints, and using $\ep\leq1$,
\begin{align*}
    \Phi\br{\frac{\rho}{2\jb{\rho}^2},\rho}
    &=-\frac{\rho}{2}+\ep^2 \frac{\rho}{2\jb{\rho}^2}
        \br{\rho-\frac{\rho}{2\jb{\rho}^2}}^2
    \leq-\frac{\rho}{2}+\frac{\rho^3}{2\jb{\rho}^2}
    =-\frac{\rho}{2\jb{\rho}^2}<0,\\
    \Phi\br{\frac{\rho}{\jb{\rho}^2},\rho}
    &=\ep^2 \frac{\rho}{\jb{\rho}^2}
        \br{\rho-\frac{\rho}{\jb{\rho}^2}}^2\geq0.
\end{align*}
This proves~\eqref{realrootupperandlowerbound}.

For the derivative bounds, expand
\begin{align*}
    \Phi(y,\rho)=\ep^2y^3-2\ep^2\rho y^2+a(\rho)y-\rho,
    \qquad
    a(\rho):=1+(1+\ep^2)\rho^2,
\end{align*}
so that $\p_y\Phi(y,\rho)=3\ep^2y^2-4\ep^2\rho y+a(\rho)$. Completing the square,
\begin{align}\label{eq:dyPhi-lower}
    \p_y\Phi(y,\rho)
    =1+\rho^2+\ep^2\sbr{3\br{y-\frac{2\rho}{3}}^2-\frac{\rho^2}{3}}
    \geq1+\frac{2}{3}\rho^2\gtrsim\jb{\rho}^2,
    \qquad y,\rho\in\R,
\end{align}
uniformly in $\ep\in(0,1]$; in particular $\p_y\Phi(\theta(\rho),\rho)$ is bounded below away from zero.

The bounds on $\theta^{(n)}$ follow by induction on $n$, the case $n=0$ being \eqref{realrootupperandlowerbound}. Differentiating the identity $\Phi(\theta(\rho),\rho)=0$ in $\rho$,
\begin{align*}
    \p_y\Phi(\theta(\rho),\rho) \theta'(\rho)=1-a'(\rho)\theta(\rho)+2\ep^2\theta(\rho)^2.
\end{align*}
Here the right-hand side is uniformly bounded: by~\eqref{realrootupperandlowerbound}, $\theta(\rho)\leq\rho/\jb{\rho}^2\leq\jb{\rho}^{-1}$, so that, using $\abs{a'}\lesssim\jb{\rho}$ and $\ep\leq1$,
\begin{align*}
    \abs{a'(\rho)\theta(\rho)}+\ep^2\theta(\rho)^2\lesssim1.
\end{align*}
The left-hand factor $\p_y\Phi(\theta(\rho),\rho)$ is bounded below by a constant multiple of $\jb{\rho}^2$ by~\eqref{eq:dyPhi-lower}, so $\abs{\theta'(\rho)}\lesssim\jb{\rho}^{-2}$, the case $n=1$.

Fix $n\geq2$ and assume $\abs{\theta^{(m)}(\rho)}\lesssim\jb{\rho}^{-m-1}$ for all $0\leq m\leq n-1$. Since $\Phi(\cdot,\rho)$ is cubic, the identity is the polynomial relation
\begin{align*}
    \ep^2\theta(\rho)^3-2\ep^2\rho \theta(\rho)^2+a(\rho)\theta(\rho)-\rho=0.
\end{align*}
On differentiating $n$ times and applying the Leibniz rule monomial by monomial, the terms containing $\theta^{(n)}$ arise only from differentiating one factor $\theta$ exactly $n$ times and leaving the remaining factors undifferentiated. These terms sum to $\br{\p_y\Phi(\theta(\rho),\rho)}\theta^{(n)}(\rho)$. Collecting the remaining terms into $\mathcal R_n$ gives
\begin{align}\label{eq:theta-Rn}
\p_y\Phi(\theta(\rho),\rho)\theta^{(n)}(\rho)
=-\mathcal R_n(\rho),
\end{align}
where $\mathcal R_n$ contains only derivatives of $\theta$ of order at most $n-1$.
Since $a$ is quadratic with $\abs{a'}\lesssim\jb{\rho}$ and
$\abs{a''}\lesssim1$, the terms contributed by $a\theta$ are bounded
by a constant multiple of $\jb{\rho}^{1-n}$.
The induction hypothesis and the Leibniz rule similarly bound the
contributions from $\rho\theta^2$ and $\theta^3$ by constant multiples of
$\jb{\rho}^{-n-1}$ and $\jb{\rho}^{-n-3}$, respectively.
Since $\ep\leq1$, it follows that
$\abs{\mathcal R_n(\rho)}\lesssim\jb{\rho}^{1-n}$,
with an implicit constant depending only on $n$.
Combining~\eqref{eq:theta-Rn} with~\eqref{eq:dyPhi-lower} gives
\begin{align*}
    \abs{\theta^{(n)}(\rho)}
    \lesssim\jb{\rho}^{-2}\jb{\rho}^{1-n}
    =\jb{\rho}^{-n-1}.
\end{align*}
This closes the induction and proves~\ref{realroot}.

For \ref{complexroot-zeta}, Vieta's formulas show that the three roots of $P(\cdot,r)$ sum to $-r$, so
\begin{align*}
    \lam_0(r)+2\gamma(r)=-r.
\end{align*}
Together with $\lam_0(r)=-r+\ep\theta(\rho)$ and $\gamma(r)=\ep\zeta(\rho)$, where $r=\ep\rho$, this gives
\begin{align*}
    \zeta(\rho)=-\frac12\theta(\rho).
\end{align*}
Both bounds in \ref{complexroot-zeta} therefore follow from \ref{realroot}, since $\zeta^{(n)}=-\frac12\theta^{(n)}$ for every $n\geq0$.

It remains to prove~\ref{complexroot-beta}. Vieta's formulas give a second relation between the roots: the sum of their pairwise products equals $(r^2+\ep^2)/\ep^2$, that is,
\begin{align*}
2\gamma(r)\lam_0(r)+\gamma(r)^2+\omega(r)^2=\frac{r^2+\ep^2}{\ep^2}.
\end{align*}
Inserting $r=\ep\rho$, $\lam_0(r)=\ep(\theta(\rho)-\rho)$, $\gamma(r)=-\frac{\ep}{2}\theta(\rho)$ and $\omega(r)=\beta(\rho)$ gives
\begin{align}\label{eq:beta-sq}
\beta(\rho)^2=\jb{\rho}^2+\ep^2\psi(\rho),
\qquad
\psi(\rho):=\theta(\rho)\br{\frac34\theta(\rho)-\rho}.
\end{align}
Before estimating $\beta$, it is extended to $\rho\leq0$. So far $\theta$, and with it $\beta$, has been defined only for $\rho>0$. By~\eqref{eq:dyPhi-lower} the function $\Phi(\cdot,\rho)$ is strictly increasing on $\R$ for every $\rho\in\R$, and $\Phi(y,\rho)\to\pm\infty$ as $y\to\pm\infty$, so it has a unique real zero $\widetilde\theta(\rho)$, now defined for all $\rho\in\R$. Since $\p_y\Phi>0$, the implicit function theorem gives $\widetilde\theta\in C^\infty(\R)$. By~\eqref{eq:Phi}, $\Phi(-y,-\rho)=-\Phi(y,\rho)$, so $-\widetilde\theta(\rho)$ is a zero of $\Phi(\cdot,-\rho)$, and uniqueness gives
\begin{align*}
    \widetilde\theta(-\rho)=-\widetilde\theta(\rho).
\end{align*}
For $\rho>0$ this zero is the $\theta(\rho)$ of~\ref{realroot}, so $\widetilde\theta$ is an odd smooth extension of $\theta$ to $\R$, denoted $\theta$ from now on.

Now set
\begin{align*}
    q(\rho):=\jb{\rho}^2+\ep^2\theta(\rho)\br{\frac34\theta(\rho)-\rho},
    \qquad\rho\in\R.
\end{align*}
Since $\theta$ is odd and smooth, $q$ is smooth and even. By~\eqref{eq:beta-sq}, $q(\rho)=\beta(\rho)^2>0$ for $\rho>0$, and $q(0)=1$ because $\theta(0)=0$. Evenness therefore gives $q>0$ on all of $\R$, so the positive square root
\begin{align*}
    \widetilde\beta(\rho):=\sqrt{q(\rho)}
\end{align*}
is a smooth even function on $\R$ agreeing with $\beta$ on $(0,\infty)$. By the standard smoothness criterion for radial functions, $k\mapsto\widetilde\beta(\ak)$ is then smooth on $\R^3$. This extension is denoted $\beta$ from now on.

For the pointwise bound, $0<\theta(\rho)\leq\rho/\jb{\rho}^2\leq\rho$, so $\frac34\theta(\rho)-\rho\leq 0$. Hence $\psi(\rho)\leq0$ and $\beta(\rho)^2\leq\jb{\rho}^2$. Moreover,
\begin{align*}
\psi(\rho)\geq-\rho\theta(\rho)\geq-\frac{\rho^2}{\jb{\rho}^2}.
\end{align*}
Thus, since $\ep\leq1$,
\begin{align*}
\beta(\rho)^2
\geq\jb{\rho}^2-\frac{\rho^2}{\jb{\rho}^2}
=\frac{1+\rho^2+\rho^4}{\jb{\rho}^2}
\geq\frac34\jb{\rho}^2,
\end{align*}
where the final inequality is equivalent to $(1-\rho^2)^2\geq0$. Therefore
\begin{align}\label{eq:beta-pointwise}
\frac{\sqrt3}{2}\jb{\rho}\leq\beta(\rho)\leq\jb{\rho}.
\end{align}

For the higher derivative bounds, record that, for $n\geq1$,
\begin{align*}
    \br{\rho\theta}^{(n)}=\rho\theta^{(n)}+n\theta^{(n-1)},
\end{align*}
so that $\abs{\br{\rho\theta}^{(n)}}\lesssim\jb{\rho}\jb{\rho}^{-n-1}+\jb{\rho}^{-n}\lesssim\jb{\rho}^{-n}$ by~\ref{realroot}, while every term of the Leibniz expansion $\br{\theta^2}^{(n)}=\sum_{j=0}^n\binom{n}{j}\theta^{(j)}\theta^{(n-j)}$ is $\lesssim\jb{\rho}^{-n-2}$. Applied to $\psi=\frac34\theta^2-\rho\theta$ this gives, for every $n\geq0$,
\begin{align}\label{eq:psi-derivs}
\abs{\psi^{(n)}(\rho)}\lesssim\jb{\rho}^{-n}.
\end{align}
Since $\abs{(\jb{\rho}^2)^{(n)}}\lesssim\jb{\rho}^{2-n}$, differentiating \eqref{eq:beta-sq} gives
\begin{align}\label{eq:beta-square-derivs}
\abs{(\beta^2)^{(n)}(\rho)}\lesssim\jb{\rho}^{2-n},
\qquad n\geq0.
\end{align}
The bound $\abs{\beta^{(n)}(\rho)}\lesssim\jb{\rho}^{1-n}$ follows by induction. The case $n=0$ is \eqref{eq:beta-pointwise}. For $n\geq1$, differentiating $\beta^2$ exactly $n$ times gives
\begin{align*}
2\beta(\rho)\beta^{(n)}(\rho)
=(\beta^2)^{(n)}(\rho)
-\sum_{j=1}^{n-1}\binom{n}{j}\beta^{(j)}(\rho)\beta^{(n-j)}(\rho).
\end{align*}
The induction hypothesis and \eqref{eq:beta-square-derivs} bound the right-hand side by $\jb{\rho}^{2-n}$. Dividing by $2\beta(\rho)\gtrsim\jb{\rho}$ gives
\begin{align}\label{eq:beta-general-derivs}
\abs{\beta^{(n)}(\rho)}\lesssim\jb{\rho}^{1-n},
\qquad n\geq0.
\end{align}
This is the general derivative bound in~\ref{complexroot-beta}.

For the sharp bound on $\beta'$, differentiating \eqref{eq:beta-sq} gives
\begin{align}\label{eq:beta-prime-identity}
\beta(\rho)\beta'(\rho)=\rho+\frac{\ep^2}{2}\psi'(\rho).
\end{align}
Since
\begin{align*}
\psi'(\rho)=\frac32\theta(\rho)\theta'(\rho)-\theta(\rho)-\rho\theta'(\rho),
\end{align*}
the bounds in~\ref{realroot} give
\begin{align}\label{eq:psi-prime-sharp}
\abs{\psi'(\rho)}
\leq K\frac{\rho}{\jb{\rho}^2}
\leq K\rho
\end{align}
for an absolute constant $K>0$. Choose $\ep_0\in(0,1]$ so that $\ep_0^2K\leq1$. Then, for $\ep\in(0,\ep_0]$, the perturbation in \eqref{eq:beta-prime-identity} is at most $\rho/2$, and hence
\begin{align*}
\frac{\rho}{2}
\leq\beta(\rho)\beta'(\rho)
\leq\frac{3\rho}{2}.
\end{align*}
In particular $\beta'(\rho)>0$. Dividing by $\beta(\rho)$ and using \eqref{eq:beta-pointwise} gives
\begin{align*}
\frac12\frac{\rho}{\jb{\rho}}
\leq\beta'(\rho)
\leq\sqrt3\frac{\rho}{\jb{\rho}}.
\end{align*}

For $\beta''$, differentiating \eqref{eq:beta-prime-identity} gives
\begin{align}\label{eq:beta-double-identity}
\beta(\rho)\beta''(\rho)+\beta'(\rho)^2
=1+\frac{\ep^2}{2}\psi''(\rho).
\end{align}
Multiplying \eqref{eq:beta-double-identity} by $\beta(\rho)^2$, and using
\begin{align*}
\beta(\rho)^2\beta'(\rho)^2
=\br{\rho+\frac{\ep^2}{2}\psi'(\rho)}^2
\end{align*}
from \eqref{eq:beta-prime-identity}, together with \eqref{eq:beta-sq}, gives
\begin{align*}
\beta(\rho)^3\beta''(\rho)
&=\beta(\rho)^2
-\br{\rho+\frac{\ep^2}{2}\psi'(\rho)}^2
+\frac{\ep^2}{2}\beta(\rho)^2\psi''(\rho)\\
&=1+\ep^2\sbr{
\psi(\rho)-\rho\psi'(\rho)
+\frac12\beta(\rho)^2\psi''(\rho)
}
-\frac{\ep^4}{4}\psi'(\rho)^2.
\end{align*}
By \eqref{eq:psi-derivs}, \eqref{eq:psi-prime-sharp}, and \eqref{eq:beta-pointwise},
\begin{align*}
\abs{\psi(\rho)}
+\abs{\rho\psi'(\rho)}
+\abs{\beta(\rho)^2\psi''(\rho)}
+\abs{\psi'(\rho)^2}
\leq M
\end{align*}
for an absolute constant $M>0$. Hence the perturbative part of the right-hand side is bounded in absolute value by $M\ep^2$. Shrinking $\ep_0$ if necessary so that $M\ep_0^2\leq\frac12$, for all $\ep\in(0,\ep_0]$,
\begin{align*}
\frac12
\leq\beta(\rho)^3\beta''(\rho)
\leq\frac32.
\end{align*}
Thus $\beta''(\rho)>0$. Dividing by $\beta(\rho)^3$ and using \eqref{eq:beta-pointwise} gives
\begin{align*}
\frac{1}{2\jb{\rho}^3}
\leq\beta''(\rho)
\leq\frac{4}{\sqrt3\jb{\rho}^3}.
\end{align*}
Together with \eqref{eq:beta-pointwise}, \eqref{eq:beta-general-derivs}, and the bound for $\beta'$, this proves~\ref{complexroot-beta}.
\end{proof}

\begin{proof}[Proof of Lemma~\ref{lem:factorisations}]
As established in the proof of Lemma~\ref{lem:real-root}, $\lam_0(\ep\rho)=\ep(\theta(\rho)-\rho)$ and $\gamma(\ep\rho)=-\frac{\ep}{2}\theta(\rho)$. With $\omega(\ep\rho)=\beta(\rho)$, the root differences~\eqref{eq:pq} read
\begin{align*}
    p_\pm(\rho)&=\ep u(\rho)\pm i\beta(\rho),
    &
    q_\pm(\rho)&=\ep w(\rho)\pm i\beta(\rho),
    \\
    u(\rho)&:=\rho-\frac12\theta(\rho),
    &
    w(\rho)&:=\rho-\frac32\theta(\rho).
\end{align*}
Since $\theta$ and $\beta$ are real-valued, so are $u$ and $w$; thus $\ep u(\rho)$ and $\ep w(\rho)$ are the real parts of $p_\pm(\rho)$ and $q_\pm(\rho)$, and $\pm\beta(\rho)$ their imaginary parts.

By Lemma~\ref{lem:real-root}\ref{realroot} and~\ref{complexroot-beta},
\begin{align*}
    \abs{\theta^{(m)}(\rho)}\leq C_m \jb{\rho}^{-m-1},
    \qquad
    \abs{\beta^{(m)}(\rho)}\leq C_m \jb{\rho}^{1-m},
    \qquad m\geq0,
\end{align*}
with constants $C_m$ independent of $\ep$ and $\rho$. Hence $u$ and $w$ satisfy $\abs{u^{(m)}}+\abs{w^{(m)}}\lesssim\jb{\rho}^{1-m}$; this uses $\abs{u},\abs{w}\lesssim\jb{\rho}$ for $m=0$, $\abs{u'},\abs{w'}\lesssim1$ for $m=1$, and $u^{(m)}=-\frac12\theta^{(m)}$, $w^{(m)}=-\frac32\theta^{(m)}$ for $m\geq2$. Consequently, for $\ep\leq1$, \eqref{eq:Rpm-factor-derivs} holds.

For the lower bounds, Lemma~\ref{lem:real-root}\ref{complexroot-beta} gives $\beta(\rho)\geq c_\beta\jb{\rho}$, so $\abs{1/\beta(\rho)}\leq\jb{\rho}^{-1}/c_\beta$; since $\abs{q_\pm(\rho)}^2=\ep^2w(\rho)^2+\beta(\rho)^2\geq\beta(\rho)^2$, this proves~\eqref{eq:q-beta-lower}.

The reciprocal bounds~\eqref{eq:reciprocal-derivs} follow by induction on $n$. The case $n=0$ is \eqref{eq:q-beta-lower}. For $n\geq1$, differentiating $(1/q_\pm)q_\pm=1$,
\begin{align*}
    \br{\frac{1}{q_\pm}}^{(n)}=-\frac{1}{q_\pm}\sum_{j=1}^n\binom{n}{j}q_\pm^{(j)}\br{\frac{1}{q_\pm}}^{(n-j)},
\end{align*}
and \eqref{eq:Rpm-factor-derivs} with the induction hypothesis give $\abs{q_\pm^{(j)}(1/q_\pm)^{(n-j)}}\lesssim\jb{\rho}^{1-j}\jb{\rho}^{-(n-j)-1}=\jb{\rho}^{-n}$ for $1\leq j\leq n$. The sum has $n$ terms and its binomial coefficients depend only on $n$, so both are absorbed into a constant depending on $n$ alone; combined with $\abs{1/q_\pm}\lesssim\jb{\rho}^{-1}$, this gives $\abs{(1/q_\pm)^{(n)}}\lesssim\jb{\rho}^{-n-1}$, with a constant independent of $\ep$ and $\rho$. The bound on $(1/\beta)^{(n)}$ is identical.
\end{proof}
\section{Oscillatory-integral estimates}\label{app:dispersive}

This appendix establishes the general dispersive estimates used in the proof of Lemma~\ref{lem:KG-dispersive}. For the standard Klein--Gordon phase, kernel estimates at low and high frequencies are classical~\cite{GinibreVelo1985}. For the phase considered here, the argument follows Han-Kwan, Nguyen and Rousset~\cite[Appendix~E]{HanKwanNguyenRousset2025LinearVM}. Unlike their estimates, the bounds here retain the intermediate wave branch. The constants are also tracked uniformly in $\ep$, and the frequency-localised bounds are summed in Besov norms with dyadic summability index~$1$.

\subsection{Auxiliary oscillatory-integral estimates}\label{app:aux}
A radial function on $\R^3$ and its profile on $[0,\infty)$ are denoted by the same symbol, so that $\vp(k)=\vp(\ak)$. For a non-negative integer $N$ and a set $S\subset\R^3$ on which $u$ has continuous derivatives up to order $N$, $\norm{u}_{C^N(S)}:=\max_{\aal\leq N}\sup_{\xi\in S}\abs{\p^\alpha u(\xi)}$; the set is omitted when it is $\R^3$.

The one-dimensional input is the second-order van der Corput estimate with an amplitude. In the applications the lower bound on the second derivative of the phase varies with the frequency scale, so the estimate is used in the following form, in which the dependence on that lower bound is displayed.

\begin{proposition}[van der Corput estimate with amplitude]\label{prop:vdc-weighted}
Let $\psi\in C^2(a,b)$ be real-valued with $\abs{\psi''(x)}\geq M>0$ for all $x\in(a,b)$, and let $f\in C^1\br{[a,b]}$. Then, for all $t>0$,
    \begin{align}\label{eq:vdc-weighted}
        \abs{\int_a^be^{it\psi(x)}f(x)\,dx}
        \lesssim\frac{1}{\sqrt{Mt}}\sbr{\abs{f(b)}+\int_a^b\abs{f'(x)}\,dx},
    \end{align}
with an absolute implicit constant, independent of $\psi$, of $M$, of $t$ and of the interval $(a,b)$.
\end{proposition}

\begin{proof}
Stein~\cite{stein-harmonic-analysis}, Chapter~VIII, \S1.2, Proposition~2 with $k=2$, together with the corollary following it, gives the normalised estimate
    \begin{align*}
        \abs{\int_a^be^{i\Lambda u(x)}f(x)\,dx}
        \lesssim\frac{1}{\sqrt{\Lambda}}\sbr{\abs{f(b)}+\int_a^b\abs{f'(x)}\,dx},
        \qquad\Lambda>0,
    \end{align*}
for real-valued $u\in C^2(a,b)$ with $\abs{u''}\geq1$, the constant there being independent of $u$, of $\Lambda$ and of the interval, and the argument uses only two derivatives of the phase. Applying this to $u=\psi/M$, which satisfies $\abs{u''}\geq1$, with $\Lambda=Mt$ yields~\eqref{eq:vdc-weighted}.
\end{proof}

A three-dimensional stationary-phase estimate is needed at low frequencies. The following quantitative version of the standard stationary-phase theorem records the dependence of the implicit constant needed in the application; see, for example, H\"ormander~\cite[Theorem~7.7.5]{Hormander1990}.

\begin{proposition}[Quantitative stationary phase]\label{prop:stationary-phase}
Let $B\subset\R^3$ be an open ball of radius $R$, let $\kappa>0$, and let $\psi\in C^3(B)$ be real-valued with $\norm{\psi}_{C^3(B)}$ finite and
    \begin{align}\label{eq:hessian-lower}
        D^2\psi(\xi)\geq\kappa \mathrm{Id},\qquad\xi\in B,
    \end{align}
and suppose that $\nabla\psi(\xi_0)=0$ for some $\xi_0\in B$. Then, for every $a\in C_c^2(B)$ and every $t>0$,
    \begin{align}\label{eq:stationary-phase}
        \abs{\int_{B}e^{it\psi(\xi)}a(\xi)\,d\xi}
        \lesssim\frac{1}{t^{\frac{3}{2}}}\norm{a}_{C^2(\R^3)},
    \end{align}
with implicit constant depending only on $\kappa$, on $R$ and on $\norm{\psi}_{C^3(B)}$.
\end{proposition}

\begin{proof}
Since $B$ is convex, \eqref{eq:hessian-lower} and $\nabla\psi(\xi_0)=0$ give
    \begin{align*}
        \abs{\nabla\psi(\xi)}\abs{\xi-\xi_0}
        \geq\br{\nabla\psi(\xi)-\nabla\psi(\xi_0)}\cdot(\xi-\xi_0)
        \geq\kappa\abs{\xi-\xi_0}^2,
    \end{align*}
and therefore
    \begin{align}\label{eq:gradient-lower}
        \abs{\nabla\psi(\xi)}\geq\kappa\abs{\xi-\xi_0},\qquad\xi\in B.
    \end{align}
The profiles $\chi$ and $\vp$ of the fixed Littlewood--Paley partition serve here as bump functions, at the scale $\rho$ and centred at $\xi_0$: $\chi(\cdot/2)$ equals $1$ on $\set{\abs{\xi}\leq1}$ and is supported in $\set{\abs{\xi}\leq2}$, and $\vp=\chi(\cdot/2)-\chi$ is supported in $\set{1/2\leq\abs{\xi}\leq2}$. For $\rho>0$ write $\sigma_m:=2^m\rho$, and choose the smallest integer $M\geq0$ such that $2^{M+1}\rho\geq4R$. The telescoping identity $\chi(\xi/2)+\sum_{m=1}^M\vp(2^{-m}\xi)=\chi(2^{-(M+1)}\xi)$ then shows that the functions
    \begin{align*}
        a_0:=a \chi\br{\frac{\cdot-\xi_0}{2\rho}},\qquad
        a_m:=a \vp\br{\frac{\cdot-\xi_0}{\sigma_m}},\quad 1\leq m\leq M,
    \end{align*}
satisfy $a=a_0+\sum_{m=1}^Ma_m$. Here $\supp a_m\subset\set{\sigma_m/2\leq\abs{\xi-\xi_0}\leq2\sigma_m}$, and $a_m$ vanishes identically unless $\sigma_m\leq4R$, so that
    \begin{align}\label{eq:am-derivatives}
        \norm{\p^\alpha a_m}_{L^\infty}\lesssim\frac{\norm{a}_{C^2}}{\sigma_m^{\abs{\alpha}}},
        \qquad1\leq m\leq M,\ \abs{\alpha}\leq2,
    \end{align}
uniformly in $m$, with implicit constant depending on $R$.

The first piece is bounded by the measure of its support,
    \begin{align*}
        \abs{\int_{B}e^{it\psi}a_0\,d\xi}\lesssim\norm{a}_{L^\infty}\rho^3.
    \end{align*}
For $m\geq1$, set $V:=\nabla\psi/\abs{\nabla\psi}^2$, which is of class $C^2$ on $B\setminus\set{\xi_0}$, and $Sv:=-(it)^{-1}\nabla\cdot(Vv)$. Since $(it)^{-1}V\cdot\nabla e^{it\psi}=e^{it\psi}$, integrating by parts twice gives
    \begin{align*}
        \int_{B}e^{it\psi}a_m\,d\xi=\int_{B}e^{it\psi}S^2a_m\,d\xi.
    \end{align*}
On $\supp a_m$ the bound~\eqref{eq:gradient-lower} reads $\abs{\nabla\psi}\geq\kappa\sigma_m/2$, whence
    \begin{align*}
        \abs{V}\lesssim\frac{1}{\sigma_m},\qquad
        \abs{\nabla V}\lesssim\frac{1}{\sigma_m^2},\qquad
        \abs{\nabla^2V}\lesssim\frac{1}{\sigma_m^3},
    \end{align*}
with implicit constants depending only on $\kappa$, $R$ and $\norm{\psi}_{C^3(B)}$; the bound on $\nabla^2V$ uses $\sigma_m\leq4R$ to absorb a term of size $\sigma_m^{-2}$ into $\sigma_m^{-3}$. Expanding $S^2a_m$ produces the four terms $V^2\p^2a_m$, $V(\nabla V)\p a_m$, $(\nabla V)^2a_m$ and $V(\nabla^2V)a_m$, each divided by $t^2$; by the display above and~\eqref{eq:am-derivatives}, every one of them is bounded by $t^{-2}\sigma_m^{-4}\norm{a}_{C^2}$ up to such a constant. Since $\supp a_m$ has measure $\lesssim\sigma_m^3$,
    \begin{align*}
        \abs{\int_{B}e^{it\psi}a_m\,d\xi}
        \lesssim\sigma_m^3\frac{\norm{a}_{C^2}}{t^2\sigma_m^4}
        =\frac{\norm{a}_{C^2}}{t^2\sigma_m}.
    \end{align*}
Summing the geometric series $\sum_{m\geq1}\sigma_m^{-1}\lesssim\rho^{-1}$ and adding the first piece,
    \begin{align*}
        \abs{\int_{B}e^{it\psi}a\,d\xi}
        \lesssim\norm{a}_{C^2}\br{\rho^3+\frac{1}{t^2\rho}}.
    \end{align*}
The choice $\rho:=t^{-1/2}$ gives~\eqref{eq:stationary-phase}.
\end{proof}

The radial Fourier formula used below introduces the kernel $4\pi\sin(r)/r$. The following lemma decomposes this kernel into the two oscillatory factors $e^{\pm ir}$ with smooth amplitudes whose derivatives satisfy symbol-type decay. Fix once and for all a cutoff $\chi_0\in C^\infty\br{[0,\infty),\R}$ with $\chi_0(r)=1$ for $r\leq1/2$ and $\chi_0(r)=0$ for $r\geq1$.

\begin{lemma}[Radial kernel decomposition]
Define
    \begin{align*}
        Z(r):=-\frac{2\pi i}{r}\br{1-\frac{\chi_0(r)}{2}\br{1+e^{-2ir}}},\qquad r>0.
    \end{align*}
Then $Z$ extends to a smooth function on $[0,\infty)$ which satisfies, for every $j\geq0$,
    \begin{align}\label{eq:Z-decay}
        \abs{\p_r^jZ(r)}\lesssim_j\frac{1}{(1+r)^{1+j}},
    \end{align}
and, for all $r\geq0$, with $\sin(r)/r$ understood by continuous extension, with value $1$ at $r=0$,
    \begin{align}\label{eq:Z-identity}
        e^{ir}Z(r)+e^{-ir}\overline{Z(r)}=\frac{4\pi\sin(r)}{r}.
    \end{align}
\end{lemma}

\begin{proof}
For $r>0$ the function $Z$ is smooth by construction. For $0<r<1/2$, where $\chi_0(r)=1$,
    \begin{align*}
        Z(r)=-\frac{2\pi i}{r}\br{1-\frac{1}{2}\br{1+e^{-2ir}}}
        =-\frac{\pi i}{r}\br{1-e^{-2ir}}
        =2\pi e^{-ir}\frac{\sin(r)}{r}.
    \end{align*}
Since $r\mapsto\sin(r)/r$ extends smoothly to $r=0$, this representation shows that $Z$ extends smoothly at the origin and that $\abs{\p_r^jZ(r)}\lesssim_j1$ for $0\leq r\leq1/2$. For $r\geq1$, where $\chi_0(r)=0$, the formula reduces to $Z(r)=-2\pi i/r$, so that $\abs{\p_r^jZ(r)}\lesssim_jr^{-1-j}$. On the transition region $r\in[1/2,1]$ the function $Z$ is smooth and $(1+r)^{-1-j}\sim1$, which completes~\eqref{eq:Z-decay}. The identity~\eqref{eq:Z-identity} follows for $r>0$ by direct computation, $\chi_0$ being real-valued, and extends to $r=0$ by continuity.
\end{proof}

\subsection{Dyadic frequencies}\label{app:dyadic}
The hypotheses below are modelled on a Klein--Gordon dispersion relation $\beta(r)\sim\jb{r}$: the bounds on $\beta',\beta'',\beta'''$ are exactly those satisfied by $\sqrt{1+r^2}$, and the last branch of the lemma, $\lam^{5/2}/t^{3/2}$, is the three-dimensional Klein--Gordon dyadic dispersive rate.

\begin{lemma}[Dyadic dispersive estimate]\label{lem:dyadic-dispersive}
Let $\beta\in C^3\br{[0,\infty),\R}$ and suppose there exist positive constants $c_1,C_1,c_2,C_2,C_3$ such that, for all $r\geq0$,
    \begin{align*}
        \frac{c_1r}{\jb{r}}\leq\beta'(r)\leq\frac{C_1r}{\jb{r}},\qquad
        \frac{c_2}{\jb{r}^3}\leq\beta''(r)\leq\frac{C_2}{\jb{r}^3},\qquad
        \abs{\beta'''(r)}\leq\frac{C_3}{\jb{r}^2}.
    \end{align*}
Let $\vp\in C_c^2(\R^3)$ be radial with $\supp\vp\subset\set{\frac14\leq\abs{\xi}\leq4}$. Then, for every $\lam\geq1$ and $t>0$,
    \begin{align}\label{eq:disp-bound}
        \norm{\int_{\R^3}e^{\pm i\beta(\ak)t+ix\cdot k}\vp\br{\frac{k}{\lam}}\,dk}_{L^\infty}
        \lesssim\min\set{\lam^3,\frac{\lam^2}{t},\frac{\lam^{\frac{5}{2}}}{t^{\frac{3}{2}}}},
    \end{align}
where the implicit constant depends only on $c_1,C_1,c_2,C_2,C_3$ and on $\norm{\vp}_{C^2(\R^3)}$.
\end{lemma}

\begin{proof}
It suffices to treat the $+\beta$ case, the $-\beta$ case being identical. Fix $x\in\R^3$ and $t>0$. The kernel $\sin(\ax r)/(\ax r)$ below is again understood by continuous extension, since the $L^\infty_x$ bound includes $x=0$. The radial Fourier identity and~\eqref{eq:Z-identity} give
    \begin{align*}
        \int_{\R^3}e^{i\beta(\ak)t+ix\cdot k}\vp\br{\frac{k}{\lam}}\,dk
        &=\int_0^\infty e^{i\beta(r)t}\vp\br{\frac{r}{\lam}}\frac{4\pi\sin(\ax r)}{\ax r}r^2\,dr\\
        &=\int_0^\infty e^{i\beta(r)t}\vp\br{\frac{r}{\lam}}
          \br{e^{i\ax r}Z(\ax r)+e^{-i\ax r}\overline{Z(\ax r)}}r^2\,dr.
    \end{align*}
Setting $Z_+=Z$ and $Z_-=\overline Z$, and substituting $r\mapsto\lam r$,
    \begin{align*}
        \int_{\R^3}e^{i\beta(\ak)t+ix\cdot k}\vp\br{\frac{k}{\lam}}\,dk
        =\lam^3\sum_\pm\int_{\frac{1}{4}}^4e^{it\psi^\pm_\lam(\ax,r)}Z_\pm(\ax\lam r)\vp(r)r^2\,dr,
    \end{align*}
where the phase is
    \begin{align*}
        \psi^\pm_\lam(\ax,r):=\beta(\lam r)\pm\frac{\ax}{t}\lam r,
    \end{align*}
with $\p_r\psi^\pm_\lam=\lam\beta'(\lam r)\pm\frac{\ax}{t}\lam$ and $\p_r^2\psi^\pm_\lam=\lam^2\beta''(\lam r)$.

For $r\in[\frac14,4]$ and $\lam\geq1$ one has $\lam r\geq\frac14$ and $\jb{\lam r}\sim\lam$. Since $s\mapsto s/\jb{s}$ is increasing and takes values in $[\frac{1/4}{\jb{1/4}},1)$ on $[\frac14,\infty)$, the hypotheses give $\beta'(\lam r)\sim1$, $\abs{\beta''(\lam r)}\lesssim\lam^{-3}$ and $\abs{\beta'''(\lam r)}\lesssim\lam^{-2}$. Moreover, \eqref{eq:Z-decay} and the chain rule yield, for $j\leq2$,
    \begin{align*}
        \abs{\p_r^j\sbr{Z_\pm(\lam\ax r)}}
        =(\lam\ax)^j\abs{Z_\pm^{(j)}(\lam\ax r)}
        \lesssim\frac{(\lam\ax)^j}{(1+\lam\ax r)^{1+j}}\lesssim1,
    \end{align*}
so that $g(r):=Z_\pm(\lam\ax r)\vp(r)r^2$ satisfies
    \begin{align}\label{eq:g-bound}
        \abs{g(r)}+\abs{g'(r)}+\abs{g''(r)}\lesssim\norm{\vp}_{C^2}.
    \end{align}
Bounding the radial integral directly by~\eqref{eq:g-bound} gives $\abs{\lam^3\int_{1/4}^4e^{it\psi^\pm_\lam}g(r)\,dr}\lesssim\lam^3$ for both signs and every $x$, which is the first branch of~\eqref{eq:disp-bound}; the two dispersive branches are obtained in two cases.

\medskip \noindent\emph{Case 1: $\p_r\psi^\pm_\lam$ bounded away from zero.} Suppose $\abs{\p_r\psi^\pm_\lam(\ax,r)}\gtrsim\lam\br{1+\frac{\ax}{t}}$ on $[\frac14,4]$. Because $\vp$ and $\vp'$ vanish at the endpoint radii, integrating by parts twice in $r$ produces no boundary terms and gives
    \begin{align*}
        \lam^3\int_{\frac{1}{4}}^4e^{it\psi^\pm_\lam}g(r)\,dr
        =-\frac{\lam^3}{t^2}\int_{\frac{1}{4}}^4e^{it\psi^\pm_\lam}
          \p_r\br{\frac{1}{\p_r\psi^\pm_\lam}\p_r\sbr{\frac{1}{\p_r\psi^\pm_\lam}g(r)}}\,dr,
    \end{align*}
where, for any real-valued $h\in C^3$ with $h'$ non-vanishing,
    \begin{align*}
        \p_r\br{\frac{1}{h'}\p_r\sbr{\frac{1}{h'}g}}
        =\br{\frac{3(h'')^2}{(h')^4}-\frac{h'''}{(h')^3}}g
        -\frac{3h''}{(h')^3}g'+\frac{1}{(h')^2}g''.
    \end{align*}
For $r\in[\frac14,4]$ and $\lam\geq1$, the derivative bounds on $\beta$ yield
    \begin{align*}
        \abs{\p_r^2\psi^\pm_\lam}=\lam^2\abs{\beta''(\lam r)}\lesssim\lam^{-1},\qquad
        \abs{\p_r^3\psi^\pm_\lam}=\lam^3\abs{\beta'''(\lam r)}\lesssim\lam,
    \end{align*}
while $\abs{\p_r\psi^\pm_\lam}\gtrsim\lam\br{1+\frac{\ax}{t}}\geq\lam$. Each of the four coefficients in the expansion above, evaluated at $h=\psi^\pm_\lam$, is therefore bounded by $\lam^{-2}$, so that, with~\eqref{eq:g-bound},
    \begin{align*}
        \abs{\lam^3\int_{\frac{1}{4}}^4e^{it\psi^\pm_\lam}g(r)\,dr}\lesssim\frac{\lam}{t^2}.
    \end{align*}
The geometric means of this bound and of the direct bound $\lam^3$, taken with equal weights and with weights $3/4$ and $1/4$ respectively, give
    \begin{align}\label{eq:disp-nonstationary}
        \abs{\lam^3\int_{\frac{1}{4}}^4e^{it\psi^\pm_\lam}g(r)\,dr}
        \lesssim\min\set{\br{\frac{\lam}{t^2}}^{\frac{1}{2}}\br{\lam^3}^{\frac{1}{2}},\br{\frac{\lam}{t^2}}^{\frac{3}{4}}\br{\lam^3}^{\frac{1}{4}}}
        =\min\set{\frac{\lam^2}{t},\frac{\lam^{\frac{3}{2}}}{t^{\frac{3}{2}}}}.
    \end{align}
Using $\lam\geq1$, together with the direct bound, yields~\eqref{eq:disp-bound}.

\medskip \noindent\emph{Application of Case 1.} The $\psi^+_\lam$ term always falls in Case 1: since $\beta'(\lam r)\geq c_1\frac{\lam r}{\jb{\lam r}}\gtrsim1$ on $[\frac14,4]$,
    \begin{align*}
        \p_r\psi^+_\lam(\ax,r)\gtrsim\lam\br{1+\frac{\ax}{t}}.
    \end{align*}
For the $\psi^-_\lam$ term, since $r\mapsto r/\jb{r}$ is increasing on $[0,\infty)$, the hypotheses provide $c_1'>0$ with $c_1'\leq\beta'(\lam r)\leq C_1$ for all $r\in[\frac14,4]$ and $\lam\geq1$. Hence, if either $\frac{\ax}{t}\geq3C_1$ or $\frac{\ax}{t}\leq\frac{c_1'}{3}$, then
    \begin{align*}
        \abs{\p_r\psi^-_\lam(\ax,r)}\gtrsim\lam\br{1+\frac{\ax}{t}},
    \end{align*}
and \eqref{eq:disp-nonstationary} applies.

\medskip \noindent\emph{Case 2: $\frac{c_1'}{3}\leq\frac{\ax}{t}\leq3C_1$ in the $\psi^-_\lam$ term.} In this regime a critical point of $\psi^-_\lam$ may lie in $[\frac14,4]$. By~\eqref{eq:Z-decay} and $\ax\sim t$ in this regime, for $r\in[\frac14,4]$,
    \begin{align*}
        \abs{g(r)}
        \lesssim\frac{1}{1+\lam\ax r}
        \lesssim\frac{1}{\lam t},
        \qquad
        \abs{g'(r)}
        \lesssim\frac{\lam\ax}{(1+\lam\ax r)^2}+\frac{1}{1+\lam\ax r}
        \lesssim\frac{1}{\lam t}.
    \end{align*}
Estimating the radial integral in absolute value therefore gives
    \begin{align*}
        \abs{\lam^3\int_{\frac{1}{4}}^4e^{it\psi^-_\lam}g(r)\,dr}
        \leq\lam^3\int_{\frac{1}{4}}^4\abs{g(r)}\,dr
        \lesssim\frac{\lam^2}{t}.
    \end{align*}
For the remaining branch, the second-derivative lower bound is used: for $r\in[\frac14,4]$ and $\lam\geq1$, $\jb{\lam r}\sim\lam$, so that
    \begin{align*}
        \abs{\p_r^2\psi^-_\lam(\ax,r)}=\lam^2\abs{\beta''(\lam r)}
        \geq c_2\frac{\lam^2}{\jb{\lam r}^3}\gtrsim\frac{1}{\lam}.
    \end{align*}
Since $g(4)=0$, because $\vp(4)=0$, Proposition~\ref{prop:vdc-weighted} and the bound on $g'$ give
    \begin{align*}
        \abs{\lam^3\int_{\frac{1}{4}}^4e^{it\psi^-_\lam}g(r)\,dr}
        \lesssim\frac{\lam^3}{\sqrt{\frac{t}{\lam}}}\int_{\frac{1}{4}}^4\abs{g'(r)}\,dr
        \lesssim\frac{\lam^3}{\sqrt{\frac{t}{\lam}}}\frac{1}{\lam t}
        =\frac{\lam^{\frac{5}{2}}}{t^{\frac{3}{2}}}.
    \end{align*}
Together with the direct bound $\lam^3$, this gives~\eqref{eq:disp-bound} in Case 2 as well.
\end{proof}

\begin{corollary}[Dyadic dispersive estimate, $\ep$-scaled]\label{cor:dispersive-eps}
Under the hypotheses of Lemma~\ref{lem:dyadic-dispersive}, for all $\ep\in(0,1]$, $\lam\geq\ep$ and $t>0$,
    \begin{align*}
        \norm{\int_{\R^3}e^{\pm i\beta(\frac{\ak}{\ep})t+ix\cdot k}\vp\br{\frac{k}{\lam}}\,dk}_{L^\infty}
        \lesssim\min\set{\lam^3,\frac{\ep\lam^2}{t},\frac{\ep^{\frac{1}{2}}\lam^{\frac{5}{2}}}{t^{\frac{3}{2}}}},
    \end{align*}
with constant as in Lemma~\ref{lem:dyadic-dispersive}.
\end{corollary}

\begin{proof}
The change of variables $k\mapsto \ep k$ gives
    \begin{align*}
        \int_{\R^3}e^{\pm i\beta\br{\frac{\ak}{\ep}}t+ix\cdot k}\vp\br{\frac{k}{\lam}}\,dk
        =\ep^3\int_{\R^3}e^{\pm i\beta(\abs{k})t+i(\ep x)\cdot k}
          \vp\br{\frac{k}{\frac{\lam}{\ep}}}\,dk.
    \end{align*}
Lemma~\ref{lem:dyadic-dispersive} applies with frequency parameter $\lam/\ep\geq 1$, and multiplying its bound by the Jacobian $\ep^3$ yields $\ep^3\min\{(\lam/\ep)^3,(\lam/\ep)^2/t,(\lam/\ep)^{5/2}/t^{3/2}\}=\min\{\lam^3,\ep\lam^2/t,\ep^{1/2}\lam^{5/2}/t^{3/2}\}$.
\end{proof}

\subsection{Low frequencies}\label{app:lowfreq}

On the frequency ball the phase $\beta(\ak)t+x\cdot k$ may have a critical point, and Proposition~\ref{prop:stationary-phase} is used there in place of the radial reduction.

\begin{lemma}[Low-frequency dispersive estimate]\label{lem:low-frequency-dispersive}
Let $\beta\in C^\infty\br{[0,\infty)}$ be such that the radial function $k\mapsto\beta(\ak)$ is smooth on $\R^3$, and suppose there exist constants $c_1,c_2>0$ such that, for all $r\in[0,6]$,
    \begin{align*}
        \beta'(r)\geq\frac{c_1r}{\jb{r}},\qquad
        \beta''(r)\geq\frac{c_2}{\jb{r}^3}.
    \end{align*}
Let $\vartheta\in C_c^2(\R^3)$ be radial with $\supp\vartheta\subset\set{\ak\leq4}$. Then, for all $t>0$ and $x\in\R^3$,
    \begin{align*}
        \abs{\int_{\R^3}e^{\pm i\beta(\ak)t+ix\cdot k}\vartheta(k)\,dk}
        \lesssim\frac{1}{t^{\frac{3}{2}}},
    \end{align*}
where the implicit constant depends only on $c_1$, $c_2$, $\norm{\beta(\abs{\cdot})}_{C^3(\set{\ak\leq6})}$ and $\norm{\vartheta}_{C^2(\R^3)}$.
\end{lemma}

\begin{proof}
It suffices to treat the $+\beta$ case: complex conjugation turns the $-\beta$ integral into the $+\beta$ integral with $x$ replaced by $-x$ and $\vartheta$ by its complex conjugate, which changes neither the hypotheses nor the $C^2$ norm. Fix $x\in\R^3$ and $t>0$, write $y:=x/t$, and set
    \begin{align*}
        \phi_y(k):=\beta(\ak)+y\cdot k,
    \end{align*}
which is of class $C^3$ on the ball $\set{\ak<6}$ because $k\mapsto\beta(\ak)$ is smooth. Its Hessian is independent of $y$ and, for $k\neq0$, is given by
    \begin{align*}
        D^2\phi_y(k)_{ij}
        =\beta''(\ak)\frac{k_ik_j}{\ak^2}
        +\frac{\beta'(\ak)}{\ak}\br{\delta_{ij}-\frac{k_ik_j}{\ak^2}}.
    \end{align*}
Its eigenvalues are $\beta''(\ak)$ in the radial direction and $\beta'(\ak)/\ak$ in each of the two tangential directions. At the origin, $D^2\phi_y(0)=\beta''(0)\mathrm{Id}$. The hypotheses give $\beta''(r)\geq c_2\jb{6}^{-3}$ and $\beta'(r)/r\geq c_1\jb{6}^{-1}$ for $0<r\leq6$, so that, by continuity at the origin,
    \begin{align}\label{eq:low-freq-hessian}
        D^2\phi_y(k)\geq\kappa \mathrm{Id},\qquad
        \kappa:=\min\br{\frac{c_2}{\jb{6}^3},\frac{c_1}{\jb{6}}}>0,
    \end{align}
uniformly in $y$. Smoothness of the radial extension at the origin implies $\beta'(0)=0$, and $\beta''>0$ on $[0,6]$ makes $\beta'$ strictly increasing there; set $A:=\beta'(5)>0$.

\medskip \noindent\emph{Case 1: $\abs{y}\leq A$.} The critical point equation $\nabla\phi_y(k)=\beta'(\ak)\frac{k}{\ak}+y=0$ has the solution $k_c=0$ if $y=0$, and otherwise $k_c=-\rho\frac{y}{\abs{y}}$, where $\rho\in[0,5]$ is the unique root of $\beta'(\rho)=\abs{y}$, which exists by continuity and strict monotonicity of $\beta'$ together with $\beta'(0)=0$ and $\abs{y}\leq\beta'(5)$. In either case $\phi_y$ has a critical point in the ball $\set{\ak<6}$, and
    \begin{align*}
        \norm{\phi_y}_{C^3(\set{\ak<6})}
        \lesssim\norm{\beta(\abs{\cdot})}_{C^3(\set{\ak\leq6})}+A.
    \end{align*}
Proposition~\ref{prop:stationary-phase}, applied on that ball with the uniform Hessian bound~\eqref{eq:low-freq-hessian} and the amplitude $\vartheta$, gives
    \begin{align*}
        \abs{\int_{\R^3}e^{it\phi_y(k)}\vartheta(k)\,dk}\lesssim\frac{1}{t^{\frac{3}{2}}}.
    \end{align*}

\medskip \noindent\emph{Case 2: $\abs{y}\geq A$.} Passing to radial coordinates and using~\eqref{eq:Z-identity},
    \begin{align*}
        \int_{\R^3}e^{it\phi_y(k)}\vartheta(k)\,dk
        =\sum_\pm\int_0^4e^{it\psi^\pm_y(r)}Z_\pm(\ax r)\vartheta(r)r^2\,dr,
    \end{align*}
where $\psi^\pm_y(r):=\beta(r)\pm\abs{y} r$. Since $\beta'$ is strictly increasing on $[0,6]$ and $\abs{y}\geq\beta'(5)$, one has $\p_r\psi^-_y(r)=\beta'(r)-\abs{y}\leq\beta'(4)-\beta'(5)<0$ on $[0,4]$, so $\psi^-_y$ has no critical point there. The second-derivative lower bound
    \begin{align*}
        \abs{\p_r^2\psi^\pm_y(r)}=\beta''(r)\geq\frac{c_2}{\jb{4}^3}=:m>0,
        \qquad r\in[0,4],
    \end{align*}
treats both signs uniformly and allows Proposition~\ref{prop:vdc-weighted} to be applied to $g(r):=Z_\pm(\ax r)\vartheta(r)r^2$, which satisfies $g(4)=0$ because $\vartheta(4)=0$. This gives
    \begin{align*}
        \abs{\int_0^4e^{it\psi^\pm_y(r)}g(r)\,dr}
        \lesssim\frac{1}{\sqrt{mt}}\int_0^4\abs{g'(r)}\,dr.
    \end{align*}
By~\eqref{eq:Z-decay},
    \begin{align*}
        \abs{g'(r)}\lesssim\frac{\ax r^2}{(1+\ax r)^2}+\frac{r}{1+\ax r}
        \lesssim\frac{r}{1+\ax r},
    \end{align*}
so that $\int_0^4\abs{g'(r)}\,dr\lesssim\ax^{-1}$. Since $\ax\geq At$,
    \begin{align*}
        \abs{\int_{\R^3}e^{it\phi_y(k)}\vartheta(k)\,dk}
        \lesssim\frac{1}{\sqrt t}\frac{1}{\ax}
        \leq\frac{1}{\sqrt t}\frac{1}{At}
        \lesssim\frac{1}{t^{\frac{3}{2}}}. \tag*{\qedhere}
    \end{align*}
\end{proof}

The change of variables $k\mapsto \ep k$ expands the support of the amplitude to $\set{\abs{k}\leq4/\ep}$, so the $\ep$-scaled low-frequency bound is not a direct application of Lemma~\ref{lem:low-frequency-dispersive}: the rescaled amplitude is decomposed by Littlewood--Paley, and both frequency-localised estimates are applied to the pieces. This is where the two dispersive branches $\ep t^{-1}$ and $\ep^{1/2}t^{-3/2}$ are produced.

\begin{corollary}[Low-frequency dispersive estimate, $\ep$-scaled]\label{cor:low-freq-eps}
Let $\beta\in C^\infty\br{[0,\infty),\R}$ satisfy the hypotheses of Lemmas~\ref{lem:dyadic-dispersive} and~\ref{lem:low-frequency-dispersive}, and let $\vartheta\in C_c^2(\R^3)$ be radial with $\supp\vartheta\subset\set{\ak\leq4}$. Then, for all $\ep\in(0,1]$ and $t>0$,
    \begin{align*}
        \norm{\int_{\R^3}e^{\pm i\beta(\frac{\ak}{\ep})t+ix\cdot k}\vartheta(k)\,dk}_{L^\infty}
        \lesssim\min\set{1,\frac{\ep}{t},\frac{\ep^{\frac{1}{2}}}{t^{\frac{3}{2}}}},
    \end{align*}
where the implicit constant depends only on $c_1,c_2,C_1,C_2,C_3$, $\norm{\beta(\abs{\cdot})}_{C^3(\set{\ak\leq6})}$, $\norm{\vartheta}_{C^2(\R^3)}$ and the fixed Littlewood--Paley partition.
\end{corollary}

\begin{proof}
The change of variables $k\mapsto \ep k$ gives
    \begin{align}\label{eq:low-freq-eps-sub}
        \int_{\R^3}e^{\pm i\beta\br{\frac{\ak}{\ep}}t+ix\cdot k}\vartheta(k)\,dk
        =\ep^3\int_{\R^3}e^{\pm i\beta(\abs{k})t+i(\ep x)\cdot k}\vartheta(\ep k)\,dk.
    \end{align}
The amplitude $\vartheta(\ep k)$ is supported in $\set{\abs{k}\leq4/\ep}$, which exceeds the support allowed in Lemma~\ref{lem:low-frequency-dispersive}, and is therefore decomposed along the fixed Littlewood--Paley partition~\eqref{eq:LP-partition}, with $\vp_{-1}$ supported in $\set{\abs{k}\leq2}$ and $\vp_{\ell}$ in $\set{2^{\ell-1}\leq\abs{k}\leq2^{\ell+1}}$. The pieces with $2^{\ell-1}>4/\ep$ vanish on $\supp\vartheta(\ep \cdot)$; this covers every $\ell>L$, where $L$ is the smallest integer such that $2^L\geq8/\ep$, so that, by minimality,
    \begin{align*}
        \vartheta(\ep k)=\vartheta(\ep k)\vp_{-1}(k)
        +\sum_{\ell=0}^L\vartheta(\ep k)\vp(2^{-\ell}k),\qquad \frac{8}{\ep}\leq2^L<\frac{16}{\ep}.
    \end{align*}

\medskip \noindent\emph{Low piece.} The amplitude $a_{\mathrm{lo}}(k):=\vartheta(\ep k)\vp_{-1}(k)$ is supported in $\set{\abs{k}\leq2}$ and satisfies, since $\ep\leq1$, $\norm{a_{\mathrm{lo}}}_{C^2(\R^3)}\lesssim\norm{\vartheta}_{C^2}\norm{\vp_{-1}}_{C^2}$ uniformly in $\ep$. Lemma~\ref{lem:low-frequency-dispersive}, applied with amplitude $a_{\mathrm{lo}}$ and with $x$ replaced by $\ep x$, bounds the integral of this piece by $t^{-3/2}$, while bounding it directly by $\norm{a_{\mathrm{lo}}}_{L^\infty}$ times the measure of $\set{\abs{k}\leq2}$ bounds it by a constant; hence
    \begin{align*}
        \abs{\int_{\R^3}e^{\pm i\beta(\abs{k})t+i(\ep x)\cdot k}a_{\mathrm{lo}}(k)\,dk}
        \lesssim\min\set{1,\frac{1}{t^{\frac{3}{2}}}}.
    \end{align*}

\medskip \noindent\emph{Dyadic pieces.} For $0\leq \ell\leq L$, set $\vp^{(\ell)}(\xi):=\vartheta(\ep 2^\ell\xi)\vp(\xi)$, so that $\vartheta(\ep k)\vp(2^{-\ell}k)=\vp^{(\ell)}(k/2^\ell)$ and $\supp\vp^{(\ell)}\subset\set{1/2\leq\abs{\xi}\leq2}$. On this support $\abs{\xi}\leq2$ and $\ep 2^\ell\leq\ep 2^L<16$, so the Leibniz rule gives the uniform $C^2$ bound
    \begin{align*}
        \norm{\vp^{(\ell)}}_{C^2(\R^3)}
        \lesssim\sum_{n=0}^2(\ep 2^\ell)^n\norm{\vartheta}_{C^n}\norm{\vp}_{C^{2-n}}
        \lesssim\norm{\vartheta}_{C^2}\norm{\vp}_{C^2},
    \end{align*}
uniformly in $\ell$ and $\ep$. Lemma~\ref{lem:dyadic-dispersive}, applied with amplitude $\vp^{(\ell)}$, parameter $\lam=2^\ell\geq1$ and $x$ replaced by $\ep x$, yields
    \begin{align*}
        \abs{\int_{\R^3}e^{\pm i\beta(\abs{k})t+i(\ep x)\cdot k}\vp^{(\ell)}\br{\frac{k}{2^\ell}}\,dk}
        \lesssim\min\set{2^{3\ell},\frac{2^{2\ell}}{t},\frac{2^{\frac{5\ell}{2}}}{t^{\frac{3}{2}}}}.
    \end{align*}

\medskip \noindent\emph{Conclusion.} The two families are combined with the prefactor $\ep^3$ in~\eqref{eq:low-freq-eps-sub}. For the dyadic pieces, a sum of minima is at most the minimum of the sums, and $\sum_{\ell=0}^L 2^{\theta \ell}\lesssim2^{\theta L}$ for $\theta>0$ with $2^L<16/\ep$, so that
    \begin{align*}
        \ep^3\sum_{\ell=0}^L\min\set{2^{3\ell},\frac{2^{2\ell}}{t},\frac{2^{\frac{5\ell}{2}}}{t^{\frac{3}{2}}}}\lesssim\ep^3\min\set{2^{3L},\frac{2^{2L}}{t},\frac{2^{\frac{5L}{2}}}{t^{\frac{3}{2}}}}
        \lesssim\min\set{1,\frac{\ep}{t},\frac{\ep^{\frac{1}{2}}}{t^{\frac{3}{2}}}}.
    \end{align*}
The low piece contributes $\ep^3\min\set{1,t^{-3/2}}$, which is bounded by the same three-term minimum in each of the three time regimes: for $t\leq\ep$ the minimum equals $1\geq\ep^3$; for $\ep\leq t\leq\ep^{-1}$ it equals $\ep/t\geq\ep^2\geq\ep^3$; and for $t\geq\ep^{-1}$ it equals $\ep^{1/2}t^{-3/2}\geq\ep^3t^{-3/2}$. Adding the two contributions proves the corollary.
\end{proof}

\subsection{Besov summation and the uniform-in-time estimate}\label{app:besov}
The bounds from the preceding two subsections are now summed
over the Littlewood--Paley decomposition.
Write $e^{\pm i\beta(\abs{\nabla_x}/\ep)t}$ for the Fourier
multiplier with symbol $e^{\pm i\beta(\ak/\ep)t}$.
The resulting dispersive estimate for $t>0$ is then combined
with a direct short-time bound to obtain an estimate uniform
for $t\geq0$.

\begin{lemma}[Besov dispersive estimate]\label{lem:short-time-trivial-besov}
Let $\beta\in C^\infty\br{[0,\infty),\R}$ and suppose there exist positive constants $c_\beta,C_\beta$ such that, for all $r\geq0$,
    \begin{align*}
        \frac{c_\beta r}{\jb{r}}\leq\beta'(r)\leq\frac{C_\beta r}{\jb{r}},\qquad
        \frac{c_\beta }{\jb{r}^3}\leq\beta''(r)\leq\frac{C_\beta }{\jb{r}^3},\qquad
        \abs{\beta'''(r)}\leq\frac{C_\beta }{\jb{r}^2}.
    \end{align*}
Suppose in addition that the radial function $k\mapsto\beta(\ak)$ is smooth on $\R^3$. Then, for all $\ep\in(0,1]$, $t\geq 0$ and $f\in\Bes^{3}_{1,1}$, with the minimum understood to equal $1$ at $t=0$,
    \begin{align*}
        \norm{e^{\pm i\beta\br{\frac{\abs{\nabla_x}}{\ep}}t}f}_{\Bes^0_{\infty,1}}
        \lesssim\min\set{1,\frac{\ep}{t},\frac{\ep^{\frac{1}{2}}}{t^{\frac{3}{2}}}}\norm{f}_{\Bes^{3}_{1,1}},
    \end{align*}
with implicit constant depending only on $c_\beta$, $C_\beta$, $\norm{\beta(\abs{\cdot})}_{C^3(\set{\ak\leq6})}$ and the fixed Littlewood--Paley partition.
\end{lemma}

\begin{proof}
The definition of the Besov norm gives
    \begin{align*}
        \norm{e^{\pm i\beta\br{\frac{\abs{\nabla_x}}{\ep}}t}f}_{\Bes^0_{\infty,1}}
        =\norm{e^{\pm i\beta\br{\frac{\abs{\nabla_x}}{\ep}}t}P_{-1}f}_{L^\infty}+\sum_{\ell\geq0}\norm{e^{\pm i\beta\br{\frac{\abs{\nabla_x}}{\ep}}t}P_\ell f}_{L^\infty}.
    \end{align*}
To estimate each block by Young's inequality, write $e^{\pm i\beta(\abs{\nabla_x}/\ep)t}P_\ell f$ as a convolution with $P_\ell f$. Fix radial cutoffs $\widetilde\vp_{-1},\widetilde\vp\in C_c^2(\R^3)$ such that
    \begin{align*}
        \widetilde\vp_{-1}=1\ \text{ on }\set{\ak\leq2},\qquad
        \supp\widetilde\vp_{-1}\subset\set{\ak\leq4},
    \end{align*}
    \begin{align*}
        \widetilde\vp=1\ \text{ on }\set{\frac12\leq\ak\leq2},\qquad
        \supp\widetilde\vp\subset\set{\frac14\leq\ak\leq4},
    \end{align*}
and set $\widetilde\vp_{\ell}:=\widetilde\vp(2^{-\ell}\cdot)$ for $\ell\geq0$. Then $\widetilde\vp_{\ell}=1$ on $\supp\vp_{\ell}$ for every $\ell \geq-1$, while $\supp\widetilde\vp_{-1}$ and $\supp\widetilde\vp$ fall within the supports allowed in Lemmas~\ref{lem:low-frequency-dispersive} and~\ref{lem:dyadic-dispersive} respectively. Since $\widetilde\vp_{\ell}\vp_{\ell}=\vp_{\ell}$, Fourier inversion and Fubini's theorem represent
\begin{align*}
    e^{\pm i\beta\br{\frac{\abs{\nabla_x}}{\ep}}t}P_\ell f(x)=K_{\ep,\ell}(t,\cdot)*P_\ell f(x),\qquad K_{\ep,\ell}(t,x):=\frac{1}{(2\pi)^3}\int_{\R^3}e^{\pm i\beta\br{\frac{\ak}{\ep}}t+ix\cdot k}\widetilde\vp_{\ell}(k)\,dk.
\end{align*}
By Corollary~\ref{cor:low-freq-eps} applied to $\widetilde\vp_{-1}$ for $\ell=-1$, and Corollary~\ref{cor:dispersive-eps} applied to $\widetilde\vp$ with $\lam=2^\ell$ for $\ell\geq0$, the low block satisfies
    \begin{align*}
        \norm{K_{\ep,-1}(t)}_{L^\infty_x}\lesssim\min\set{1,\frac{\ep}{t},\frac{\ep^{\frac{1}{2}}}{t^{\frac{3}{2}}}},
    \end{align*}
while, for $\ell\geq0$, the wave branch contributes $\ep 2^{2\ell}t^{-1}$ and the Klein--Gordon branch $\ep^{1/2}2^{5\ell/2}t^{-3/2}$, so that, by $2^{2\ell}\leq2^{5\ell/2}$,
    \begin{align*}
        \norm{K_{\ep,\ell}(t)}_{L^\infty_x}
        \lesssim\min\set{\frac{\ep 2^{2\ell}}{t},\frac{\ep^{\frac{1}{2}}2^{\frac{5\ell}{2}}}{t^{\frac{3}{2}}}}
        \leq2^{\frac{5\ell}{2}}\min\set{\frac{\ep}{t},\frac{\ep^{\frac{1}{2}}}{t^{\frac{3}{2}}}}.
    \end{align*}
Young's inequality then yields, for all $\ell \geq-1$,
    \begin{align*}
        \norm{e^{\pm i\beta\br{\frac{\abs{\nabla_x}}{\ep}}t}P_\ell f}_{L^\infty}
        \lesssim2^{\frac{5\ell}{2}}\min\set{\frac{\ep}{t},\frac{\ep^{\frac{1}{2}}}{t^{\frac{3}{2}}}}\norm{P_\ell f}_{L^1},
    \end{align*}
the case $\ell=-1$ following from the low-block kernel bound after enlarging the constant by $2^{5/2}$. Therefore
    \begin{align*}
        \norm{e^{\pm i\beta\br{\frac{\abs{\nabla_x}}{\ep}}t}f}_{\Bes^0_{\infty,1}}
        \lesssim\min\set{\frac{\ep}{t},\frac{\ep^{\frac{1}{2}}}{t^{\frac{3}{2}}}}
        \sum_{\ell \geq-1}2^{\frac{5}{2}\ell}\norm{P_\ell f}_{L^1}
        \leq\min\set{\frac{\ep}{t},\frac{\ep^{\frac{1}{2}}}{t^{\frac{3}{2}}}}\norm{f}_{\Bes^{\frac{5}{2}}_{1,1}}.
    \end{align*}

For the complementary bound no kernel is required: since $\widehat{P_\ell f}$ is supported in $\supp\vp_{\ell}$, Fourier inversion gives
    \begin{align*}
        \norm{e^{\pm i\beta\br{\frac{\abs{\nabla_x}}{\ep}}t}P_\ell f}_{L^\infty}
        \lesssim\int_{\supp\vp_{\ell}}\abs{\widehat{P_\ell f}(k)}\,dk
        \lesssim2^{3\ell}\norm{P_\ell f}_{L^1},
    \end{align*}
uniformly in $\ep$ and $t\geq0$. Summing,
    \begin{align*}
        \norm{e^{\pm i\beta\br{\frac{\abs{\nabla_x}}{\ep}}t}f}_{\Bes^0_{\infty,1}}
        \lesssim\sum_{\ell \geq-1}2^{3\ell}\norm{P_\ell f}_{L^1}
        \leq\norm{f}_{\Bes^3_{1,1}}.
    \end{align*}
Taking the minimum of the three bounds and using $\norm{\cdot}_{\Bes^{5/2}_{1,1}}\leq\norm{\cdot}_{\Bes^3_{1,1}}$ yields the result.
\end{proof}

\section{Littlewood--Paley, multiplier and embedding estimates}\label{app:LP}
This appendix collects the Littlewood--Paley, multiplier and embedding estimates used throughout, in the notation of the Besov norm defined in Section~\ref{mainresultssubsection}, together with an elementary lemma on velocity weights. 

\subsection{Basic Littlewood--Paley estimates}\label{app:LP-basic}
First, the splitting~\eqref{eq:chi-splitting} is verified as follows. Since $\chi(\cdot/2)=1$ on $\supp\chi$, the definition $\vp_0=\chi(\cdot/2)-\chi$ gives $\chi(1-\chi)=\chi\vp_0$ and $\vp_0=\chi(\cdot/2)(1-\chi)$, while $\chi\vp_{\ell}=0$ for $\ell\geq1$. These identities give the high-block equivalence, for all $s\in\R$ and $1\leq p\leq\infty$,
\begin{align}\label{eq:high-blocks}
    \norm{(1-\chi(i\nabla_x))f}_{\Bes^s_{p,1}}
    \sim
    \sum_{\ell\geq0}2^{s\ell}\norm{P_\ell f}_{L^p}.
\end{align}
The left side is bounded by the right block by block: its low term equals $\norm{\chi(i\nabla_x)P_0f}_{L^p}$, its block $\ell\geq0$ equals $\norm{(1-\chi(i\nabla_x))P_\ell f}_{L^p}$, and both are controlled by the $L^p$ boundedness of $\chi(i\nabla_x)$ and $1-\chi(i\nabla_x)$. Conversely, $P_\ell f=P_\ell (1-\chi(i\nabla_x))f$ for $\ell\geq1$, while $P_0f=(P_{-1}+P_0)(1-\chi(i\nabla_x))f$ because $\chi(\cdot/2)=\vp_{-1}+\vp_0$; so every block of the right side is bounded by terms of the left. Adding the low term $\norm{\chi(i\nabla_x)f}_{L^p}$ to both sides of~\eqref{eq:high-blocks} proves~\eqref{eq:chi-splitting}, by the definition of the norm. 

A basic tool in this appendix is the following scale-uniform Fourier multiplier estimate. For $a\in C_c^\infty(\R^3)$, $\ell\in\mathbb{Z}$ and $1\leq p\leq\infty$, the operator $a(2^{-\ell}i\nabla_x)$ is bounded on $L^p$ uniformly in $\ell$ and $p$, with norm controlled by a fixed norm of the symbol:
\begin{align}\label{eq:scaled-compact-multiplier}
    \norm{a(2^{-\ell}i\nabla_x)f}_{L^p}\leq\norm{\mathcal F^{-1}a}_{L^1}\norm{f}_{L^p},
    \qquad
    \norm{\mathcal F^{-1}a}_{L^1}\lesssim\norm{(1-\Delta_k)^{2}a}_{L^1}.
\end{align}
Indeed, $a(2^{-\ell}i\nabla_x)f=2^{3\ell}(\mathcal F^{-1}a)(2^{\ell}\cdot)*f$, so Young's inequality and a change of variables give
\begin{align*}
    \norm{a(2^{-\ell}i\nabla_x)f}_{L^p}
    \leq2^{3\ell}\norm{(\mathcal F^{-1}a)(2^{\ell}\cdot)}_{L^1}\norm{f}_{L^p}
    =\norm{\mathcal F^{-1}a}_{L^1}\norm{f}_{L^p},
\end{align*}
and the kernel is integrable: since $\jb{x}^{4}\mathcal F^{-1}a=\mathcal F^{-1}\sbr{(1-\Delta_k)^{2}a}$,
\begin{align*}
    \norm{\mathcal F^{-1}a}_{L^1}
    \leq\norm{\jb{x}^{-4}}_{L^1}\norm{\mathcal F^{-1}\sbr{(1-\Delta_k)^{2}a}}_{L^\infty}
    \lesssim\norm{(1-\Delta_k)^{2}a}_{L^1}.
\end{align*}
In particular $\chi(i\nabla_x)$ and the $P_\ell $ are bounded on $L^p$ uniformly in $\ell$, and so is $1-\chi(i\nabla_x)$, the difference of the identity and $\chi(i\nabla_x)$; and, since $\sigma\chi=(\sigma\chi(\cdot/2))\chi$, every $\sigma$ smooth on a neighbourhood of $\set{\ak\leq2}$ acts boundedly on the low block:
\begin{align}\label{eq:low-blocks}
    \norm{\sigma(i\nabla_x)\chi(i\nabla_x)f}_{L^p}\lesssim_{\sigma}\norm{\chi(i\nabla_x)f}_{L^p}.
\end{align}

Choose a smooth cutoff $\widetilde\vp$ equal to $1$ on
$\set{1/2\leq\ak\leq2}$ and supported in
$\set{\frac14\leq\ak\leq4}$.
For $\ell\geq0$, set
$\widetilde\vp_{\ell}:=\widetilde\vp(2^{-\ell}\cdot)$,
so that $\widetilde\vp_{\ell}\vp_{\ell}=\vp_{\ell}$.

\begin{lemma}[Bernstein estimates and derivative characterisation]\label{lem:LP-derivative-estimates}
Let $1\leq p\leq\infty$. For every integer $M\geq0$ and every $\ell\geq0$,
\begin{align}\label{eq:block-derivatives}
    2^{M\ell}\norm{P_\ell f}_{L^p}\sim\sum_{\abs{\gamma}=M}\norm{P_\ell \nabla_x^{\gamma}f}_{L^p}.
\end{align}
Consequently, for every integer $N\geq0$ and every $g$,
\begin{align}\label{eq:besov-derivatives}
    \norm{g}_{\Bes^{N}_{p,1}}\sim\sum_{\aal\leq N}\norm{\nabla_x^\alpha g}_{\Bes^{0}_{p,1}}.
\end{align}
\end{lemma}
\begin{proof}
For the bound $\gtrsim$ in~\eqref{eq:block-derivatives},
let $\abs{\gamma}=M$ and set
$a_\gamma(\xi):=(i\xi)^{\gamma}\widetilde\vp(\xi)
\in C_c^\infty(\R^3)$.
The symbol of $P_\ell\nabla_x^{\gamma}$ is
$2^{M\ell}a_\gamma(2^{-\ell}k)\vp_{\ell}(k)$.
Applying~\eqref{eq:scaled-compact-multiplier} to $a_\gamma$
gives $\norm{P_\ell\nabla_x^{\gamma}f}_{L^p}
\lesssim2^{M\ell}\norm{P_\ell f}_{L^p}$.
For the bound $\lesssim$, since
$\ak^{2M}=\sum_{\abs{\gamma}=M}c_\gamma(k^{\gamma})^2$
by the multinomial theorem,
\begin{align*}
    \vp_{\ell}(k)=\sum_{\abs{\gamma}=M}c_\gamma\frac{\widetilde\vp_{\ell}(k)k^{\gamma}}{\ak^{2M}}\,\vp_{\ell}(k)k^{\gamma},
\end{align*}
where $\vp_{\ell}(k)k^{\gamma}\widehat f$ is the Fourier transform of $i^{-M}P_\ell \nabla_x^{\gamma}f$, and $\widetilde\vp_{\ell}(k)k^{\gamma}\ak^{-2M}$ is $2^{-M\ell}$ times the dilate by $2^{\ell}$ of the fixed symbol $\widetilde\vp(\xi)\xi^{\gamma}\abs{\xi}^{-2M}\in C_c^\infty(\R^3)$; \eqref{eq:scaled-compact-multiplier} applied to that symbol gives the bound $\lesssim$.

Finally, \eqref{eq:besov-derivatives} follows from the definition of the Besov norm,~\eqref{eq:low-blocks} and~\eqref{eq:block-derivatives}.
\end{proof}
\subsection{Embedding and comparison estimates}\label{app:LP-norms}
The Besov norms compare with the Fourier, $L^\infty$ and Sobolev norms as follows.
\begin{lemma}[Continuous embeddings]\label{lem:embeddings}
Equip the space
\begin{align*}
\mathcal F^{-1}L^1_k
:=
\set{f\in\mathcal S'(\R^3):\widehat f\in L^1_k}
\end{align*}
with the norm $\norm{f}_{\mathcal F^{-1}L^1_k}:=\norm{\widehat f}_{L^1_k}$. Then the following continuous embeddings hold:
\begin{enumerate}[label=(\roman*)]
\item\label{emb:fourier-L1} $\mathcal F^{-1}L^1_k
\hookrightarrow
\Bes^0_{\infty,1}
\hookrightarrow
L^\infty_x.$
\item\label{emb:sobolev-besov} For every $1\leq p\leq\infty$, every integer $m\geq0$ and every $s<m+1$, $W^{m+1,p}\hookrightarrow\Bes^{s}_{p,1}$.
\item\label{emb:besov-sobolev} For every $1\leq p\leq\infty$ and every integer $m\geq0$, $\Bes^{m}_{p,1}\hookrightarrow W^{m,p}.$ In particular, $\Bes^0_{p,1}\hookrightarrow L^p.$
\end{enumerate}
\end{lemma}
\begin{proof}
Since $\sum_{\ell \geq-1}\norm{P_\ell f}_{L^p}=\norm{f}_{\Bes^0_{p,1}}$, the series $\sum_{\ell \geq-1}P_\ell f$ converges absolutely in $L^p$, and by~\eqref{eq:LP-partition} its sum is $f$; hence $\norm{f}_{L^p}\leq\norm{f}_{\Bes^0_{p,1}}$ for every $1\leq p\leq\infty$, which is both the second embedding of~\ref{emb:fourier-L1}, at $p=\infty$, and the final assertion of~\ref{emb:besov-sobolev}.
For the first embedding of~\ref{emb:fourier-L1}, set $\Omega_\ell:=\supp\vp_{\ell}$ for $\ell \geq-1$. Then $\norm{P_\ell f}_{L^\infty_x}\lesssim\int_{\Omega_\ell}\abs{\widehat f(k)}\,dk$, so, by the finite overlap of the sets $\Omega_\ell$,
\begin{align*}
    \norm{f}_{\Bes^0_{\infty,1}}
    \lesssim\sum_{\ell \geq-1}\int_{\Omega_\ell}\abs{\widehat f(k)}\,dk
    \lesssim\norm{\widehat f}_{L^1_k}.
\end{align*}
For~\ref{emb:sobolev-besov}, \eqref{eq:block-derivatives} with $M=m+1$ and~\eqref{eq:scaled-compact-multiplier} give $\norm{P_\ell f}_{L^p}\lesssim2^{-(m+1)\ell}\norm{f}_{W^{m+1,p}}$ for every $\ell\geq0$, so the high blocks satisfy
\begin{align*}
    \sum_{\ell\geq0}2^{s\ell}\norm{P_\ell f}_{L^p}
    \lesssim
    \sum_{\ell\geq0}2^{(s-m-1)\ell}\norm{f}_{W^{m+1,p}}
    \lesssim
    \norm{f}_{W^{m+1,p}},
\end{align*}
the series converging because $s<m+1$, with constant depending on $m+1-s$, while the low block satisfies $\norm{\chi(i\nabla_x)f}_{L^p}\lesssim\norm{f}_{L^p}$ by~\eqref{eq:scaled-compact-multiplier}.
For~\ref{emb:besov-sobolev}, the embedding $\Bes^0_{p,1}\hookrightarrow L^p$ just proved and~\eqref{eq:besov-derivatives} give, for every multi-index $\alpha$ with $\aal\leq m$,
\begin{align*}
    \norm{\nabla_x^\alpha f}_{L^p}
    \leq\norm{\nabla_x^\alpha f}_{\Bes^0_{p,1}}
    \lesssim\norm{f}_{\Bes^{m}_{p,1}},
\end{align*}
and summing over $\aal\leq m$ gives the claim.
\end{proof}

Velocity weights in the norms $\norm{\jb{v}^{q}h}_{W^{m,1}_{x,v}}$ are handled by the following elementary estimates.
\begin{lemma}[Velocity weights]\label{lem:velocity-weights}
Let $m\geq0$ and $q\geq0$ be integers. For every $h$ for which the right-hand side is finite,
\begin{align*}
    \sum_{\aal+\aab\leq m}\norm{\jb{v}^{q}\nabla_x^{\alpha}\nabla_v^{\beta}h}_{L^1_{x,v}}
    \lesssim_{m,q}\norm{\jb{v}^{q}h}_{W^{m,1}_{x,v}}.
\end{align*}
Consequently, for every integer $0\leq r\leq q$ and every multi-index $\kappa$ with $\abs{\kappa}\leq q$,
\begin{align*}
    \norm{\jb{v}^{r}h}_{W^{m,1}_{x,v}}+\norm{v^{\kappa}h}_{W^{m,1}_{x,v}}
    \lesssim_{m,q}\norm{\jb{v}^{q}h}_{W^{m,1}_{x,v}}.
\end{align*}
\end{lemma}
\begin{proof}
Since $\abs{\nabla_v^{\gamma}\jb{v}^{q}}\lesssim\jb{v}^{q}$ for every multi-index $\gamma$, the first estimate follows from the Leibniz rule by induction on $\aab$. The derivatives of $\jb{v}^{r}$ and $v^{\kappa}$ are likewise bounded by constant multiples of $\jb{v}^{q}$, so the Leibniz rule and the first estimate give the second.
\end{proof}

\subsection{Multiplier estimates}\label{app:LP-multipliers}
The lemma below concerns symbols satisfying, for some order $\nu\in\R$, the homogeneous Mikhlin bound
\begin{align}\label{mikh}
    \abs{\nabla_k^\gamma m(k)}
    \lesssim\ak^{\nu-\abs{\gamma}},
    \qquad\abs{\gamma}\leq4,
\end{align}
on the region specified in each statement, and in general singular at the origin. The estimates depend linearly on the implicit constant in~\eqref{mikh}. Part~\ref{highfreq-mikhlin} is the multiplier theorem of~\cite[Proposition~2.78]{BahouriCheminDanchin2011} applied on the high-frequency block, where the singularity is invisible; it is proved directly below in the norms of Section~\ref{mainresultssubsection}. On the low block the symbol is instead traded against the smoothing operator $\abs{\nabla_x}^{-\delta}$ in part~\ref{lowfreq-smoothing}, while part~\ref{lowfreq-interpolation}, a variant of~\cite[Lemma~C.4]{HanKwanNguyenRousset2025LinearVM} used for the convolution term of Section~\ref{complexrootSection}, bounds the symbol applied to a derivative by a geometric mean. The lemma is stated for scalar symbols and applies entry by entry to the matrix-valued amplitudes of Section~\ref{complexrootSection}, with the symbol bounds measured in a fixed matrix norm.
\begin{lemma}[Mikhlin bounds]\label{lem:multiplier-bounds}
Let $m\in C^\infty(\R^3\setminus\set{0})$ and $1\leq p\leq\infty$. Each estimate below holds for every $f\in\mathcal S'(\R^3)$ for which the right-hand side is finite, the left-hand side being defined by the identity used in its proof.
\begin{enumerate}[label=(\roman*)]
\item\label{highfreq-mikhlin} If $m$ satisfies~\eqref{mikh} on $\set{\ak\geq\frac14}$, then, for every $s\in\R$,
\begin{align*}
    \norm{m(i\nabla_x)(1-\chi(i\nabla_x))f}_{\Bes^s_{p,1}}
    \lesssim_{s,\nu}\norm{(1-\chi(i\nabla_x))f}_{\Bes^{s+\nu}_{p,1}}.
\end{align*}
\item\label{lowfreq-smoothing} If $m$ satisfies~\eqref{mikh} with $\nu=0$ on $\set{0<\ak\leq2}$, then, for every $\delta\in(0,1]$,
\begin{align*}
    \norm{\chi(i\nabla_x)m(i\nabla_x)f}_{L^p}\lesssim_{\delta}\norm{\chi(i\nabla_x)\abs{\nabla_x}^{-\delta}f}_{L^p}.
\end{align*}
\item\label{lowfreq-interpolation} Under the same hypothesis on $m$,
\begin{align*}
    \norm{\chi(i\nabla_x)m(i\nabla_x)\nabla_xf}_{L^p}
    \lesssim\norm{f}_{L^p}^{\frac{1}{2}}\norm{\nabla_xf}_{L^p}^{\frac{1}{2}}.
\end{align*}
\end{enumerate}
\end{lemma}
\begin{proof}
For~\ref{highfreq-mikhlin}, set $\widetilde m:=\br{1-\chi(2\cdot)}m$, which vanishes on $\set{\ak\leq\frac14}$ and satisfies~\eqref{mikh} on $\R^3\setminus\set{0}$. Since $\br{1-\chi(2\cdot)}\br{1-\chi}=1-\chi$, the left side of~\ref{highfreq-mikhlin} equals $\norm{(1-\chi(i\nabla_x))\widetilde m(i\nabla_x)f}_{\Bes^s_{p,1}}$, so by~\eqref{eq:high-blocks} it suffices to prove the block bound
\begin{align*}
    2^{s\ell}\norm{P_\ell \widetilde m(i\nabla_x)f}_{L^p}\lesssim2^{(s+\nu)\ell}\norm{P_\ell f}_{L^p},\qquad \ell\geq0.
\end{align*}
This is the Bernstein estimate of Lemma~\ref{lem:LP-derivative-estimates} with $(ik)^{\gamma}$ replaced by $\widetilde m$: as there, $P_\ell \widetilde m(i\nabla_x)=(\widetilde\vp_{\ell}\widetilde m)(i\nabla_x)P_\ell $, and the rescaled symbols $2^{-\nu \ell}\widetilde\vp(\xi)\widetilde m(2^{\ell}\xi)$ are supported in $\set{\frac14\leq\abs{\xi}\leq4}$ with derivatives bounded uniformly in $\ell$ by~\eqref{mikh}, so \eqref{eq:scaled-compact-multiplier} applies.
Both low-frequency bounds rest on the same shell decomposition. By the telescoping definition of $\vp$, $\sum_{\ell\leq2}\vp(2^{-\ell}k)=1$ for $0<\ak\leq2$, so any symbol $\sigma$ supported in $\set{\ak\leq2}$ splits, for $k\neq0$, as $\sigma=\sum_{\ell\leq2}\sigma_\ell$ with $\sigma_\ell:=\sigma\vp(2^{-\ell}\cdot)$. If $\sigma$ satisfies~\eqref{mikh} with some $\nu\geq0$, write $\sigma_\ell(k)=2^{\ell\nu}a_\ell(2^{-\ell}k)$: then $\abs{\p_\xi^\gamma a_\ell(\xi)}\lesssim1$ uniformly in $\ell\leq2$, with $a_\ell$ supported in $\set{1/2\leq\abs{\xi}\leq2}$, so, by~\eqref{eq:scaled-compact-multiplier},
\begin{align*}
    \norm{\mathcal F^{-1}\sigma_\ell}_{L^1}=2^{\ell\nu}\norm{\mathcal F^{-1}a_\ell}_{L^1}\lesssim2^{\ell\nu}.
\end{align*}
For~\ref{lowfreq-smoothing}, since $\chi(\cdot/2)=1$ on $\supp\chi$, $\chi m=\tau\chi\ak^{-\delta}$ with $\tau:=\chi(\cdot/2)m\ak^{\delta}$, so that $\chi(i\nabla_x)m(i\nabla_x)f=\tau(i\nabla_x)\sbr{\chi(i\nabla_x)\abs{\nabla_x}^{-\delta}f}$. The symbol $\tau$ is supported in $\set{\ak\leq2}$ and satisfies~\eqref{mikh} with $\nu=\delta$ by the Leibniz rule, so
\begin{align*}
    \norm{\mathcal F^{-1}\tau}_{L^1}\leq\sum_{\ell\leq2}\norm{\mathcal F^{-1}\tau_\ell}_{L^1}\lesssim\sum_{\ell\leq2}2^{\ell\delta}<\infty,
\end{align*}
the series converging because $\delta>0$, and Young's inequality gives the claim. For~\ref{lowfreq-interpolation}, the symbols $\sigma:=\chi m$ and $ik\sigma$ satisfy~\eqref{mikh} with $\nu=0$ and $\nu=1$ respectively, so, $ik\sigma_\ell$ being the $\ell$-th shell of $ik\sigma$, the shell bounds give $\norm{\mathcal F^{-1}\sigma_\ell}_{L^1}\lesssim1$ and $\norm{\mathcal F^{-1}(ik\sigma_\ell)}_{L^1}\lesssim2^{\ell}$. As $\nabla_x$ has symbol $ik$, the function $\sigma_\ell(i\nabla_x)\nabla_xf$ is also the Fourier multiplier with symbol $ik\sigma_\ell$ applied to $f$, so Young's inequality, applied to either expression, gives
\begin{align*}
    \norm{\sigma_\ell(i\nabla_x)\nabla_xf}_{L^p}
    \lesssim\min\br{\norm{\nabla_xf}_{L^p},2^{\ell}\norm{f}_{L^p}}
    \leq2^{\frac{\ell}{2}}\norm{f}_{L^p}^{\frac{1}{2}}\norm{\nabla_xf}_{L^p}^{\frac{1}{2}},
\end{align*}
the last step by $\min(a,b)\leq\sqrt{ab}$, and summing the resulting geometric series over $\ell\leq2$ proves the claim.
\end{proof}
\begingroup
\renewcommand{\addcontentsline}[3]{}
\section*{Acknowledgements}

The author would like to thank Cl\'ement Mouhot for many valuable discussions and for his generous guidance on the mathematical context and presentation of the results.

\section*{Declaration on the use of AI}
During the preparation of this work, the author used OpenAI's ChatGPT (5.6 Sol and later 6 Astra) to assist with bibliographical searches, reference checking, and the organisation and exposition of the manuscript. Except where attributed to other sources, all mathematical results and arguments are the author's own. The author independently verified all mathematical arguments and references and takes full responsibility for the manuscript.
\endgroup
\printbibliography

@article{Landau1946,
  author  = {Landau, Lev D.},
  title   = {On the vibrations of the electronic plasma},
  journal = {Akad. Nauk SSSR. Zhurnal Eksper. Teoret. Fiz.},
  volume  = {16},
  pages   = {574--586},
  year    = {1946},
  addendum = {Russian; English translation in J. Phys. (USSR) 10 (1946), 25--34},
}

@article{MouhotVillani2011,
  author  = {Mouhot, Cl{\'e}ment and Villani, C{\'e}dric},
  title   = {On {L}andau damping},
  journal = {Acta Math.},
  volume  = {207},
  number  = {1},
  pages   = {29--201},
  year    = {2011},
  doi     = {10.1007/s11511-011-0068-9},
}

@article{HanKwanNguyenRousset2021Screened,
  author  = {Han-Kwan, Daniel and Nguyen, Toan T. and Rousset, Fr{\'e}d{\'e}ric},
  title   = {Asymptotic stability of equilibria for screened {V}lasov--{P}oisson systems via pointwise dispersive estimates},
  journal = {Ann. PDE},
  volume  = {7},
  number  = {2},
  pages   = {Paper No. 18, 37 pp.},
  year    = {2021},
  doi     = {10.1007/s40818-021-00110-5},
}

@article{HanKwanNguyenRousset2021Linearized,
  author  = {Han-Kwan, Daniel and Nguyen, Toan T. and Rousset, Fr{\'e}d{\'e}ric},
  title   = {On the linearized {V}lasov--{P}oisson system on the whole space around stable homogeneous equilibria},
  journal = {Comm. Math. Phys.},
  volume  = {387},
  number  = {3},
  pages   = {1405--1440},
  year    = {2021},
  doi     = {10.1007/s00220-021-04228-2},
}

@article{IonescuPausaderWangWidmayer2024Poisson,
  author     = {Ionescu, Alexandru D. and Pausader, Benoit and Wang, Xuecheng and Widmayer, Klaus},
  title      = {Nonlinear {L}andau damping for the {V}lasov--{P}oisson system in $\mathbb{R}^3$: the {P}oisson equilibrium},
  journal    = {Ann. PDE},
  volume     = {10},
  number     = {1},
  pages      = {Paper No. 2, 78 pp.},
  year       = {2024},
  doi        = {10.1007/s40818-023-00161-w},
  eprint     = {2205.04540},
  eprinttype = {arxiv},
}

@article{BedrossianMasmoudiMouhot2022,
  author  = {Bedrossian, Jacob and Masmoudi, Nader and Mouhot, Cl{\'e}ment},
  title   = {Linearized wave-damping structure of {V}lasov--{P}oisson in $\mathbb R^3$},
  journal = {SIAM J. Math. Anal.},
  volume  = {54},
  number  = {4},
  pages   = {4379--4406},
  year    = {2022},
  doi     = {10.1137/20M1386141},
}

@article{BedrossianMasmoudiMouhot2018,
  author  = {Bedrossian, Jacob and Masmoudi, Nader and Mouhot, Cl{\'e}ment},
  title   = {Landau damping in finite regularity for unconfined systems with screened interactions},
  journal = {Comm. Pure Appl. Math.},
  volume  = {71},
  number  = {3},
  pages   = {537--576},
  year    = {2018},
  doi     = {10.1002/cpa.21730},
}

@article{HanKwanNguyenRousset2025LinearVM,
  author  = {Han-Kwan, Daniel and Nguyen, Toan T. and Rousset, Fr{\'e}d{\'e}ric},
  title   = {Linear {L}andau damping for the {V}lasov--{M}axwell system in $\mathbb{R}^3$},
  journal = {Ann. PDE},
  volume  = {11},
  number  = {2},
  pages      = {Paper No. 26, 91 pp.},
  year       = {2025},
  doi        = {10.1007/s40818-025-00217-z},
  eprint     = {2402.11402},
  eprinttype = {arxiv},
}

@incollection{AsanoUkai1986,
  author    = {Asano, Kiyoshi and Ukai, Seiji},
  title     = {On the {V}lasov--{P}oisson limit of the {V}lasov--{M}axwell equation},
  booktitle = {Patterns and Waves: Qualitative Analysis of Nonlinear Differential Equations},
  editor    = {Nishida, Takaaki and Mimura, Masayasu and Fujii, Hiroshi},
  series    = {Stud. Math. Appl.},
  volume    = {18},
  pages     = {369--383},
  publisher = {North-Holland, Amsterdam},
  year      = {1986},
  doi       = {10.1016/S0168-2024(08)70137-1},
}

@article{Degond1986,
  author  = {Degond, Pierre},
  title   = {Local existence of solutions of the {V}lasov--{M}axwell equations and convergence to the {V}lasov--{P}oisson equations for infinite light velocity},
  journal = {Math. Methods Appl. Sci.},
  volume  = {8},
  number  = {4},
  pages   = {533--558},
  year    = {1986},
  doi     = {10.1002/mma.1670080135},
}

@article{Schaeffer1986,
  author  = {Schaeffer, Jack},
  title   = {The classical limit of the relativistic {V}lasov--{M}axwell system},
  journal = {Comm. Math. Phys.},
  volume  = {104},
  number  = {3},
  pages   = {403--421},
  year    = {1986},
  doi     = {10.1007/BF01210948},
}

@article{BrigouleixHanKwan2022,
  author     = {Brigouleix, Nicolas and Han-Kwan, Daniel},
  title      = {The non-relativistic limit of the {V}lasov--{M}axwell system with uniform macroscopic bounds},
  journal    = {Ann. Fac. Sci. Toulouse Math. (6)},
  volume     = {31},
  number     = {2},
  pages      = {545--594},
  year       = {2022},
  doi        = {10.5802/afst.1702},
  eprint     = {2004.13323},
  eprinttype = {arxiv},
}

@article{HanKwanNguyenRousset2018LongTime,
  author  = {Han-Kwan, Daniel and Nguyen, Toan T. and Rousset, Fr{\'e}d{\'e}ric},
  title   = {Long time estimates for the {V}lasov--{M}axwell system in the non-relativistic limit},
  journal = {Comm. Math. Phys.},
  volume  = {363},
  number  = {2},
  pages   = {389--434},
  year    = {2018},
  doi     = {10.1007/s00220-018-3208-7},
}

@article{HanKwanNguyen2016Instability,
  author     = {Han-Kwan, Daniel and Nguyen, Toan T.},
  title      = {Nonlinear instability of {V}lasov--{M}axwell systems in the classical and quasineutral limits},
  journal    = {SIAM J. Math. Anal.},
  volume     = {48},
  number     = {5},
  pages      = {3444--3466},
  year       = {2016},
  doi        = {10.1137/15M1028765},
  eprint     = {1506.08537},
  eprinttype = {arxiv},
}

@book{stein-harmonic-analysis,
  author    = {Stein, Elias M.},
  title     = {Harmonic Analysis: Real-Variable Methods, Orthogonality, and Oscillatory Integrals},
  series    = {Princeton Mathematical Series},
  volume    = {43},
  publisher = {Princeton University Press, Princeton, NJ},
  year      = {1993},
  note      = {With the assistance of Timothy S. Murphy},
}

@book{BahouriCheminDanchin2011,
  author    = {Bahouri, Hajer and Chemin, Jean-Yves and Danchin, Rapha{\"e}l},
  title     = {Fourier Analysis and Nonlinear Partial Differential Equations},
  series    = {Grundlehren der mathematischen Wissenschaften},
  volume    = {343},
  publisher = {Springer, Heidelberg},
  year      = {2011},
  doi       = {10.1007/978-3-642-16830-7},
}

@book{Hormander1990,
  author    = {H{\"o}rmander, Lars},
  title     = {The Analysis of Linear Partial Differential Operators {I}: Distribution Theory and {F}ourier Analysis},
  series    = {Classics in Mathematics},
  publisher = {Springer-Verlag, Berlin},
  year      = {2003},
  note      = {Reprint of the second (1990) edition},
  doi       = {10.1007/978-3-642-61497-2},
}

@article{Bigorgne2022VM,
  author  = {Bigorgne, L{\'e}o},
  title   = {Global existence and modified scattering for the solutions to the {V}lasov--{M}axwell system with a small distribution function},
  journal = {Anal. PDE},
  volume  = {18},
  number  = {3},
  pages   = {629--714},
  year       = {2025},
  doi        = {10.2140/apde.2025.18.629},
  eprint     = {2208.08360},
  eprinttype = {arxiv},
}

@online{Bigorgne2023ScatteringMap,
  author     = {Bigorgne, L{\'e}o},
  title      = {Scattering map for the {V}lasov--{M}axwell system around source-free electromagnetic fields},
  year       = {2023},
  eprint     = {2312.12214},
  eprinttype = {arxiv},
}

@article{HongPankavich2026,
  author  = {Hong, Younghun and Pankavich, Stephen},
  title   = {The nonrelativistic limit of scattering states for the {V}lasov equation with short-range interaction potentials},
  journal = {SIAM J. Math. Anal.},
  volume  = {58},
  number  = {4},
  pages   = {3720--3750},
  year       = {2026},
  doi        = {10.1137/25M1798551},
  eprint     = {2509.08072},
  eprinttype = {arxiv},
}

@article{IonescuPausaderWangWidmayer2023Stability,
author = {Ionescu, Alexandru D. and Pausader, Benoit and Wang, Xuecheng and Widmayer, Klaus},
title = {On the stability of homogeneous equilibria in the {V}lasov--{P}oisson system on $\mathbb{R}^3$},
journal = {Class. Quantum Grav.},
volume = {40},
number = {18},
pages = {Paper No. 185007, 32 pp.},
year = {2023},
doi = {10.1088/1361-6382/acebb0},
eprint = {2305.11166},
eprinttype = {arxiv},
}

@article{GlasseySchaeffer1994,
  author  = {Glassey, Robert and Schaeffer, Jack},
  title   = {Time decay for solutions to the linearized {V}lasov equation},
  journal = {Transport Theory Statist. Phys.},
  volume  = {23},
  number  = {4},
  pages   = {411--453},
  year    = {1994},
  doi     = {10.1080/00411459408203873},
}

@article{GlasseySchaeffer1995,
  author  = {Glassey, Robert and Schaeffer, Jack},
  title   = {On time decay rates in {L}andau damping},
  journal = {Comm. Partial Differential Equations},
  volume  = {20},
  number  = {3--4},
  pages   = {647--676},
  year    = {1995},
  doi     = {10.1080/03605309508821107},
}

@article{Weibel1959,
  author  = {Weibel, Erich S.},
  title   = {Spontaneously growing transverse waves in a plasma due to an anisotropic velocity distribution},
  journal = {Phys. Rev. Lett.},
  volume  = {2},
  number  = {3},
  pages   = {83--84},
  year    = {1959},
  doi     = {10.1103/PhysRevLett.2.83},
}

@article{NguyenSurvival2026,
  author     = {Nguyen, Toan T.},
  title      = {Landau damping and survival threshold},
  journal    = {J. Funct. Anal.},
  volume     = {290},
  number     = {8},
  pages      = {Paper No. 111357},
  year       = {2026},
  doi        = {10.1016/j.jfa.2026.111357},
  eprint     = {2305.08672},
  eprinttype = {arxiv},
}

@article{Bedrossian2021Echoes,
  author  = {Bedrossian, Jacob},
  title   = {Nonlinear echoes and {L}andau damping with insufficient regularity},
  journal = {Tunis. J. Math.},
  volume  = {3},
  number  = {1},
  pages   = {121--205},
  year    = {2021},
  doi     = {10.2140/tunis.2021.3.121},
  eprint  = {1605.06841},
  eprinttype = {arxiv},
}

@article{GlasseyStrauss1987,
author = {Glassey, Robert T. and Strauss, Walter A.},
title = {Absence of shocks in an initially dilute collisionless plasma},
journal = {Comm. Math. Phys.},
volume = {113},
number = {2},
pages = {191--208},
year = {1987},
doi = {10.1007/BF01223511},
}

@article{Bigorgne2020SharpAsymptotics,
author = {Bigorgne, L{\'e}o},
title = {Sharp asymptotic behavior of solutions of the 3d {V}lasov--{M}axwell system with small data},
journal = {Comm. Math. Phys.},
volume = {376},
number = {2},
pages = {893--992},
year = {2020},
doi = {10.1007/s00220-019-03604-3},
eprint = {1812.11897},
eprinttype = {arxiv},
}

@article{PankavichBenArtzi2025,
author = {Pankavich, Stephen and Ben-Artzi, Jonathan},
title = {Modified scattering of solutions to the relativistic {V}lasov--{M}axwell system inside the light cone},
journal = {J. Lond. Math. Soc.},
volume = {112},
number = {5},
pages = {Paper No. e70346},
year = {2025},
doi = {10.1112/jlms.70346},
eprint = {2306.11725},
eprinttype = {arxiv},
}

@article{GinibreVelo1985,
author = {Ginibre, Jean and Velo, Giorgio},
title = {Time decay of finite energy solutions of the non linear {K}lein--{G}ordon and {S}chr{\"o}dinger equations},
journal = {Ann. Inst. H. Poincar{\'e} Phys. Th{\'e}or.},
volume = {43},
number = {4},
pages = {399--442},
year = {1985},
}

@article{NguyenWeiZhang2025,
  author  = {Nguyen, Quoc-Hung and Wei, Dongyi and Zhang, Zhifei},
  title   = {A new proof of nonlinear {L}andau damping for the 3{D} {V}lasov--{P}oisson system near {P}oisson equilibrium},
  journal = {Acta Math. Sci.},
  volume  = {45},
  number  = {6},
  pages   = {2669--2684},
  year       = {2025},
  doi        = {10.1007/s10473-025-0616-6},
  eprint     = {2411.18408},
  eprinttype = {arxiv},
}

@article{Wang2022VM,
author = {Wang, Xuecheng},
title = {Propagation of regularity and long time behavior of the 3{D} massive relativistic transport equation {II}: {V}lasov--{M}axwell system},
journal = {Comm. Math. Phys.},
volume = {389},
number = {2},
pages = {715--812},
year = {2022},
doi = {10.1007/s00220-021-04257-x},
eprint = {1804.06566},
eprinttype = {arxiv},
}

@article{WeiYang2021,
author = {Wei, Dongyi and Yang, Shiwu},
title = {On the 3{D} relativistic {V}lasov--{M}axwell system with large {M}axwell field},
journal = {Comm. Math. Phys.},
volume = {383},
number = {3},
pages = {2275--2307},
year = {2021},
doi = {10.1007/s00220-021-04001-5},
eprint = {2005.06130},
eprinttype = {arxiv},
}

@article{Breton2026ModifiedScattering,
author = {Breton, Emile},
title = {Modified scattering for small data solutions to the {V}lasov--{M}axwell system: a short proof},
journal = {Asymptot. Anal.},
volume = {148},
number = {2},
pages = {707--725},
year = {2026},
doi = {10.1177/09217134251378921},
eprint = {2503.01677},
eprinttype = {arxiv},
}

@article{Breton2026Completeness,
author = {Breton, Emile},
title = {A note on the non {$L^1$}-asymptotic completeness of the {V}lasov--{M}axwell system},
journal = {Kinet. Relat. Models},
volume = {25},
pages = {42--56},
year = {2027},
doi = {10.3934/krm.2026023},
eprint = {2509.04025},
eprinttype = {arxiv},
}

\end{document}